\documentclass{amsart}
\usepackage{amssymb}
\usepackage{fullpage} 
\newtheorem{theorem}{Theorem}

\newtheorem{corollary}[theorem]{Corollary}
\newtheorem{sublemma}{Lemma}[theorem]
\newtheorem{lemma}[theorem]{Lemma}

\newtheorem{question}[theorem]{Question}
\newtheorem{questions}[theorem]{Questions}

\newtheorem{definition}[theorem]{Definition}

\newtheorem{example}[theorem]{Example}

\newcommand{\QED}{\end{proof}}

\def\proclaim[#1]{{\bf #1}}
\def\BF#1.{{\bf #1.}}
\newcommand{\url}[1]{{\tt #1}}

\newcommand{\Bukovsky}{Bukovsk\`y}

\newcommand{\Vopenka}{Vop\v{e}nka}
\newcommand{\Los}{\L o\'s}

\newcommand{\B}{{\mathbb B}}
\newcommand{\C}{{\mathbb C}}

\renewcommand{\P}{{\mathbb P}}

\newcommand{\Vbar}{{\overline{V}}}
\newcommand{\Vhat}{{\hat{V}}}

\renewcommand{\Ddot}{{\dot D}}

\newcommand{\id}{\mathop{\hbox{\small id}}}

\newcommand{\of}{\subseteq}
\newcommand{\ofnoteq}{\subsetneq}

\newcommand{\fo}{\supseteq}

\newcommand{\set}[1]{\{\,{#1}\,\}}
\newcommand{\singleton}[1]{\left\{{#1}\right\}}
\newcommand{\compose}{\circ}

\newcommand{\inverse}{{-1}}

\newcommand{\dom}{\mathop{\rm dom}}

\newcommand{\dirlim}{\mathop{\rm dir\,lim}}

\newcommand{\ran}{\mathop{\rm ran}}

\newcommand{\Add}{\mathop{\rm Add}}

\newcommand{\image}{\mathbin{\hbox{\tt\char'42}}}
\newcommand{\plus}{{+}}

\newcommand{\restrict}{\upharpoonright} 
\newcommand{\satisfies}{\models}
\newcommand{\forces}{\Vdash}

\newcommand{\cross}{\times}

\newcommand{\union}{\cup}

\newcommand{\Union}{\bigcup}

\newcommand{\intersect}{\cap}
\newcommand{\Intersect}{\bigcap}

\newcommand{\smalllt}{\mathrel{\mathchoice{\raise2pt\hbox{$\scriptstyle<$}}{\raise1pt\hbox{$\scriptstyle<$}}{\raise0pt\hbox{$\scriptscriptstyle<$}}{\scriptscriptstyle<}}}
\newcommand{\smallleq}{\mathrel{\mathchoice{\raise2pt\hbox{$\scriptstyle\leq$}}{\raise1pt\hbox{$\scriptstyle\leq$}}{\raise1pt\hbox{$\scriptscriptstyle\leq$}}{\scriptscriptstyle\leq}}}

\newcommand{\ltomega}{{{\smalllt}\omega}}

\newcommand{\ltkappa}{{{\smalllt}\kappa}}
\newcommand{\leqkappa}{{{\smallleq}\kappa}}

\newcommand{\ltdelta}{{{\smalllt}\delta}}

\newcommand{\boolval}[1]{\mathopen{\lbrack\!\lbrack}\,#1\,\mathclose{\rbrack\!\rbrack}}
\def\[#1]{\boolval{#1}}

\newcommand{\UnderTilde}[1]{{\setbox1=\hbox{$#1$}\baselineskip=0pt\vtop{\hbox{$#1$}\hbox to\wd1{\hfil$\sim$\hfil}}}{}}
\newcommand{\Undertilde}[1]{{\setbox1=\hbox{$#1$}\baselineskip=0pt\vtop{\hbox{$#1$}\hbox to\wd1{\hfil$\scriptstyle\sim$\hfil}}}{}}
\newcommand{\undertilde}[1]{{\setbox1=\hbox{$#1$}\baselineskip=0pt\vtop{\hbox{$#1$}\hbox to\wd1{\hfil$\scriptscriptstyle\sim$\hfil}}}{}}
\newcommand{\UnderdTilde}[1]{{\setbox1=\hbox{$#1$}\baselineskip=0pt\vtop{\hbox{$#1$}\hbox to\wd1{\hfil$\approx$\hfil}}}{}}
\newcommand{\Underdtilde}[1]{{\setbox1=\hbox{$#1$}\baselineskip=0pt\vtop{\hbox{$#1$}\hbox to\wd1{\hfil\scriptsize$\approx$\hfil}}}{}}

\newcommand{\st}{\mid}

\renewcommand{\implies}{\mathrel{\rightarrow}}

\renewcommand{\iff}{\mathrel{\leftrightarrow}}
\newcommand{\Iff}{\mathrel{\Leftrightarrow}}

\newcommand{\iso}{\cong}
\def\<#1>{\langle#1\rangle}
\newcommand{\ot}{\mathop{\rm ot}\nolimits}
\newcommand{\val}{\mathop{\rm val}\nolimits}

\newcommand{\cp}{\mathop{\rm cp}}

\newcommand{\RO}{\mathop{{\rm RO}}}
\newcommand{\ORD}{\mathop{{\rm Ord}}}

\newcommand{\ZFC}{{\rm ZFC}}

\newcommand{\CH}{{\rm CH}}

\newcommand{\cell}[1]{\boxit{\hbox to 17pt{\strut\hfil$#1$\hfil}}}
\newcommand{\head}[2]{\lower2pt\vbox{\hbox{\strut\footnotesize\it\hskip3pt#2}\boxit{\cell#1}}}
\newcommand{\boxit}[1]{\setbox4=\hbox{\kern2pt#1\kern2pt}\hbox{\vrule\vbox{\hrule\kern2pt\box4\kern2pt\hrule}\vrule}}
\newcommand{\Col}[3]{\hbox{\vbox{\baselineskip=0pt\parskip=0pt\cell#1\cell#2\cell#3}}}
\newcommand{\tapenames}{\raise 5pt\vbox to .7in{\hbox to .8in{\it\hfill input: \strut}\vfill\hbox to
.8in{\it\hfill scratch: \strut}\vfill\hbox to .8in{\it\hfill output: \strut}}}
\newcommand{\Head}[4]{\lower2pt\vbox{\hbox to25pt{\strut\footnotesize\it\hfill#4\hfill}\boxit{\Col#1#2#3}}}
\newcommand{\Dots}{\raise 5pt\vbox to .7in{\hbox{\ $\cdots$\strut}\vfill\hbox{\ $\cdots$\strut}\vfill\hbox{\
$\cdots$\strut}}}
\newcommand{\df}{\it} 
\usepackage{diagrams}
\diagramstyle[tight,centredisplay%
,dpi=600
,heads=LaTeX%
]%
\newarrow{Extension}....>
\newcommand{\sqtimes}{\mathrel{\scriptstyle\boxtimes}}
\begin{document}
\author{Joel David Hamkins}
\address[J.D.H.]{Mathematics Program, The Graduate Center of The City University of New York, 365 Fifth Avenue, New York, NY 10016
 \& Mathematics, The College of Staten Island of CUNY, Staten Island, NY 10314}
\email{jhamkins@gc.cuny.edu, http://jdh.hamkins.org}
\author{Daniel Evan Seabold}
\address[D.E.S.]{103 Hofstra University, Hempstead, NY 11566}
\email{Daniel.E.Seabold@hofstra.edu}
\thanks{The research of the first author has been supported in part by grants from the CUNY Research Foundation,
        the US National Science Foundation DMS 0800762 and the Simons Foundation. In addition, much of this work was undertaken at the Universiteit van Amsterdam, Institute for Logic, Language and Computation, where he was visiting professor in 2007 and supported by grants \textit{NWO
Bezoekersbeurs} \textsf{B 62-612} and \textsf{B 62-619 2006/00782/IB} in earlier years. His research was also supported by a visiting fellowship of the Isaac Newton Institute of Mathematical Sciences, Cambridge, UK in spring 2012, when the project was completed. This paper was the topic of the first author's tutorial lecture series at the Young Set Theorists Workshop at the Hausdorff Center for Mathematics in K\"onigswinter near Bonn, Germany, March 2011. Commentary on this article can be made at http://jdh.hamkins.org/boolean-ultrapowers.}

\begin{abstract}
Boolean ultrapowers extend the classical ultrapower construction to work with ultrafilters on any complete Boolean algebra, rather than
only on a power set algebra. When they are well-founded, the associated Boolean ultrapower embeddings exhibit a large cardinal
nature, and the Boolean ultrapower construction thereby unifies two central themes of set theory---forcing and large cardinals---by revealing
them to be two facets of a single underlying construction, the Boolean ultrapower.
\end{abstract}

\title{Well-founded Boolean ultrapowers as large cardinal embeddings}
\maketitle

\begin{quote}
 \tiny\tableofcontents
\end{quote}

Boolean ultrapowers generalize the classical ultrapower construction to work with ultrafilters on an arbitrary complete Boolean algebra, rather than
only on a power set algebra as is customary with ordinary ultrapowers. The method was introduced by \Vopenka\
\cite{Vopenka1965:OnNablaModelOfSetTheory} as a presentation of the method of forcing via Boolean-valued models,
and Bell \cite{Bell1985:BooleanValuedModelsAndIndependenceProofs} remains an excellent exposition of this
approach, based on early notes of Dana Scott and Robert Solovay. Quite apart from forcing, Mansfield
\cite{Mansfield1970:TheoryOfBooleanUltrapowers} emphasized Boolean ultrapowers as a purely model-theoretic technique, applicable generally
to any kind of structure. In this article, we shall focus on the possibility of {\it well-founded} Boolean ultrapowers, whose corresponding
Boolean ultrapower embeddings exhibit a large cardinal nature. The Boolean ultrapower construction thereby unifies the two central concerns
of set theory---forcing and large cardinals---and reveals them to be two facets of a single underlying construction, the Boolean
ultrapower.

Our main concern is with the well-founded Boolean ultrapowers and their accompanying large cardinal Boolean ultrapower embeddings, and
this material begins in \S\ref{Section.WellFoundedBooleanUltrapowers}. Although the Boolean-valued model approach to forcing has been long known, our investigation depends on some of the subtle details of the method, which are not as well known, and so we provide an efficient but thorough summary of the topic in \S\ref{Section.QuickReviewOfForcing}, followed by our introduction of the Boolean ultrapower itself in \S\ref{Section.TransformingIntoActualModels} and \S\ref{Section.TheBooleanUltrapower}, and a summary of the naturalist account of forcing in \S\ref{Section.NaturalistAccountOfForcing}. In \S\ref{Section.AlgebraicApproach}, we give the connection between forcing and the
algebraic approach to the Boolean ultrapower, which gives rise to the extender-like characterization in \S\ref{Section.BooleanUltrapowersAsDirectLimits}. In \S\ref{Section.BooleanUltrapowersViaPartialOrders}, we explain the subtle issues for the Boolean ultrapower with respect to partial orders versus Boolean algebras. In sections \S\ref{Section.SubalgebrasEtc}, \ref{Section.Products} and \ref{Section.Ideals} we provide the basic interaction of the Boolean ultrapower with respect to subalgebras, quotients, products and ideals. In \S\ref{Section.BooleanUltrapowersVsClassicalUltrapowers} we characterize the circumstances under which a Boolean ultrapower is a classical ultrapower. And in   \S\ref{Section.BooleanUltrapowersAsLargeCardinalEmbeddings}, we begin to explore the extent to which large cardinal embeddings may be realized as Boolean ultrapowers.

\section{A quick review of forcing via Boolean-valued models}\label{Section.QuickReviewOfForcing}

The concept of a Boolean-valued model is a general model-theoretic notion, having nothing especially to do with set theory or forcing. For
example, one can have Boolean-valued groups, rings, graphs, partial orders and so on, using any first-order language. Suppose that $\B$ is
a complete Boolean algebra and $\mathcal{L}$ is any first-order language. A {\df $\B$-valued model} $\mathcal{M}$ in the language
$\mathcal{L}$ consists of an underlying set $M$, whose elements are called {\df names}, and an assignment of the atomic formulas $\[s=t]$,
$\[R(s_0,\ldots,s_n)]$ and $\[y=f(s_0,\ldots,s_n)]$, with parameters $s,t,s_0,\ldots,s_n\in M$, to elements of $\B$, providing the Boolean
value that this atomic assertion holds. These assignments must obey the laws of equality, so that
$$\begin{array}{rcl}
\[s=s]  &=& 1\\
\[s=t]  &=& \[t=s]\\
\[s=t]\wedge\[t=u] &\leq& \[s=u]\\
\bigwedge_{i<n}\[s_i=t_i]\wedge\[R(\vec s)] &\leq & \[R(\vec t)].\\
\end{array}$$
If the language includes functions symbols, then we also insist on:
$$\begin{array}{rcl}
\bigwedge_{i<n}\[s_i=t_i]\wedge\[y=f(\vec s)] &\leq& \[y=f(\vec t)]\\
\bigvee_{t\in M}\[t=f(\vec s)] &=& 1\\
\[t_0=f(\vec s)]\wedge\[t_1=f(\vec s)] &\leq& \[t_0=t_1].\\
\end{array}$$
These requirements assert that with Boolean value one, the equality axiom holds for functions and that the function takes on some unique value. Once
the Boolean values of atomic assertions are provided, then one extends the Boolean value assignment to all formulas by a simple, natural recursion:
$$\begin{array}{rcl}
\[\varphi\wedge\psi] &=& \[\varphi]\wedge\[\psi]\\
\[\neg\varphi]       &=& \neg\[\varphi]\\
\[\exists x\,\varphi(x,\vec s)] &=& \bigvee_{t\in M}\[\varphi(t,\vec s)]\\
\end{array}$$
The reader may check by induction on $\varphi$ that the general equality axiom now has Boolean value one.
 $$\[\vec s=\vec t\wedge \varphi(\vec s)\implies\varphi(\vec t)]=1$$
The Boolean-valued structure $\mathcal{M}$ is {\df full}
if for every formula $\varphi(x,\vec x)$ in $\mathcal{L}$ and $\vec s$ in $M$, there is some $t\in M$ such that $\[\exists
x\,\varphi(x,\vec s)]=\[\varphi(t,\vec s)]$.

An easy example is provided by the classical ultrapower construction, which has a Boolean-valued model underlying it, namely, if $M_i$ is
an $\mathcal{L}$-structure for each $i\in I$, then for the set of names we simply take all of the functions used in the ordinary ultrapower
construction, $f\in \Pi_i M_i$ and then define the Boolean values $\[\varphi(f_0,\ldots,f_n)]=\set{i\in I\st
M_i\satisfies\varphi(f_0(i),\ldots,f_n(i))]}$, which makes this a $\B$-valued structure for the power set $\B=P(I)$, a complete Boolean
algebra.

In this article, we focus on the Boolean-valued models arising in set theory and forcing, and so our basic language will be the language of
set theory $\set{\in}$, consisting of one binary relation symbol $\in$. Later, we will augment this language with a unary predicate for the
ground model.

Let us now describe how every complete Boolean algebra gives rise canonically to a Boolean-valued model of set theory. Suppose that $\B$ is
a complete Boolean algebra and let $V$ denote the universe of all sets. The class of $\B$-names, denoted $V^\B$, is defined by recursion:
$\tau$ is a {\df $\B$-name} if $\tau$ is a set of pairs $\<\sigma,b>$, where $\sigma$ is a $\B$-name and $b\in \B$. The atomic Boolean
values are defined by double recursion:
$$\begin{array}{rcl}
\[\tau\in\sigma] &=& \bigvee_{\<\eta,b>\in\sigma}\[\tau=\eta]\wedge b\\
\[\tau=\sigma]   &=& \[\tau\of\sigma]\wedge\[\sigma\of\tau]\\
\[\tau\of\sigma] &=& \bigwedge_{\eta\in\dom(\tau)}(\[\eta\in\tau]\implies\[\eta\in\sigma])\\
\end{array}$$
We leave it as an exercise for the reader to verify that these assignments satisfy the laws of equality. One then extends the Boolean value assignment
to all assertions $\varphi(\tau_0,\ldots,\tau_n)$ in the language of set theory by the general recursion given previously.

\begin{lemma}[Mixing lemma]\label{Lemma.MixingLemma}
If\/ $A\of\B$ is an antichain and $\<\tau_a\st a\in A>$ is any sequence of names indexed by $A$, then there is a name $\tau$ such that
$a\leq\[\tau=\tau_a]$ for each $a\in A$.
\end{lemma}

\begin{proof} If $A\of\B$ is an antichain and $\tau_a$ is a name
for each $a\in A$, then let
 $$\tau=\set{\<\sigma,b\wedge a>\st \<\sigma,b>\in\tau_a\And a\in A},$$
which looks exactly like $\tau_a$, as far as Boolean values up to $a$ are concerned, and we leave the calculations verifying
$a\leq\[\tau=\tau_a]$ to the reader.\end{proof}

In the special case that $A$ is a maximal antichain and $\[\tau_a\neq\tau_b]=1$ for $a\neq b$ in $A$, then it is easy to see that
$a=\[\tau=\tau_a]$ for every $a\in A$, and we shall regard this extra fact as also part of the mixing lemma.

\begin{lemma}[Fullness principle]\label{FullnessPrinciple}
 The Boolean-valued model $V^\B$ is full.
\end{lemma}

\begin{proof} Consider the assertion $\exists
x\,\varphi(x,\vec\tau)$, where $\vec\tau$ is a finite sequence of names and $\varphi$ is a formula in the language of set theory. By
definition, the Boolean value $b=\[\exists x\,\varphi(x,\vec\tau)]$ is the join of the set $S=\set{\[\varphi(\sigma,\vec\tau)]\st \sigma\in
V^\B}$ of Boolean values that arise. Let $A$ be a maximal antichain below $b$, contained in the downward closure of $S$, so that every
element of $A$ is below an element of $S$ and $\vee A=b$. For each $a\in A$ choose $\sigma_a$ such that
$a\leq\[\varphi(\sigma_a,\vec\tau)]$. By the mixing lemma, find a name $\sigma$ such that $a\leq\[\sigma=\sigma_a]$ for all $a\in A$. It
follows by the equality axiom that $a\leq\[\varphi(\sigma,\vec\tau)]$ for each $a\in A$, and so $b=\vee A\leq\[\varphi(\sigma,\vec\tau)]$.
Since $\[\varphi(\sigma,\vec\tau)]\in S$, we must also have $\[\varphi(\sigma),\vec\tau)]\leq \vee S=b$, and so $\[\exists
x\,\varphi(x,\vec\tau)]=\[\varphi(\sigma,\vec\tau)]$, as desired.\end{proof}

\begin{theorem}[Fundamental theorem of forcing]\label{Theorem.[[ZFC]]=1}
 For any complete Boolean algebra $\B$, every axiom of \ZFC\ has Boolean value one in $V^\B$.
\end{theorem}

\begin{proof}[proof sketch] We merely sketch the proof, as the theorem is widely known. The theorem is actually a theorem
scheme, making a separate claim of each \ZFC\ axiom. The axiom of extensionality has Boolean value one because
$\[\tau=\sigma\iff\tau\of\sigma\text{ and }\sigma\of\tau]=1$ as a matter of definition. For the pairing, union, power set and infinity
axioms, it is an elementary exercise to create suitable names for the desired objects in each case.

To illustrate, we consider the power set axiom in more detail. Suppose that $\tau$ is any $\B$-name. We shall prove that
$\sigma=\set{\<\eta,b>\st \eta\of \dom(\tau)\times\B\And b=\[\eta\of\tau]}$ is a name for the power set of $\tau$. For any name $\nu$,
observe that $\[\nu\in\sigma]=\bigvee_{\<\eta,b>\in\sigma}\[\nu=\eta]\wedge
b=\bigvee_{\<\eta,b>\in\sigma}\[\nu=\eta{\And}\eta\of\tau]\leq\[\nu\of\tau]$, and consequently
$\[\forall\nu\,(\nu\in\sigma\implies\nu\of\tau)]=1$. Conversely, for any name $\nu$, consider the name $\nu'=\set{\<\eta,b>\st
\eta\in\dom(\tau)\And b=\[\eta\in\nu\intersect\tau]}$. The reader may check that $\[\nu'=\nu\intersect\tau]=1$, and consequently,
$\[\nu\of\tau\implies\nu=\nu']=1$. Since $\nu'\of \dom(\tau)\times\B$, it follows that $\nu'\in\dom(\sigma)$ and since
$\[\nu'\in\sigma]\geq\[\nu'\of\tau]$ by the definition of $\sigma$, we conclude $\[\nu\of\tau\implies\nu\in\sigma]=1$. So
$\[\sigma=P(\tau)]=1$, as desired.

For the separation axiom, if $\tau$ is any $\B$-name and $\varphi(x)$ is any formula in the forcing language, then $\set{\<\sigma,b>\st
\sigma\in\dom(\tau), b=\[\sigma\in\tau\And \varphi(\sigma)]}$ is a name for the corresponding subset of $\tau$ defined by $\varphi$. For
replacement, it suffices to verify collection instead, and it is easy to do so simply by using the name $\set{\<\tau,1>\st \tau\in
V_\alpha\intersect V^\B}$ for a sufficiently large ordinal $\alpha$, so that all possible Boolean values realized by witnesses are realized
by witnesses in $V_\alpha\intersect V^\B$, by appealing to collection in $V$. Finally, the foundation axiom can be verified, essentially by
choosing a minimal rank name for an element of a given name.\end{proof}

The map $\varphi\mapsto\[\varphi]$ is defined by induction on $\varphi$ in the metatheory, which allows us to verify the forcing relation
$p\forces\varphi$ in the ground model. Although the forcing slogan is, ``the forcing relation is definable in the ground model,'' this is
not strictly true when one wants to vary the formula. Rather, it is the restriction of the forcing relation to any fixed formula, thought
of as a relation on the forcing conditions and the names to be used as parameters in that formula, which is definable in the ground model.
By means of universal formulas, one can similarly get the forcing relation in the ground model for formulas of any fixed bounded
complexity. But the full forcing relation on all formulas cannot be definable, just as and for the same reasons that Tarski proved that the
full first-order satisfaction relation is not definable.

We now elucidate in greater detail the way in which the ground model $V$ sits inside the Boolean-valued model $V^\B$. For each set $x\in
V$, we recursively define the check names by $\check x=\set{\<\check y,1>\st y\in x}$, and we introduce the predicate $\check V$ into the
forcing language, defining $\[\tau\in\check V]=\bigvee_{x\in V}\[\tau=\check x]$ for any $\B$-name $\tau$. This definition arises naturally
by viewing $\check V$ as the class name $\set{\<\check x,1>\st x\in V}$, and the reader may verify that this predicate obeys the
corresponding equality axiom. Note that for any $x\in V$, we have $\[\check x\in\check V]=1$.

\begin{lemma}\label{Lemma.VandVcheck}
 $V\satisfies\varphi[x_0,\ldots,x_n]$ if and only if\/ $\[\varphi^{\check V}(\check x_0,\ldots,\check x_n)]=1$.
\end{lemma}

\begin{proof} First, let us explain more precisely what the lemma
means. The formula $\varphi^{\check V}$ is the relativization of $\varphi$ to the class $\check V$, obtained by bounding all quantifiers to
this class. Specifically, in the easy cases, $\varphi^{\check V}=\varphi$ for $\varphi$ atomic, $(\neg\varphi)^{\check
V}=\neg\varphi^{\check V}$ and $(\varphi\wedge\psi)^{\check V}=\varphi^{\check V}\wedge\psi^{\check V}$; for the quantifier case, let
$(\exists x\,\varphi)^{\check V}=\exists x\in\check V \,\varphi^{\check V}$, so that $\[(\exists x\,\varphi(x))^{\check
V}]=\bigvee_{\tau\in V^\B}\[\tau\in{\check V}\wedge\varphi^{\check V}(\tau)]=\bigvee_{x\in V}\[\varphi(\check x)]$, where we use the fact
that $\[\tau\in{\check V}]=\bigvee_{x\in V}\[\tau=\check x]$ to obtain the last equality.

The lemma is clear for atomic formulas, and the induction step for logical connectives is also easy. For the existential case, in the
forward direction, suppose that $V\satisfies\exists x\,\varphi(x,x_0,\ldots,x_n)$. Then $V\satisfies\varphi(x,x_0,\ldots,x_n)$ for some
specific $x$, and so $\[\varphi^{\check V}(\check x,\check x_0,\ldots,\check x_n)]=1$. Since $\[\check x\in{\check V}]=1$, this implies
that $\[(\exists x\,\varphi(x,\check x_0,\ldots,\check x_n))^{\check V}]=1$, as desired. Conversely, suppose that the Boolean value of
$\[(\exists x\,\varphi(x,\check x_0,\ldots,\check x_n))^{\check V}]$ is $1$. By the remarks of the previous paragraph, this is equivalent
to $\bigvee_{x\in V}\[\varphi(\check x,\check x_0,\ldots,\check x_n)]=1$. So there must be some $x\in V$ with $\[\varphi(\check x,\check
x_0,\ldots,\check x_n)]>0$. By induction, this Boolean value is either $0$ or $1$, and since it is not $0$, we conclude $\[\varphi(\check
x,\check x_0,\ldots,\check x_n)]=1$. Therefore, $V\satisfies\varphi[x,x_0,\ldots,x_n]$, and so $V\satisfies(\exists
x\,\varphi)[x_0,\ldots,x_n]$, as desired.\end{proof}

The ordinal rank of a set $x$ is defined recursively by $\rho(x)=\sup\set{\rho(y)+1\st y\in x}$; this is the same as the least ordinal
$\alpha$ such that $x\in V_{\alpha+1}$. It is easy to see in \ZFC\ that if $\beta$ is an ordinal, then $\rho(\beta)=\beta$.

\begin{lemma}\label{Lemma.NameRanks}
 If $\tau\in V^\B$ and $\rho(\tau)=\alpha$, then
$\[\rho(\tau)\leq\check\alpha]=1$.
\end{lemma}

\begin{proof} The lemma asserts that with Boolean value one, the
rank of a set is bounded by the rank of its name. That is, in the assumption $\rho(\tau)=\alpha$ is referring to the rank of the name
$\tau$, but in the conclusion, $\rho(\tau)\leq\alpha$ is referring inside the Boolean brackets to the rank of the set that $\tau$ names,
which could be less. Suppose that this is true for names of rank less than $\alpha$. In particular, if $\eta\in\dom(\tau)$, then $\eta$ has
lower rank than $\tau$ and so $\[\rho(\eta)<\check\alpha]=1$. Thus, $\bigwedge_{\eta\in\dom(\tau)}\[\rho(\eta)<\check\alpha]=1$. But
$\[\sigma\in\tau]=\bigvee_{\<\eta,b>\in\tau}\[\sigma=\eta]\wedge b$, and so $\[\sigma\in\tau]\leq
\[\rho(\sigma)<\check\alpha]$. It follows that
$\[\forall\sigma\in\tau\,(\rho(\sigma)<\check\alpha)]=1$
and so $\[\rho(\tau)\leq\check\alpha]=1$, as desired.\end{proof}

The converse implication can fail, because a small set can have a large name. Let us now further explore the nature of $\check V$ inside
$V^\B$.

\begin{lemma}\label{Lemma.VcheckTransitive}
 $\[\check V\text{ is a transitive class, containing all ordinals}]=1$.
\end{lemma}

\begin{proof} To see that $\check V$ is transitive in $V^\B$,
consider any two $\B$-names $\sigma$ and $\tau$. By the definition of Boolean value in the atomic case, we first observe for any set $x\in
V$ that $\[\sigma\in\check x]=\bigvee_{y\in x}\[\sigma=\check y]$. Using this, compute
$$\begin{array}{rcl}
 \[\sigma\in\tau\in\check V] &=& \[\sigma\in\tau]\wedge\[\tau\in\check V]\\
                             &=& \[\sigma\in\tau]\wedge\bigvee_{x\in V}\[\tau=\check x]\\
                             &=& \bigvee_{x\in V}\[\sigma\in\tau\wedge\tau=\check x]\\
                             &\leq& \bigvee_{x\in V}\[\sigma\in\check x]\\
                             &=& \bigvee_{x\in V}\bigvee_{y\in x}\[\sigma=\check y]\\
                             &=& \bigvee_{y\in V}\[\sigma=\check y]\\
                             &=& \[\sigma\in\check V].\\
\end{array}$$
Thus, $\[\sigma\in\tau\in\check V\implies\sigma\in\check V]=1$, and consequently, $\[\check V\text{ is transitive}]=1$, as desired. It remains to
show that $\check V$ contains all the ordinals with Boolean value one. If $\tau$ is any name, let $\alpha$ be an ordinal larger than $\rho(\tau)$, so
that $\[\rho(\tau)<\check\alpha]=1$ by lemma \ref{Lemma.NameRanks}. Thus, $\[\tau\text{ is an ordinal }\implies\tau\in\check\alpha]=1$. By
transitivity, we have $\[\tau\in\check\alpha\implies\tau\in\check V]=1$, and so $\[\tau\text{ is an ordinal }\implies\tau\in\check V]=1$, as
desired.\end{proof}

If one wanted to build a Boolean-valued sub-model of $V^\B$ corresponding to the ground model, one might guess at first that the natural
choice would be to use $\set{\check x\st x\in V}$, which could be denoted $V^2$ for the complete Boolean subalgebra
$2=\singleton{0,1}\of\B$. After all, we have the atomic Boolean values $\[\tau=\sigma]$ and $\[\tau\in\sigma]$ induced from $V^\B$, and so
this is indeed a Boolean-valued model. Indeed, this model has \ZFC\ with Boolean value one, and it is full. But this is largely because all
the Boolean values in this model are $0$ or $1$, as a consequence of lemma \ref{Lemma.VandVcheck}. Specifically, it is easy to see by
induction that $\[\varphi(\check x_0,\ldots,\check x_n)]$ as computed in $\set{\check x\st x\in V}$ is the same as $\[\varphi^{\check
V}(\check x_0,\ldots,\check x_n)]$ in $V^\B$, which corresponds simply to the truth in $V$ of $\varphi(x_0,\ldots,x_n)$ by lemma
\ref{Lemma.VandVcheck}. So the Boolean-valued model arising with $\set{\check x\st x\in V}$ is actually a 2-valued model, an isomorphic
copy of $V$.

The more subtle and interesting Boolean-valued model corresponding to the ground model, rather, has domain
$\Vhat=\set{\tau\st\[\tau\in\check V]=1}$. (Please note the difference between this notation $\hat V$, pronounced ``$V$-hat'', and the
notation $\check V$, pronounced ``$V$-check''.) Of course every $\check x$ is in $\Vhat$, but in general, if $\B$ is nontrivial, then by
the mixing lemma there will be additional names $\tau$ in $\Vhat$ that are not $\check x$ for any particular $x$. Since names $\tau$ in
$\hat V$ have $\bigvee_{x\in V}\[\tau=\check x]=1$, they constitute a kind of quantum superposition of sets, which have Boolean value one
of being {\it some} $\check x$, without necessarily being fully committed to being any particular $\check x$. We may regard $\Vhat$ as a
Boolean $\B$-valued model using the same atomic values $\[\tau=\sigma]$ and $\[\tau\in\sigma]$ as in $V^\B$. And with this interpretation,
the model is full.

\begin{lemma} Suppose that $\vec\tau=\<\tau_0,\ldots,\tau_n>$,
where $\[\tau_i\in\check V]^\B=1$. Then:
\begin{enumerate}
 \item $\[\varphi(\vec\tau)]^{\Vhat}=\[\varphi^{\check V}(\vec\tau)]^{V^\B}$.
 \item As a $\B$-valued model, $\Vhat$ is full. That is, for any formula $\varphi$ with parameters $\vec\tau$ from $\Vhat$, there is a
     name $\sigma\in\Vhat$ such that $\[\exists x\, \varphi(x,\vec\tau)]^{\Vhat}=\[\varphi(\sigma,\vec\tau)]^{\Vhat}$.
\end{enumerate}
\end{lemma}

\begin{proof} Statement (1) is clear for atomic formulas, and the
inductive step is clear for logical connectives. For the quantifier case, $\[\exists
x\,\varphi(x,\vec\tau)]^{\Vhat}=\bigvee_{\tau\in\Vhat}\[\varphi(\tau,\vec\tau)]^\Vhat=\bigvee_{x\in V}\[\varphi^{\check V}(\check
x,\vec\tau)]=\[(\exists x\,\varphi(x,\vec\tau))^{\check V}]$, where the second $=$ relies on the fact that $\tau\in\Vhat$ exactly when
$\[\tau\in\check V]=\bigvee_{x\in V}\[\tau=\check x]=1$.

Now consider (2). By the fullness of $V^\B$, there is some name $\eta$ such that $\[\exists x\in\check V\, \varphi^{\check
V}(x,\vec\tau)]=\[\eta\in\check V\wedge\varphi^{\check V}(\eta,\vec\tau)]$. Let this Boolean value be $b$. By the mixing lemma, there is a
name $\sigma$ such that $b\leq\[\eta=\sigma]$ and $1-b\leq\[\sigma=\check\emptyset]$. Thus, $\[\sigma\in\check V]=1$, and also
$b\leq\[\varphi^{\check V}(\sigma,\vec\tau)]\leq b$, so $\[\exists x\in\check V\,\varphi^{\check V}(x,\vec\tau)]=\[\varphi^{\check
V}(\sigma,\vec\tau)]$, as desired.\end{proof}

We now show that the full model $V^\B$ believes that it is a forcing extension of the class $\check V$. The {\df canonical name for the
generic filter} is the name $\dot G=\set{\<\check b,b>\st b\in\B}$.

\begin{lemma}\label{Lemma.dotGischeckVgeneric}
 The following assertions have $\B$-value one in $V^\B$:
\begin{enumerate}
 \item $\dot G$ is an ultrafilter on $\check\B$.
 \item $\dot G$ is a $\check V$-generic filter on $\check\B$.
\end{enumerate}
\end{lemma}

\begin{proof} The reader is asked to verify that $\[\check
b\in\dot G]=b$ for every $b\in\B$. It is easy to see that $\[\dot G\of\check\B]=1$, and also $\[\check 0\notin\dot G\wedge\check 1\in \dot
G]=1$. If $b\leq c$ in $\B$, then $\[\check b\in\dot G]=b\leq c=\[\check c\in\dot G]$, so $\[\check b\in\dot G\wedge\check b\leq\check
c\implies \check c\in\dot G]=1$. Also, for any $b,c,\in\B$, we have $\[\check b\in\dot G\mathrel{\wedge}\check c\in\dot G]=b\wedge c=\[
(b\wedge c)^{\check{}}\in\dot G]$. So with full Boolean value, $\dot G$ is an ultrafilter on $\check\B$, establishing (1). For (2), observe
that if $D\in V$ is a dense subset of $\B$, then $\[\check D\intersect \dot G\neq\emptyset]=\[\exists x\, (x\in \check D\intersect\dot
G)]=\bigvee_{b\in D}\[\check b\in\dot G]=\bigvee_{b\in D} b=1$. So with Boolean value one, the filter $\dot G$ meets every dense subset of
$\check\B$ in $\check V$, and consequently $\[\forall D\in\check V\, (\dot G\intersect D\neq\emptyset)]=1$, as desired.\end{proof}

If $\tau$ is a $\B$-name and $F\of\B$ is any filter, the {\df value} of $\tau$ with respect to $F$ is defined recursively by
$\val(\tau,F)=\set{\val(\sigma,F)\st \<\sigma,b>\in\tau\text{ for some }b\in F}$.

\begin{lemma}\label{Lemma.V^BthinksItIsVcheck[dotG]}
 For any $\B$-name $\tau$, we have
$\[\tau=\val(\check\tau,\dot G)]=1$. In other words, $V^\B$ believes that the set named by $\tau$ is the value of the name $\tau$ by the
canonical generic filter. As a consequence, $V^\B$ believes with Boolean value one that every set is the interpretation of a
$\check\B$-name in $\check V$ by the filter $\dot G$. In short, $V^\B$ believes that it is the forcing extension $\check V[\dot G]$.
\end{lemma}

\begin{proof} To clarify, observe that although $\tau$ is the
name of a set, $\check\tau$ is the name of a name and $\dot G$ is the name of a filter, so the notation $\val(\check\tau,\dot G)$ is
perfectly sensible inside the Boolean brackets. The claim asserts that with Boolean value one, the set named by the name of $\tau$, when
using the generic ultrafilter $\dot G$, is exactly $\tau$. This is proved by induction on the rank of $\tau$. If $\<\sigma,b>\in\tau$, then
$\[\sigma=\val(\check\sigma, \dot G)]=1$ by induction. Thus, $\[\eta\in\val(\check\tau,\dot
G)]=\bigvee_{\<\sigma,b>\in\tau}\[\eta=\sigma]\wedge\[\check b\in\dot G]=\bigvee_{\<\sigma,b>\in\tau}\[\eta=\sigma]\wedge b$. Since this is
the same as $\[\eta\in\tau]$, it follows that with Boolean value one the sets $\tau$ and $\val(\check \tau,\dot G)$ have the same members
in $V^\B$, and so the lemma is proved.\end{proof}

\section{Transforming Boolean-valued models into actual models}\label{Section.TransformingIntoActualModels}

We now transform the Boolean-valued models into actual classical first-order models by means of the quotient with respect to an ultrafilter
on $\B$. This process works with any sort of Boolean-valued model, whether it is a group, a ring, a graph or a partial order, and it is
simply the Boolean ultrapower analogue of the quotient process for classical ultrapowers. Let us give the details in the case of the
Boolean-valued model of set theory $V^\B$ we have constructed above. Suppose that $U$ is an ultrafilter on $\B$.
 \medskip\begin{quotation}
  \it There is NO need for the ultrafilter $U$ to be generic in any sense.
 \end{quotation}\medskip
The construction works equally well for ultrafilters $U$ in the ground model $V$ as for those not in $V$. Define two relations on the class
of $\B$-names:
$$ \tau =_U\sigma\quad \Iff\quad \[\tau=\sigma]\in U$$
$$ \tau \in_U \sigma\; \quad\Iff\quad\, \[\tau\in\sigma]\in U,$$
and also define the predicate for the ground model:
 $$\tau\in_U \check V_U\quad \Iff\quad \[\tau\in\check V]\in U.$$
We invite the reader to verify that $=_U$ is indeed an equivalence relation on $V^\B$, and it is congruence with respect to $\in_U$ and
$\check V_U$ as defined above, meaning that these relations are well defined on the corresponding equivalence classes.

For any $\tau\in V^\B$, let $[\tau]_U$ be the {\df restricted equivalence class} of $\tau$, namely, the set of all those names $\sigma$,
having minimal rank, such that $\tau=_U\sigma$. Using only the names of minimal rank is analogous to Scott's trick in the case of ordinary
ultrapowers, and ensures that the restricted equivalence class is a set, rather than a proper class. The quotient structure
$V^\B/U=\set{[\tau]_U\st\tau\in V^\B}$, is the class of restricted equivalence classes. Since the relation $=_U$ is a congruence with
respect to $\in_U$ and $\check V_U$, those relations are well-defined on the quotient structure, and we ambiguously use the notation
$V^\B/U$ to represent the full structure $\<V^\B/U,\in_U,\check V_U>$, as well as its domain.

\begin{lemma}[\Los\ lemma]\label{Lemma.BooleanValuedLos}
 If\/ $\B$ is a complete Boolean algebra and $U$ is
any ultrafilter on $\B$, then $V^\B/U\satisfies\varphi[\,[\tau_0]_U,\ldots,[\tau_n]_U]$ if and only if\/
$\[\varphi(\tau_0,\ldots,\tau_n)]\in U$.
\end{lemma}

\begin{proof} This is the classical \Los's lemma for this
context. For atomic formulas, it was arranged by definition in $V^\B/U$. For logical connectives, the inductive argument is
straightforward, and the quantifier case amounts to the fullness principle.\end{proof}

In particular, since the axioms of \ZFC\ hold with $\B$-value one by theorem \ref{Theorem.[[ZFC]]=1}, it follows that $V^\B/U$ is a
first-order model of \ZFC. In order to prove that a certain set-theoretic statement $\psi$ is consistent with \ZFC, therefore, it suffices
to find a Boolean algebra $\B$ such that $\[\psi]\neq 0$; for having done so, one then selects any ultrafilter $U$ on $\B$ with $\[\psi]\in
U$ and observes that $V^\B/U\satisfies\ZFC+\psi$ by lemma \ref{Lemma.BooleanValuedLos}. This method provides a way to prove relative
consistency results by forcing, without ever needing to consider countable transitive models or to construct filters that are in any way
generic or even psuedo-generic over a particular structure. (Indeed, we constructed $V^\B/U$ and proved all the fundamental facts about it
without even mentioning generic filters.) In this way, one can perform forcing over $V$, without ever leaving $V$, for the quotient
structure $V^\B/U$ is constructed inside $V$, and provides a (class) model of the desired theory. From any set model $M$ of \ZFC, one
constructs inside it the model $M^\B/U$, which satisfies $\ZFC+\psi$, and this is exactly what one seeks in a general method for
consistency and independence proofs.

We are getting closer to the Boolean ultrapower. To help us arrive there, let us investigate more thoroughly how the ground model class
$\check V$ interacts with the quotient procedure of $V^\B/U$. For any ultrafilter $U$ on $\B$, let $\check
V_U=\set{[\tau]_U\st\[\tau\in\check V]\in U}$. Although $[\check x]_U\in \check V_U$ for every $x\in V$, we emphasize that $\check V_U$ is
not necessarily the same as the class $\set{[\check x]_U\st x\in V}$, since in the case that the filter $U$ is not $V$-generic, there will
be some names $\tau$ that are mixtures of various $\check x$ via antichains not met by $U$. Indeed, this very point will be the key to
understanding the nature of the Boolean ultrapower in the next section. In light of the following lemma, $\check V_U$ is exactly identical
to the quotient of the Boolean-valued model $\hat V$ by $U$, and so we could use $\check V_U$ and $\hat V_U$ interchangeably.

\goodbreak
\begin{lemma}\label{Lemma.Vcheck/U}
 The following are equivalent:
 \begin{enumerate}
   \item $[\tau]_U\in \check V_U$; that is, $[\tau]_U=[\sigma]_U$ for some $\sigma$ with $\[\sigma\in\check V]\in U$.
   \item $\[\tau\in\check V]\in U$
   \item $[\tau]_U=[\sigma]_U$ for some $\sigma$ with $\[\sigma\in\check V]=1$.
 \end{enumerate}
\end{lemma}

\begin{proof} For $(1)\implies(2)$, suppose $[\tau]_U\in\check
V_U$. By definition, there is a name $\sigma$ such that $\[\sigma\in\check V]\in U$ and $[\sigma]_U=[\tau]_U$. By the equality axiom,
therefore, $\[\tau\in\check V]\in U$ as well. For $(2)\implies(3)$, suppose that $\[\tau\in\check V]\in U$. Let $b=\[\tau\in\check V]$, and
apply the mixing lemma to build a name $\sigma$ such that $\[\sigma=\tau]\geq b$ and $\[\sigma=\emptyset]\geq \neg b$. Thus,
$\[\sigma\in\check V]\geq b\vee\neg b=1$. But since $b\in U$, we have $\sigma=_U\tau$ and so $[\tau]_U=[\sigma]_U$, as desired. The
converse implications $(3)\implies(2)\implies(1)$ are clear.\end{proof}

\begin{lemma}\label{Lemma.TruthInVcheck/U}
 $\check V_U\satisfies\varphi[[\tau_0]_U,\ldots,[\tau_n]_U]$ if and only if\/ $\[\varphi^{\check V}(\tau_0,\ldots,\tau_n)]\in U$.
\end{lemma}

\begin{proof} This is true for atomic $\varphi$ as a matter of
definition, and the induction proceeds smoothly through logical connectives and the forward implication of the quantifier case. For the
converse implication, suppose $\[(\exists x\,\varphi(x,\tau_0,\ldots,\tau_n))^{\check V}]\in U$. By fullness, there is a name $\tau$ such
that this Boolean value is exactly $\[\tau\in\check V\wedge\varphi^{\check V}(\tau,\tau_0,\ldots,\tau_n)]\in U$. So $[\tau]_U\in\check V_U$
and $\check V_U\satisfies\varphi([\tau]_U,[\tau_0]_U,\ldots,[\tau_n]_U)$ by induction.\end{proof}

\begin{lemma}\label{Lemma.Vcheck/UtransitiveEtc}\
\begin{enumerate}
 \item $\check V_U$ is transitive in $V^\B/U$. Namely, if\/ $[\tau]_U\in_U[\sigma]_U\in\check V_U$, then $[\tau]_U{\in}\check V_U$.
 \item If\/ $G=[\dot G]_U$ in $V^\B/U$, then $G$ is $\check V_U$-generic for $[\check \B]_U$.
 \item $V^\B/U$ is the forcing extension of\/ $\check V_U$ by $G$. That is, $V^\B/U=\check V_U[G]$.
\end{enumerate}
\end{lemma}

\begin{proof} These facts follow from lemmas
\ref{Lemma.VcheckTransitive}, \ref{Lemma.dotGischeckVgeneric} and \ref{Lemma.V^BthinksItIsVcheck[dotG]} by taking the quotient and applying
lemmas \ref{Lemma.VandVcheck} and \ref{Lemma.BooleanValuedLos}. Specifically, for (1), if $[\tau]_U\in_U[\sigma]_U\in\check V_U$, then
$\[\tau\in\sigma\wedge\sigma\in\check V]\in U$ and so $\[\tau\in\check V]\in U$ by lemma \ref{Lemma.VcheckTransitive}. Thus,
$[\tau]_U\in\check V_U$, as desired. For (2), recall that $\dot G$ is the canonical name of the generic filter. We know from lemma
\ref{Lemma.VandVcheck} that $\[\check\B\text{ is a complete Boolean algebra in }\check V]=1$ and from lemma \ref{Lemma.dotGischeckVgeneric}
that $\[\dot G\text{ is a }\check V\text{-generic ultrafilter for }\check\B]=1$. It follows by lemma \ref{Lemma.BooleanValuedLos} that
$G=[\dot G]_U$ is a $(\check V_U)$-generic ultrafilter for $[\check\B]_U$ in $V^\B/U$. Finally, (3) follows directly from lemma
\ref{Lemma.V^BthinksItIsVcheck[dotG]}, since $V^\B/U$ knows that every object $[\tau]_U$ is the same as the interpretation of the name
$[\check\tau]_U$ by $G$; that is, $[\tau]_U=\val([\check\tau]_U,G)$.\end{proof}

\section{The Boolean Ultrapower}\label{Section.TheBooleanUltrapower}

We are now ready to define the Boolean ultrapower.

\begin{definition}\rm
The {\df Boolean ultrapower} of the universe $V$ by the ultrafilter $U$ on the complete Boolean algebra $\B$ is the structure $\check
V_U=\set{[\tau]_U\st
\[\tau\in\check V]\in U}$, under the relation $\in_U$, with
the accompanying Boolean ultrapower map
$$j_U:V\to \check V_U$$
defined by $j_U:x\mapsto [\check x]_U$. The Boolean
ultrapower is also accompanied by the full Boolean
extension $\check V_U\of \check V_U[G]=V^\B/U$. The Boolean
ultrapower of any other structure $M$ is simply the
restriction of the map to that structure, namely,
$j_U\restrict M:M\to j_U(M)$.
\end{definition}

\begin{theorem}
For any complete Boolean algebra $\B$ and any ultrafilter $U$ on $\B$, the Boolean ultrapower embedding $j_U:V\to \check V_U$ is an
elementary embedding.
\end{theorem}

\begin{proof} If $V\satisfies\varphi(x_0,\ldots,x_n)$, then by
lemma \ref{Lemma.VandVcheck} we know $\[\varphi(\check x_0,\ldots,\check x_n)^{\check V}]=1$. By lemma \ref{Lemma.TruthInVcheck/U}, this
implies $\check V_U\satisfies\varphi([\check x_0]_U,\ldots,[\check x_n]_U)$, which means $\check
V_U\satisfies\varphi(j_U(x_0),\ldots,j_U(x_n))$, and so $j_U$ is elementary.\end{proof}

\begin{theorem}\label{Theorem.GenericIffOnto}
Suppose that $U$ is an ultrafilter on the complete Boolean algebra $\B$ (with $U$ not necessarily in $V$). Then the following are
equivalent:
\begin{enumerate}
 \item $U$ is $V$-generic.
 \item The Boolean ultrapower $j_U$ is trivial, an isomorphism of $V$ with $\check V_U$.
\end{enumerate}
\end{theorem}

\begin{proof} Suppose that $U$ is a $V$-generic filter on $\B$,
and that $[\tau]_U\in\check V_U$. Thus, the Boolean value $b=\[\tau\in\check V]$ is in $U$. Since $b=\bigvee_{x\in V}\[\tau=\check x]$, and
these Boolean values are incompatible, the set $\set{\[\tau=\check x]\st x\in V}$ is a maximal antichain below $b$ in $V$. By the
genericity of $U$, there must be some $x_0\in V$ such that $\[\tau=\check x_0]\in U$, and so $[\tau]_U=[\check x_0]_U=j_U(x_0)$. Thus,
$j_U$ is surjective from $V$ to $\check V_U$. Since it is also injective and $\in$-preserving, $j_U$ is an isomorphism of $V$ with $\check
V_U$. Conversely, suppose that $j_U$ is an isomorphism and $A\of\B$ is a maximal antichain in $V$. By mixing, let $\tau$ be a name such
that $a=\[\tau=\check a]$ for every $a\in A$. Thus, $\[\tau\in\check A]=\bigvee_{a\in A}\[\tau=\check a]=\vee A=1$. In particular,
$[\tau]_U\in \check V_U$ and consequently, since we assumed $j_U$ is surjective, $[\tau]_U=[\check a]_U$ for some $a\in V$. Necessarily,
$a\in A$ since $\[\tau\in\check A]=1$. It follows that $a=\[\tau=\check a]\in U$, and so $U$ meets $A$. Thus, $U$ is $V$-generic, as
desired.
\end{proof}

Set theorists customarily view the principal ultrafilters, on the one hand, and the generic ultrafilters, on the other, as extreme opposite
kinds of ultrafilter, for the principal ultrafilters are completely trivial and the generic ultrafilters are (usually) highly nontrivial.
For the purposes of the Boolean ultrapower, however, theorem \ref{Theorem.GenericIffOnto} brings these two extremes together: the principal
ultrafilters and the generic ultrafilters both are exactly the ultrafilters with a trivial ultrapower. Thus, the generic ultrafilters become the
trivial case. This may be surprising at first, but it should not be too surprising, because in the power set case we already knew that the
kinds of ultrafilters were identical, for in any atomic Boolean algebra, the generic filters are precisely the principal filters. Another way to view it
is that non-principality for an ultrafilter exactly means that it does not meet the maximal antichain of atoms, and therefore it is
explicitly an assertion of non-genericity. Since theorem \ref{Theorem.GenericIffOnto} shows more generally that for any Boolean algebra it
is exactly the non-generic ultrafilters that give rise to nontrivial embeddings, the summary conclusion is that:

\medskip
\centerline{{\it non-generic} is the
right generalization of {\it non-principal}.}
\bigskip

We define that the {\df critical point} of $j_U$ is the least ordinal $\kappa$ such that $j_U$ is not an isomorphism of the predecessors of $\kappa$ to the $\in_U$ predecessors of $j_U(\kappa)$. This is the same as the least ordinal $\kappa$ such that $j_U(\kappa)$ does not have
order type $\kappa$ under $\in_U$, so either $j(\kappa)$ is a larger well-order than $\kappa$ or it is not well ordered at all.

\begin{theorem}\label{Theorem.CriticalPoint}
 The critical point of the Boolean ultrapower $j_U:V\to \check V_U$ is the cardinality of the smallest maximal antichain in
$V$ not met by $U$, if either exists.
\end{theorem}

\begin{proof} Suppose that $U$ does not meet the maximal
antichain $A\of\B$. Enumerate $A=\set{a_\alpha\st\alpha<\kappa}$, where $\kappa=|A|$. By the mixing lemma, let $\tau$ be a name such that
$a_\alpha=\[\tau=\check\alpha]$ for every $\alpha<\kappa$. Since $U\intersect A=\emptyset$, we have $[\tau]_U\neq[\check\alpha]_U$ for all
$\alpha<\kappa$. But $\[\tau\in\check\kappa]=\bigvee_{\alpha<\kappa}\[\tau=\check\alpha]=\vee A=1$, and so
$[\tau]_U\in_U[\check\kappa]_U=j_U(\kappa)$. So $j_U\restrict\kappa$ is not onto the $\in_U$ predecessors of $j_U(\kappa)$, and
consequently, $j_U$ has critical point $\kappa$ or smaller.

Conversely, suppose that $j_U$ has critical point $\kappa$. Thus, the $\in_U$ predecessors of $j_U(\kappa)$ are not exhausted by
$[\check\alpha]_U$ for $\alpha<\kappa$. So there is a name $\tau$ such that $[\tau]_U\in_U j(\kappa)$ but $[\tau]_U\neq[\check\alpha]_U$
for all $\alpha<\kappa$. This means that $\[\tau\in\check\kappa]\in U$, but $\[\tau=\check\alpha]\notin U$ for all $\alpha<\kappa$. Since
$\[\tau\in\check\kappa]=\bigvee_{\alpha<\kappa}\[\tau=\check\alpha]$, this means that $U$ does not meet the maximal antichain
$\set{\[\tau=\check\alpha]\st\alpha<\kappa}\union\singleton{\neg\[\tau\in\check\kappa]}$, which has size $\kappa$.\end{proof}

We define the {\df degree of genericity} of an ultrafilter $U$ on a complete Boolean algebra $\B$ to be the cardinality of the smallest
maximal antichain in $\B$ not met by $U$, if such an antichain exists, and $\ORD$ otherwise. Thus, the generic filters are exactly those
with the largest possible degree of genericity. One can see from theorem \ref{Theorem.CriticalPoint} that the Boolean ultrapower $\check
V_U$, for an ultrafilter $U$ on a complete Boolean algebra $\B$, is well-founded up to the degree of genericity of $U$. Theorem
\ref{Theorem.WellFoundedEquivalents} improves upon this, for the case $U\in V$, by showing that if the Boolean ultrapower $\check V_U$ has
a standard $\omega$, then it is fully well-founded.

\begin{corollary}
The degree of genericity of a non-generic ultrafilter $U$ in $V$ on a complete Boolean algebra $\B$ is either $\aleph_0$ or a measurable
cardinal.
\end{corollary}

\begin{proof} Let $\kappa$ be the degree of genericity of $U$, that is, the size of the smallest maximal
antichain not met by $U$. This is the same as the critical point of the Boolean ultrapower $j_U:V\to\check V_U$. Pick any $a<j_U(\kappa)$
with $j_U(\alpha)<a$ for all $\alpha<\kappa$, and define a measure $\mu$ on $\kappa$ by $X\in\mu\Iff a\in j_U(X)$. It is easy to see that
$\mu$ is a $\kappa$-complete ultrafilter on $\kappa$, and so $\kappa$ is either $\aleph_0$ or a measurable cardinal.

For an alternative argument, suppose that $A=\set{a_\alpha\st\alpha<\kappa}\of\B$ has minimal size $\kappa$, such that it is not met by
$U$. Define $X\in\mu\Iff X\of\kappa$ and $\bigvee_{\alpha\in X}a_\alpha\in U$; using the minimality of $\kappa$, it is easy to see that
$\mu$ is a $\kappa$-complete nonprincipal ultrafilter on $P(\kappa)$, and so $\kappa$ is a measurable cardinal, or $\omega$.\end{proof}

\begin{theorem}\label{Theorem.CriticalPointKappa}
 Suppose that $\B$ is a complete Boolean algebra in $V$ and $A\of\B$ is a maximal antichain of size $\kappa$, a regular
cardinal in $V$. In a forcing extension of\/ $V$, there is an ultrafilter $U\of\B$ not meeting $A$, but meeting all smaller maximal
antichains, of size less than $\kappa$. The corresponding Boolean ultrapower $j_U:V\to \check V_U$, therefore, has critical point $\kappa$.
\end{theorem}

\begin{proof} For the purposes of this proof, let us regard a set
as {\df small} if it is in $V$ and has size less than $\kappa$ in $V$. Fix the maximal antichain $A$ of size $\kappa$, a regular cardinal
in $V$. Let $I$ be the ideal of all elements of $\B$ that are below the join of a small subset of $A$. Using the regularity of $\kappa$, it
is easy to see that $I$ is $\kappa$-closed, meaning that $I$ contains the join of any subset of $I$ of size less than $\kappa$. Let
$G\of\B/I$ be $V$-generic for this quotient forcing, and let $U=\union G$ be the corresponding ultrafilter on $\B$. Since $A\of I$, it
follows that $U$ does not meet $A$. But if $B\of \B$ is any small antichain in $\B$, then $\set{[b]_I\st b\in B\setminus I}$ is a maximal
antichain in $\B/I$ (one uses the facts that $|B|<\kappa$ and $I$ is $\ltkappa$-complete to see that this antichain in maximal). Thus,
$U\of\B$ is an ultrafilter in $\B$ that meets all small antichains in $\B$, but misses $A$. Consequently, the Boolean ultrapower $V\to
\check V_U$ has critical point $\kappa$, as desired.\end{proof}

Theorem \ref{Theorem.CriticalPointKappaIllfounded} will show that the ultrapower in the proof of theorem \ref{Theorem.CriticalPointKappa}
is necessarily ill-founded below $j_U(\kappa)$. Let us close this section with the following observation.

\begin{theorem}
Suppose that $j:V\to \Vbar\of \Vbar[G]$ is the Boolean ultrapower by an ultrafilter $U\of\B$. Then $\Vbar[G]=\ran(j)[G]$, because every
element of\/ $\Vbar[G]$ has the form $\val(j(\tau),G)$, the interpretation of a name in $\ran(j)$ by $G$.
\end{theorem}

\begin{proof}
Lemma \ref{Lemma.V^BthinksItIsVcheck[dotG]}, the extension $\Vbar[G]$, which is the same as $V^\B/U$, believes that
$[\tau]_U=\val([\check\tau]_U,[\dot G]_U)$, which is the same as $\val(j(\tau),G)$. Thus, everything in $\Vbar[G]$ is the value of a name
in $\ran(j)$ by the filter $G$, and so $\Vbar[G]=\ran(j)[G]$.
\end{proof}

This theorem is exactly analogous to the normal form theorem for classical ultrapowers $j:V\to V^I/U$, for which every element of the
ultrapower is represented as $[f]_U$, which is equal to $j(f)([\id]_U)$, an observation that is the basis of seed theory for ultrapowers.
In particular, the classical ultrapower is the Skolem hull of $\ran(j)$ and $[\id]_U$. A refinement of this idea will resurface in theorem
\ref{Theorem.ExtenderRepresentation}, providing the extender representation for Boolean ultrapowers.

\section{The naturalist account of forcing}\label{Section.NaturalistAccountOfForcing}

We now present what we call the naturalist account of forcing, first in a syntactic form, and then in a semantic form.

\begin{theorem}[Naturalist account of forcing]\label{Theorem.NaturalistAccount}
If\/ $V$ is the universe of set theory and $\B$ is a notion of forcing, then there is in $V$ a definable class model of the theory
expressing what it means to be a forcing extension of\/ $V$. Specifically, in the forcing language with $\in$, constant symbols $\check x$ for
every element $x\in V$, a predicate symbol $\check V$ to represent $V$ as a ground model, and constant symbol $\dot G$, the theory asserts:
\begin{enumerate}
 \item The full elementary diagram of\/ $V$, relativized to the predicate $\check V$, using the constant symbols for elements of\/ $V$.
 \item The assertion that $\check V$ is a transitive proper class in the (new) universe.
 \item The assertion that $\dot G$ is a $\check V$-generic ultrafilter on $\check\B$.
 \item The assertion that the (new) universe is $\check V[\dot G]$, and \ZFC\ holds there.
\end{enumerate}
\end{theorem}


\begin{proof}
This is really a theorem scheme, since $V$ cannot have access to its own elementary diagram by Tarski's theorem on the non-definability of
truth. Rather, in the theorem we define a particular class model, and then claim as a scheme that this class satisfies all the desired
properties. The class model is simply the quotient model $V^\B/U$, as constructed in section \ref{Section.TransformingIntoActualModels}.
The point is that the results of that section show that theory mentioned in the theorem holds with Boolean value one in $V^\B$, using the
check names as the constant symbols and using $\check V$ as the predicate for the ground model, and consequently actually hold in the
quotient model $V^\B/U$.
\end{proof}

The theorem legitimizes actual set-theoretic forcing practice. The typical set-theorist is working in a set-theoretic universe, called $V$
and thought of as the entire universe, but then for whatever reason he or she wants to move to a forcing extension of this universe and so
utters the phrase
\begin{quote}
\it ``Let $G$ be $V$-generic for $\B$. Argue in $V[G]$\ldots''
\end{quote}
The effect is to invoke theorem \ref{Theorem.NaturalistAccount}, for one now works precisely in the theory of theorem
\ref{Theorem.NaturalistAccount}, dropping the decorative accents: all previous claims about $V$ in the argument become claims about $V$ as
the ground model of the new model $V[G]$, which is a forcing extension of $V$, and new claims about $V[G]$ are available as a result of $G$
being $V$-generic and every set in $V[G]$ having the form $\val(\tau,G)$ for some name $\tau$ in $V$. One may then iteratively invoke
theorem \ref{Theorem.NaturalistAccount} by proceeding to further forcing extensions, and so on. In this way, the naturalist account of
forcing shows that this common informal usage of forcing is completely legitimate and fully rigorous, and there is no need to introduce
supplemental meta-theoretic explanations involving the reflection theorem or countable transitive models and so on, to explain what is ``really'' going on.  Indeed, in the use
of the theorem---as opposed to its proof---there is no need to mention Boolean-valued models or the Boolean ultrapower or any other technical
foundation of forcing; rather, the theorem allows one simply to assume that $G$ is a $V$-generic filter and that $V[G]$ is the new larger
universe and then to proceed naturally in $V[G]$ to manipulate $G$ in some useful or interesting way.

The naturalist account of forcing has the following semantic form:

\begin{theorem} For any notion of forcing $\B$, a complete Boolean algebra, the set-theoretic universe $V$ has an elementary extension to
a structure $\<\Vbar,\bar\in>$, a definable class in $V$, for which there is in $V$ a $\Vbar$-generic filter $G$ for $\bar\B$ (the image
of\/ $\B$).
$$V\precsim\Vbar\of\Vbar[G]$$
In particular, the entire extension $\Vbar[G]$ and embedding is a definable class in $V$.
\end{theorem}

\begin{proof} This is exactly what the Boolean ultrapower
provides. For any ultrafilter $U$ on $\B$ in $V$, let $\Vbar=\check V_U$ be the corresponding Boolean ultrapower, with relation
$\bar\in=\in_U$. The corresponding Boolean ultrapower embedding $j:V\to\Vbar$ is a definable class in $V$ (using parameter $U$), and by
lemma \ref{Lemma.Vcheck/UtransitiveEtc}, the full extension $V^\B/U$, which is also a definable class in $V$, is exactly the forcing
extension $\Vbar[G]$, where $G=[\dot G]_U$ is $\Vbar$-generic for $\bar\B=j(\B)$, as desired.\end{proof}

This theorem provides a way for a model of set theory to view forcing extensions of itself as actual first-order structures. Of course, we
know it is impossible for $V$ to build an actual $V$-generic filter for nontrivial forcing, since the complement of the generic filter
would be a dense set that is missed. This method slips by that obstacle by replacing the ground model with an elementary extension, so that
the complement of the filter $G$ is not in $\Vbar$, even though it is in $V$. This account therefore relies on the dual nature of the
relationship between $V$ and $\Vbar$, by which $V$ is smaller than $\Vbar$, in the sense that the embedding $V\precsim\Vbar$ places a copy
of the former in the latter, but $V$ is larger than $\Vbar$, in the sense that $\Vbar$ is a definable class of $V$. This same dichotomy is
often exploited with large cardinal embeddings $j:V\to M$, where $V$ is smaller than $M$ in that it is isomorphic to $\ran(j)\of M$, but
larger than $M$ in that $M$ is a definable class of $V$.

\section{Well-founded Boolean ultrapowers}\label{Section.WellFoundedBooleanUltrapowers}

We have finally arrived at one of the main goals of our investigation, namely, the possibility that the Boolean ultrapower $\check V_U$ may
be well-founded. In this case, the corresponding Boolean ultrapower embedding $j_U$ would be a large cardinal embedding, whose nature we
should like to study. We shall begin with several characterizations of well-foundedness. Let us say that an ultrafilter $U$ on a complete
Boolean algebra $\B$ is well-founded if the corresponding Boolean ultrapower $\check V_U$ is well-founded.

Theorem \ref{Theorem.GenericIffOnto} showed that the Boolean ultrapower $\check V_U$ by a $V$-generic ultrafilter $U$ is well-founded,
since it is isomorphic to $V$, and so we may view well-foundedness as a weak form of genericity. It is not difficult to see that if
$G\of\B$ is a $V$-generic ultrafilter, then it is $V$-complete, that is, $\kappa$-complete over $V$ for all $\kappa$, meaning that for any
sequence $\<a_\alpha\st\alpha<\kappa>\in V$ of any length $\kappa$, with $a_\alpha\in G$ for all $\alpha<\kappa$, we have $\bigwedge_\alpha
a_\alpha\in G$. (The reason is that the set of $b\in\B$ that are either incompatible with some $a_\alpha$ or below $\bigwedge_\alpha
a_\alpha$ is dense.) In particular, this meet is not zero. It is not difficult to prove conversely that this $V$-completeness property
characterizes genericity. A slight but natural change of terminology back in the early days of forcing, therefore, would now have us all
referring to {\it $V$-complete} ultrafilters instead of {\it $V$-generic} ultrafilters in our forcing arguments. Since well-foundedness, of
course, has to do with countable completeness, we arrive at the following characterization of the well-founded ultrafilters in $V$ as
exactly the ultrafilters meeting all the countable maximal antichains in $V$, an appealing kind of semi-genericity.

\begin{theorem}\label{Theorem.WellFoundedEquivalents}
If\/ $U$ is an ultrafilter in $V$ on the complete Boolean algebra $\B$, then the following are equivalent:
\begin{enumerate}
 \item The Boolean ultrapower $\check V_U$ is well-founded.
 \item The Boolean ultrapower $\check V_U$ is an $\omega$-model; that is, it has only standard natural numbers.
 \item $U$ meets all countable maximal antichains of\/ $\B$ in $V$. That is, if\/ $A\of\B$ is a countable maximal antichain in $V$,
     then $U\intersect A\neq\emptyset$.
 \item $U$ is countably complete over $V$. That is, if $a_n\in U$ for all $n<\omega$ and $\<a_n\st n<\omega>\in V$, then $\bigwedge_n
     a_n\in U$.
 \item $U$ is weakly countably complete over $V$. That is, if $a_n\in U$ for all $n<\omega$ and $\<a_n\st n<\omega>\in V$, then
     $\bigwedge_n a_n\neq 0$.
\end{enumerate}
If the ultrafilter $U$ is not in $V$, but in some larger set-theoretic universe $\Vbar\fo V$, then nevertheless statements $2$ through $5$
remain equivalent, and each of them is implied by statement $1$.
\end{theorem}

\begin{proof} First, we shall prove that $2$ through $5$ are
equivalent, and implied by $1$, without assuming that $U\in V$.

($1\implies 2$) Immediate.

($2\implies 3$) Suppose that $\check V_U$ is an $\omega$-model, and consider any countable maximal antichain $A=\set{a_n\st n<\omega}\of\B$
in $V$. By the mixing lemma, there is a name $\tau$ such that $\[\tau=\check n]=a_n$, and consequently $\[\tau\in\check\omega]=1$. Thus,
$[\tau]_U\in[\check\omega]_U$. Since $\check V_U$ is an $\omega$-model, all $\in_U$-elements of\/ $[\check\omega]_U$ have the form $[\check
n]_U$ for some $n\in\omega$. Thus, $[\tau]_U=[\check n]_U$ for some $n<\omega$. Consequently, $a_n=\[\tau=\check n]\in U$, and so $U$ meets
$A$.

($3\implies 4$) Suppose that $U$ meets all countable maximal antichains in $V$ and that $a_n\in U$ for all $n<\omega$, with $\<a_n\st
n<\omega>\in V$, but $a_\omega=\bigwedge_n a_n\notin U$. Let $b_n=a_0\wedge\cdots\wedge a_n$, so that $b_n\in U$ for all $n$, but also
$b_0\geq b_1\geq\cdots$ is descending and $\bigwedge_n b_n=\bigwedge_n a_n=a_\omega\notin U$. It follows that $\singleton{\neg
b_0}\union\set{b_n-b_{n+1}\st n<\omega}\union\singleton{a_\omega}$ is a countable maximal antichain. But $\neg b_0\notin U$ since $b_0\in
U$, and $b_n-b_{n+1}\notin U$ since $b_{n+1}\in U$. Since $U$ meets the antichain, the only remaining possibility is $a_\omega\in U$, as
desired.

($4\implies 5$) Immediate.

($5\implies 2$) Suppose that $\check V_U$ is not an $\omega$-model. Thus, $[\check\omega]_U$ has nonstandard $\in_U$-elements, and so there
is some name $\tau$ such that $[\tau]_U\in[\check\omega]_U$, but $[\tau]_U\neq[\check n]_U$ for any $n<\omega$. By mixing if necessary, we
may assume without loss of generality that $\[\tau\in\check\omega]=1$. Let $a_n=\[\tau=\check n]$, and observe that $\neg a_n\in U$ for
each $n<\omega$, and also $\<a_n\st n<\omega>\in V$. Since $\bigvee_n a_n=\[\tau\in\check\omega]=1$, it follows that $\bigwedge_n \neg
a_n=0$, contrary to $U$ being weakly countably complete over $V$.

Lastly, we prove that ($5\implies 1$) under the assumption that $U\in V$. Suppose that $U\in V$ is countably complete, but $\check V_U$ is
not well-founded. Thus, there is a sequence of names $\tau_n$ such that $[\tau_{n+1}]_U\in_U [\tau_n]_U$. Thus, the Boolean value
$a_n=\[\tau_{n+1}\in\tau_n]$ is in $U$ for every $n$. By countable completeness, we know that $\bigwedge_n a_n\in U$ as well. Let $\sigma$
be the name assembling the sets $\tau_n$ into an $\omega$-sequence, so that $\[\sigma\text{ is an }\check\omega\text{-sequence}]=1$ and
$\[\sigma(\check n)=\tau_n]=1$ for every $n<\omega$. Observe that the Boolean value $\[\forall n<\check\omega\,\, \sigma(n+\check
1)\in\sigma(n)]=\bigwedge_n a_n \in U$. This contradicts that the axiom of foundation holds in $\check V_U$.\end{proof}

In the general case that $U\notin V$, theorem \ref{Theorem.CriticalPointKappaIllfounded} will show that statement $1$ of theorem
\ref{Theorem.WellFoundedEquivalents} is no longer necessarily equivalent to statements $2$ through $5$. Nevertheless, we now provide in
theorem \ref{Theorem.BooleanUltrapowerWellFounded} an additional critical characterization of well-foundedness, which applies whether $U\in
V$ or not and which will form the basis of many of our subsequent constructions. Before this, a clarification of the connection between
$\check V_U$ and $V^\B/U$ will be useful.

\begin{lemma}\label{Lemma.MostowskiCollapseOfV^B/U}
If\/ $\check V_U$ is well-founded, then so is $V^\B/U$. In this case, the Mostowski collapse of\/ $V^\B/U$ is the forcing extension
$\Vbar[G]$ of the transitive model $\Vbar\satisfies\ZFC$ arising as the collapse of\/ $\check V_U$, as in the following commutative
diagram,
\begin{diagram}[width=3em]
 \check V_U                         &  \of  & V^\B/U \\
 \dTo^{\pi\restrict\check V_U}      &        & \dTo_{\pi} \\
 \Vbar                                  &  \of  & \Vbar[G]
\end{diagram}
where $G=\pi([\dot G]_U)$ and $\pi$ is the Mostowski collapse of\/ $V^\B/U$, which agrees with the Mostowski collapse of\/ $\check V_U$.
\end{lemma}

\begin{proof} Since $\[\check V\text{ contains all ordinals}]=1$,
it follows that $\check V_U$ has the same ordinals as $V^\B/U$, and so if either of these models is well-founded, so is the other. Since
$\check V_U$ is a transitive subclass of $V^\B/U$ by lemma \ref{Lemma.Vcheck/UtransitiveEtc}, it follows that the restriction of the
Mostowski collapse of $V^\B/U$ to $\check V_U$ is the same as the Mostowski collapse of $\check V_U$. Since lemma
\ref{Lemma.Vcheck/UtransitiveEtc} also shows that $V^\B/U$ is the forcing extension $(\check V_U)[[\dot G]_U]$, this carries over via the
isomorphism $\pi$ to show that the Mostowski collapse of $V^\B/U$ is the same as $\Vbar[G]$.\end{proof}

\begin{theorem}\label{Theorem.BooleanUltrapowerWellFounded}
If\/ $U$ is an ultrafilter (not necessarily in $V$) on a complete Boolean algebra $\B$, then the following are equivalent.
\begin{enumerate}
 \item The Boolean ultrapower $\check V_U$ is well-founded.
 \item There is an elementary embedding $j:V\to M$ into a transitive class $M$, and there is an $M$-generic filter $G\of j(\B)$ such
     that $j\image U\of G$.
\end{enumerate}
\end{theorem}

\begin{proof} Suppose that $U$ is well-founded, with the
corresponding Boolean ultrapower map $j_U:V\to\check V_U$. By lemma \ref{Lemma.MostowskiCollapseOfV^B/U}, the Mostowski collapse of\/
$V^\B/U$ has the form $\pi:V^\B\iso M[G]$, where $\pi\restrict \check V_U\iso M$ is the Mostowski collapse of\/ $\check V_U$ and
$G=\pi([\dot G]_U)$ is $M$-generic. Let $j=\pi\compose j_U$, so that $j:V\to M$ is an elementary embedding of\/ $V$ into the transitive
class $M$. Since $\[\check b\in\dot G]=b$, it follows that if $b\in U$, then $[\check b]_U\in_U[\dot G]_U$ and hence $j(b)\in G$. Thus,
$j\image U\of G$, as desired.

Conversely, suppose that $j:V\to M$ is an elementary embedding of $V$ into a transitive class $M$, for which there is an $M$-generic filter
$G$ on $j(\B)$ such that $j\image U\of G$. We want to show that the Boolean ultrapower $\check V_U$ is well-founded. We will do this by
mapping it into $M$. The typical element of $\check V_U$ has the form $[\tau]_U$, where $\[\tau\in\check V]=1$. Define $k:\check V_U\to M$
by $k:([\tau]_U)\mapsto \val(j(\tau),G)$. This makes sense, because $j(\tau)$ is a $j(\B)$-name in $M$, and so it has a value by the filter
$G$ in $M[G]$. The map is well defined, because if $[\tau]_U=[\sigma]_U$, then $\[\tau=\sigma]\in U$ and so $\[j(\tau)=j(\sigma)]^{j(\B)}$,
leading to $\val(j(\tau),G)=\val(j(\sigma),G)$. Although at first glance it appears that $k$ is mapping $\check V_U$ into $M[G]$, actually
the situation is that since $\[\tau\in\check V]=1$, we have $\[j(\tau)\in\check M]^{j(\B)}=1$, and so $\val(j(\tau),G)\in M$. Furthermore,
the map $k$ is elementary because if $\check V_U\satisfies\varphi([\tau]_U)$, then $\[\varphi^{\check V}(\tau)]\in U$ by lemma
\ref{Lemma.Vcheck/U}, and so $\[\varphi^{\check M}(j(\tau))]^{j(\B)}=1$ by the elementarity of $j$, leading to
$M[G]\satisfies\varphi^M(\val(j(\tau),G))$ and consequently $M\satisfies\varphi(\val(j(\tau),G))$. And notice also that
$k(j_U(x))=k([\check x]_U)=\val(j(\check x),G)=j(x)$, and so $k\circ j_U=j$. So $j_U$ is a factor of $j$. In particular, $\check V_U$ maps
elementarily into $M$ and is therefore well-founded.\end{proof}

The proof establishes the following.

\begin{corollary}\label{Corollary.BooleanUltrapowerFactor}
If $j:V\to M$ is an elementary embedding into a transitive class $M$ and there is an $M$-generic filter $G\of j(\B)$ for which $j\image
U\of G$, then the Boolean ultrapower embedding $j_U$ is a factor of $j$ as follows:
\begin{diagram}[height=2.5em,width=3.5em]
 V & & \rTo^{j} & & M  & \of & M[G] \\
 &\rdTo_{j_{U}}&  & \ruTo^{k\restrict \check V_U}    & & \ruTo_k \\
 & & \check V_U & \of \hskip2.7em & \check V_U[[\dot G]_U]=V^\B/U\\
\end{diagram}
\end{corollary}

\begin{proof}
As in theorem \ref{Theorem.BooleanUltrapowerWellFounded}, define $k:V^\B/U\to M[G]$ by $k:[\tau]_U\mapsto \val(j(\tau),G)$. This is well
defined and elementary on the forcing extension because $V^\B/U\satisfies\varphi([\tau]_U)\iff V\satisfies\boolval{\varphi(\tau)}\in U\iff
M\satisfies\boolval{\varphi(j(\tau))}^{j(\B)}\in G\iff M[G]\satisfies\varphi(\val(j(\tau),G))$. The restriction is elementary from $\check
V_U$ to $M$ as in theorem \ref{Theorem.BooleanUltrapowerWellFounded}.
\end{proof}

It turns out that one needs $G$ only to be $\ran(j)$-generic, rather than fully $M$-generic, for much of the previous, and this is established in theorem \ref{Theorem.RudinKeislerIff}.

\begin{theorem} Every infinite complete Boolean algebra $\B$ admits non-well-founded ultrafilters.
\end{theorem}

\begin{proof} Every infinite Boolean algebra has a countably
infinite maximal antichain $A$. For any infinite maximal antichain $A$, the set $\set{\neg a\st a\in A}$ has the finite intersection
property, and so it may be extended to an ultrafilter $U$ on $\B$. Since $U$ misses $A$, a countable maximal antichain, it follows by
theorem \ref{Theorem.WellFoundedEquivalents} that $\check V_U$ is not well-founded.\end{proof}

When the Boolean ultrapower is well-founded, we shall simplify notation by combining the Boolean ultrapower embedding with the Mostowski collapse, in effect identifying the Boolean ultrapower $\check V_U$ and the full extension $V^\B/U$ with the transitive collapses of these structures, just as classical well-founded ultrapowers are commonly replaced with their collapses. Thus, we shall refer to the well-founded Boolean ultrapower $j: V\to M\of M[G] \of V$, where $M$ is the collapse of $\check V_U$ and $M[G]$ is the collapse of the full Boolean extension $V^\B/U$, and where $G=[\dot G]_U$.

Next, we generalize to well-founded Boolean ultrapowers one of the standard and useful facts about the well-founded classical ultrapowers,
namely, that they always exhibit closure up to their critical points. It is well known, for example, that if $j:V\to M$ is the ultrapower
by a measure $\mu$ on a measurable cardinal cardinal $\kappa$, then ${}^\kappa M\subset M$. More generally, it is also true that if $j :
V\to  M$ is the ultrapower by any countably complete ultrafilter $\mu$ on any set $Z$ and $\kappa=\cp(j)$, then ${}^\kappa M\subset M$. An
even greater degree of closure, of course, is the defining characteristic of supercompactness embeddings. The general phenomenon is that if
$j : V\to  M$ is a classical ultrapower by a measure $\mu$ on a set $Z$, then ${}^\theta M\subset M$ if and only if $j\image \theta \in M$.
This equivalence does not hold generally for non-ultrapower embeddings, however, such as various extender embeddings, including the $\omega$-iteration of a normal measure on a measurable cardinal. For Boolean ultrapowers, which includes this $\omega$-iteration case, what we prove is not that ${}^\theta M\of M$, but rather ${}^\theta M[G]\of M[G]$, where $M[G]=V^\B/U$ is the full Boolean extension. This does generalize the
power set ultrapower case, on the technicality that if $\B$ is a power set and hence trivial as a forcing notion, then $M=M[G]$. After the
theorem, corollary \ref{Corollary.ClosureOfBooleanUltrapowers} shows how to obtain ${}^\theta M\of M$, provided that $\B$ is sufficiently
distributive. (And if $\B$ is a power set, then it is trivially $\delta$-distributive for every $\delta$.)

\begin{theorem}\label{Theorem.ClosureOfBooleanUltrapowers}
Suppose that $j:V\to M\of M[G]\of V$ is the well-founded Boolean ultrapower by an ultrafilter $U\of\B$ in $V$. Then for any ordinal $\theta$ we have ${}^\theta M[G]\of M[G]$ if and only if $j\image\theta\in M[G]$.
In particular, if $\kappa=\cp(j)$, then definitely ${}^\kappa M[G]\of M[G]$.
\end{theorem}

\begin{proof} As explained above, we mean that $M$ is the transitive collapse of $\check V_U$ and $M[G]$ is the transitive collapse of $V^\B/U$. The theorem is referring to $\theta$-sequences in these classes in V. The forward implication is immediate. Conversely, suppose that $j\image\theta
\in M[G]$. Fix any $\theta$-sequence $\<z_\alpha \st \alpha<\theta >\in {}^\theta M[G]$. Thus, $z_\alpha = [\tau_\alpha]_U$ for some name
$\tau_\alpha$. It is an easy exercise in name manipulation to build a $\B$-name $\sigma$ amalgamating these names into a sequence, so that
$\[\sigma\text{ is a }\check\theta\text{-sequence}] = 1$ and for each $\alpha <\theta$ we have $\[\sigma(\check\alpha ) = \tau_\alpha] =
1$. If $s = [\sigma]_U$, then $s$ is a sequence of length $[\check\theta ]_U = j(\theta)$ in $M[G]$, and $s(j(\alpha )) = z_\alpha$. Since
$j\image \theta \in M[G]$, the sequence $s \restrict j\image\theta$ is isomorphic in $M$ to $\<z_\alpha\st \alpha < \theta >$, by simply
collapsing the domain $j\image\theta$ to $\theta$. Thus, $\<z_\alpha\st \alpha  < \theta >\in M[G]$, as desired. For the final claim, if
$\kappa=\cp(j)$, then of course $j\image\kappa=\kappa$, and so we conclude ${}^\kappa M\subset M$, as desired.\end{proof}

\begin{corollary}\label{Corollary.ClosureOfBooleanUltrapowers}
If\/ $\B$ is $\ltdelta$-distributive and $j:V\to M\of M[G]\of V$ is the well-founded Boolean ultrapower by $U\of\B$ in $V$, then for any
$\theta< j(\delta)$, we have ${}^\theta M\of M$ if and only if $j\image\theta\in M[G]$. In particular, if $\kappa=\cp(j)$ and $\B$ is
$\ltkappa$-distributive, then ${}^\kappa M\of M$.
\end{corollary}

\begin{proof} The forward direction again is immediate.
Conversely, suppose that $j\image\theta\in M[G]$. By theorem \ref{Theorem.ClosureOfBooleanUltrapowers}, it follows that ${}^\theta M[G]\of
M[G]$, and consequently ${}^\theta M\of M[G]$. But since $\theta\leq j(\delta)$, we know that the forcing $G\of j(\B)$ adds no new
$\theta$-sequences over $M$, and so ${}^\theta M\of M$, as desired. The final claim is a special case, since $\kappa<j(\kappa)$ and
$j\image\kappa=\kappa\in M$.\end{proof}

Perhaps one should view theorem \ref{Theorem.ClosureOfBooleanUltrapowers} as similar to the proof of theorem
\ref{Theorem.WellFoundedEquivalents}, where we argued that if the Boolean ultrapower $\check V_U$ is an $\omega$-model, then it is
well-founded. The reason is that if it is an $\omega$-model, then $j_U\image\omega$ is represented in the Boolean ultrapower, and so the
Boolean ultrapower is closed under $\omega$-sequences. It follows that it is well-founded. Later, in theorem
\ref{Theorem.ClosureWithDisjointification}, we will prove a better closure theorem for the generic Boolean ultrapower arising via the
quotient by an ideal with the disjointifying property.

%
%
Before continuing our investigation of the well-founded Boolean ultrapowers, we shall first develop in the next several sections a
sufficient general theory of Boolean ultrapowers.

\section{A purely algebraic construction of the Boolean ultrapower}\label{Section.AlgebraicApproach}

The Boolean ultrapower construction admits of a purely algebraic or model-theoretic presentation, without any reference to forcing or
names, in a manner naturally generalizing the usual power set ultrapower construction. Such a presentation is used in
\cite{Mansfield1970:TheoryOfBooleanUltrapowers}, \cite{Canjar1987:CompleteBooleanUltraproducts} and
\cite{OuwehandRose1998:FiltralPowersOfStructures}, as opposed to the forcing-based approach of \cite{Vopenka1965:OnNablaModelOfSetTheory}
and \cite{Bell1985:BooleanValuedModelsAndIndependenceProofs}. In this section, we prove the two approaches are equivalent.

Let us give the model-theoretic account. If $\B$ is a complete Boolean algebra, $U$ is an ultrafilter on $\B$ and $\mathcal{M}$ is a
structure in a first-order language, then the {\df functional presentation} of the Boolean ultrapower of $\mathcal{M}$ by $U$ on $\B$ will
consist of certain equivalence classes of functions $f:A\to M$, where $A$ is a maximal antichain in $\B$. We refer to such functions as
{\df spanning functions}, and denote by $M^{\downarrow\B}$ the collection of all spanning functions, using any maximal antichain. If $A$
and $B$ are maximal antichains in $\B$, then we say $B$ {\df refines} $A$ if for every element $b\in B$ there is some $a\in A$ with $b\leq
a$ (note that this $a$ is unique). In this case, if $f:A\to M$ is a function, then the {\df reduction} of $f$ to $B$ is the function
$f\downarrow B:B\to M$ such that $(f\downarrow B)(b)=f(a)$ when $b\leq a$. Note that any two maximal antichains have a common refinement.
We compare two spanning functions $f:A\to M$ and $g:B\to M$ by finding a common refinement $C$ of $A$ and $B$ and comparing $f\downarrow C$
and $g\downarrow C$. Specifically, we say $f\equiv_U g$ if and only if $\bigvee\set{c\in C\st (f\downarrow C)(c)=(g\downarrow C)(c)}\in U$.
This is an equivalence relation on $M^{\downarrow \B}$, and it does not depend on the choice of $C$. The {\df functional presentation} of
the Boolean ultrapower is the quotient $\mathcal{M}^{\downarrow\B}_U$, consisting of all equivalence classes $[f]_U=\set{g\in
M^{\downarrow\B}\st f\equiv_U g}$. For any relation symbol $R$ of $\mathcal{M}$, we define $R([f_0]_U,\ldots,[f_n]_U)$ in
$\mathcal{M}^{\downarrow\B}_U$ if $\bigvee\set{c\st \mathcal{M}\satisfies R((f_0\downarrow C)(c),\ldots,(f_n\downarrow C)(c))}\in U$, where
$C$ is any common refinement of $\dom(f_0),\ldots,\dom(f_n)$. Similarly, for any function symbol $r$, we define
$r([f_0]_U,\ldots,[f_n]_U)=[f]_U$ in $\mathcal{M}^{\downarrow\B}_U$, where again $C$ is a suitable common refinement and
$f(c)=r((f_0\downarrow C)(c),\ldots,(f_n\downarrow C)(c))$. One can easily establish the \Los\ theorem for this presentation, so that
$$\mathcal{M}^{\downarrow\B}_U\satisfies\varphi\bigl[[f_0]_U,\ldots,[f_n]_U\bigr]$$
if and only if there is a common refinement $C$ of $\dom(f_0),\ldots,\dom(f_n)$ such that
$$\vee\set{c\in C\st \mathcal{M}\satisfies\varphi[(f_0\downarrow C)(c),\ldots,(f_n\downarrow C)(c)]}\in U.$$ It follows that the map $j:x\mapsto [c_x]_U$, where
$c_x:1\mapsto x$ is the constant function with domain $\singleton{1}$, is an elementary embedding $j:\mathcal{M}\to
\mathcal{M}^{\downarrow\B}_U$.

\begin{theorem}\label{Theorem.TwoPresentationsIsomorphic}
The two presentations of the Boolean ultrapower of the universe are isomorphic. Specifically, there is an isomorphism
$\pi:V^{\downarrow\B}_U\iso \check V_U$ making the following diagram commute.
\begin{diagram}[width=2em]
 & & V & & \\
 & \ldTo^j & & \rdTo^{j_U} \\
V^{\downarrow\B}_U & &\rTo^\pi~{\iso} & &\check V_U \\
\end{diagram}
\end{theorem}

\begin{proof} The point is that the use of spanning functions
$f:A\to V$ corresponds exactly to an application of the mixing lemma to the values in $\ran(f)$ using the antichain $A$. Specifically, for
any spanning function $f:A\to V$, we may apply the mixing lemma to produce a $\B$-name $\tau_f$ such that that $a\leq\[\tau_f=f(a)\check\
]$ for every $a\in A$. If $B$ refines $A$, then $\[\tau_f=\tau_{f\downarrow B}]=1$ because $\tau_{f\downarrow B}$ mixes the same value
$f(a)$ on the portion of $B$ below $a$ that $\tau_f$ has with value $a$. It follows that $f \equiv_U g$ if and only if
$\[\tau_f=\tau_g]\in U$, which is to say, if and only if $[\tau_f]_U=[\tau_g]_U$. A similar argument shows that $f\in_U g$ in
$V^{\downarrow\B}_U$ if and only if $\[\tau_f\in\tau_g]\in U$. So let us define $\pi:V^{\downarrow\B}_U\to \check V_U$ by $\pi:[f]_U\mapsto
[\tau_f]_U$. To see that this is onto, we use that every element of $\check V_U$ has the form $[\tau]_U$, where $\[\tau\in\check V]=1$. For
such a $\tau$, consider the values $x$ for which $\[\tau=\check x]\neq 0$. The corresponding set $A$ of these Boolean values is a maximal
antichain. If $f:A\to V$ is the function mapping $\[\tau=\check x]$ to $x$, it follows that $\[\tau=\tau_f]=\vee A=1$, and so
$[\tau]_U=\pi([f]_U)$. So $\pi$ is an isomorphism of $V^{\downarrow\B}_U$ with $\check V_U$. The diagram commutes because $\tau_{c_x}$ is a
name such that $\[\tau_{c_x}=\check x]=1$, and so $\pi\compose j(x)=\pi([c_x]_U)=[\tau_{c_x}]_U=[\check x]_U=j_U(x)$.\end{proof}

Thus, the two accounts of the Boolean ultrapower are equivalent. The equivalence is true generally, for any first-order structure
$\mathcal{M}$, but the forcing and name presentation seems to be the most illuminating in the case of a Boolean ultrapower of the
set-theoretic universe $V$. As we see it, the advantage of the forcing and name viewpoint, for those who know forcing, is that it places
the Boolean ultrapower $\check V_U$ into the broader context of the full extension $V^\B/U$, where the existence of the canonical generic object can be illuminating. This broader context will be helpful in our later investigation of Boolean ultrapowers as large cardinal embeddings.
Next, we prove that Boolean ultrapowers include all instances of the classical ultrapower by an ultrafilter on a set.

\begin{theorem}\label{Theorem.BooleanUltrapowersGeneralizeUsualUltrapowers}
Boolean ultrapowers generalize the usual ultrapowers. Specifically, if\/ $U$ is an ultrafilter on a set $Z$, meaning that it is an
ultrafilter on the power set Boolean algebra $\B=P(Z)$, then the usual ultrapower by $U$ is the same as (isomorphic to) the Boolean
ultrapower by $U$.
\end{theorem}

\begin{proof} In the case of the power set Boolean algebra $\B=P(Z)$, every spanning function $f:A\to V$ has a maximal refinement down to the
maximal antichain consisting of the atoms $\mathcal{A}=\set{\singleton{z}\st z\in Z}$. Thus, for the purposes of the Boolean ultrapower
$V^{\downarrow\B}_U$, it suffices to consider functions $f:\mathcal{A}\to V$. Any such function naturally corresponds to a function $f:Z\to
V$, by stripping off one layer of set braces, and this correspondence is an isomorphism of $V^{\downarrow\B}_U$ with $V^Z/U$. Furthermore,
this correspondence respects constant functions, and so the ultrapower embeddings themselves are also isomorphic.\end{proof}

We now investigate a method to limit the size of the antichains needed to represent spanning functions in the Boolean ultrapower. Of
course, the chain condition of $\B$ provides a crude upper bound, for if $\B$ is $\delta$-c.c., then all maximal antichains have size less
than $\delta$. The concept of {\df descents}, however, refines this idea. Specifically, for any ultrafilter $U$ on a complete Boolean
algebra $\B$, a {\df descent} through $U$ is a continuous descending sequence from $1$ through $U$ with meet $0$, that is, a sequence
$\<b_\alpha\st \alpha<\kappa>$ beginning with $b_0=1$, having $b_\alpha\in U$ for all $\alpha$, descending $\alpha<\beta\implies
b_\alpha\geq b_\beta$, continuous $b_\lambda=\bigwedge_{\alpha<\lambda} b_\alpha$ at limit ordinals $\lambda$, and with meet zero
$\bigwedge_{\alpha<\kappa} b_\alpha=0$. The descent is {\df strict} if $b_{\alpha}>b_{\alpha+1}$ for all $\alpha<\kappa$. The {\df descent
spectrum} of $U$ is the collection of all $\kappa$ for which there is a descent through $U$ of order type $\kappa$. For any such descent,
the corresponding {\df difference antichain} has elements $d_\alpha=b_\alpha-b_{\alpha+1}$, and these form a maximal antichain
$\set{d_\alpha\st\alpha<\kappa}$ (allowing $0$ in the non-strict case). Every element $b_\alpha$ in the descent is simply the join of all
subsequent differences $b_\alpha=\bigvee\set{d_\beta\st \beta\geq\alpha}$. If $U$ is $\kappa$-complete but not $\kappa^\plus$-complete,
then of course there is a descent of order type $\kappa$, and this is the critical point of the corresponding Boolean ultrapower. More
generally, we have:

\begin{theorem}\label{Theorem.DescentIffDiscontinuityPoint}
An ultrafilter $U$ on a complete Boolean algebra admits a descent of order type $\kappa$ if and only if the Boolean ultrapower $j_U$ is
discontinuous at $\kappa$.
\end{theorem}

\begin{proof} Suppose that $\<b_\alpha\st\alpha<\kappa>$ is a
descent through $U$, with corresponding difference antichain consisting of $d_\alpha=b_\alpha-b_{\alpha+1}$. By the mixing lemma, there is
a name $\tau$ with $\[\tau=\check\alpha]=d_\alpha$. It follows that
$\[\tau<\check\beta]=\bigvee_{\alpha<\beta}\[\tau=\check\alpha]=\bigvee_{\alpha<\beta}d_\alpha$. If $\beta<\kappa$, then this is disjoint
from $b_\beta$, which is in $U$, and so $j_U(\beta)\leq[\tau]_U$. But $[\tau]_U<j_U(\kappa)$ since $\[\tau<\check\kappa]=1$, and thus
$\sup(j_U\image\kappa)\leq[\tau]_U<j_U(\kappa)$, and so $j_U$ is discontinuous at $\kappa$.

Conversely, suppose that $j_U$ is discontinuous at $\kappa$, so there is some $[\tau]_U$ with $j_U(\alpha)<[\tau]_U<j_U(\kappa)$ for all
$\alpha<\kappa$. We may choose such a name $\tau$ such that also $\[\tau<\check\kappa]=1$. Let $b_\alpha=\[\tau\geq\check\alpha]$ and
observe that $b_0\geq b_1\geq\cdots$ is descending through $U$, and $\bigwedge_{\alpha<\kappa}b_\alpha=0$ because $\[\tau<\check\kappa]=1$.
Thus, $U$ admits a descent of order type $\kappa$.\end{proof}

Principal ultrafilters, of course, have no descents. More generally, $V$-generic ultrafilters have no descents in $V$, by the discussion
before theorem \ref{Theorem.WellFoundedEquivalents}. In general, however, the descent spectrum of an ultrafilter can be interesting. To
give one example, if $\mu$ is an ultrafilter on $\omega$ and $\nu$ is a measure on a measurable cardinal $\kappa$, then the usual product
measure $\mu\times\nu$ is an ultrafilter on $P(\omega\cross\kappa)$, whose spectrum consists exactly of the ordinals of cofinality $\omega$
or $\kappa$.

\begin{lemma}\label{Lemma.RepresentationOnDifferenceAntichain}
Every element of the Boolean ultrapower $V^{\downarrow\B}_U$ is either in the range of $j_U$ or has the form $[f]_U$ for some one-to-one
spanning function $f:A\to V$ on a maximal antichain arising from the differences in a strict descent, whose order type is a cardinal.
\end{lemma}

\begin{proof} Suppose that $f:A\to V$ is any spanning function,
where $A$ has minimal cardinality among all spanning functions $U$-equivalent to $f$. If $A$ is finite, then $U$ must concentrate on a
member of $A$, so by minimality $A=\singleton{1}$ and $f$ is constant, which puts $[f]_U$ into the range of $j_U$. Otherwise, $A$ is
infinite, and we enumerate it $A=\set{d_\alpha\st\alpha<\kappa}$, where $\kappa=|A|$. By amalgamating together elements of $A$ giving the
same value via $f$, we easily find an equivalent one-to-one function, on a smaller antichain; so we may assume without loss of generality
that $f$ is one-to-one. By the minimality of $\kappa$, it follows that $\bigvee_{\alpha<\beta} d_\alpha$ is not in $U$ for any
$\beta<\kappa$. Consequently, the negation $b_\beta=\bigvee_{\alpha\geq\beta} d_\alpha$ is in $U$. Furthermore, it is easy to see that
$b_0=1$, that $\alpha<\beta\implies b_\alpha>b_\beta$, that $b_\lambda=\bigwedge_{\alpha<\lambda}b_\alpha$ and that
$\bigwedge_{\alpha<\kappa}b_\alpha=0$. That is, $\<b_\alpha\st\alpha<\kappa>$ is a strict descent through $U$ of order type $\kappa$, and
$A$ is the difference antichain, as desired.\end{proof}

A modification to the functional presentation avoids the need to consider refinements of maximal antichains or indeed, maximal antichains
at all, and this will sometimes be convenient. Specifically, let us define that an {\df open dense spanning function} is a function $f:D\to
V$ such that $D\of\B$ is an open dense subset of $\B$, such that $b\leq c\in D$ implies $f(b)=f(c)$. If $f:A\to V$ is a spanning function
on a maximal antichain $A\of\B$, then the corresponding open dense spanning function $\tilde f:D\to V$ is obtained by taking $D$ to be all
elements $b$ below an element $a$ in $A$, and defining $\tilde f(b)=f(a)$. This is the same as taking the union $\tilde
f=\Union\set{f\downarrow B\st B\text{ refines }A}$ of all reductions of $f$ to finer antichains. For two open dense spanning functions
$f:D\to V$ and $g:D'\to V$, one defines $f=_U g$ if $\bigvee\set{b\in D\intersect D'\st f(b)=g(b)}\in U$. By associating elements $[f]_U\in
V^{\downarrow\B}_U$, for a spanning function $f:A\to V$, with the corresponding equivalence class $[\tilde f]_U$ for the open dense
spanning functions, one arrives at an isomorphic representation of the Boolean ultrapower.

\section{Direct limits and an extender-like presentation}\label{Section.BooleanUltrapowersAsDirectLimits}

Generalizing the previous section, we shall now prove that every Boolean ultrapower is the direct limit of a certain induced directed
system of classical power set ultrapowers. Suppose that $\B$ is a complete Boolean algebra. If $U$ is an ultrafilter on $\B$ and $A\of\B$
is a maximal antichain, we define $U_A$ to be the corresponding ultrafilter on the power set of $A$, induced by $X\in U_A$ if and only if
$X\of A$ and $\vee X\in U$. Let $j_{U_A}:V\to V^A/U_A$ be the corresponding ultrapower by $U_A$.

\begin{lemma}\label{Lemma.UltrapowerFactor}
Suppose that $\B$ is a complete Boolean algebra with ultrafilter $U$ and $A$ is a maximal antichain in $\B$. Then the Boolean ultrapower
$j_U$ factors through $j_{U_A}$. There is a natural map $\pi_A:V^A/U_A\to \check V_U$ making the following diagram commute.
\begin{diagram}[width=6em]
 V              &             & \\
& \rdTo(1,2)^{j_{U_A}} \rdTo(3,1)^{j_U} &  & {\check V_U} \\
 & V^A/U_A   &     \ruTo(2,1)^{\pi_A} &  \\
\end{diagram}
\end{lemma}

\begin{proof} Every element of $V^A/U_A$ has the form $[f]_{U_A}$
for some function $f:A\to V$. Define $\pi_A:V^A/U\to V^{\downarrow\B}_U$ by $\pi_A:[f]_{U_A}\mapsto [f]_U$. This is clearly well defined,
since $f\equiv_{U_A} g$ implies $f\equiv_U g$ in $V^{\downarrow\B}_U$. The diagram commutes, because $\pi_A$ maps the equivalence class of
a constant function to the equivalence class of a constant function. The map is elementary, because $V^A/U_A\satisfies\varphi([f]_{U_A})$
if and only if $\set{a\in A\st V\satisfies\varphi[f(a)]}\in U_A$ by \Los's theorem, and this holds if and only if $\bigvee\set{a\in A\st
V\satisfies\varphi[f(a)]}\in U$ by the definition of $U_A$, and this holds if and only if $V^{\downarrow\B}_U\satisfies\varphi([f]_U)$ by
\Los's theorem for Boolean ultrapowers in the functional presentation. To complete the proof, we simply replace $V^{\downarrow\B}_U$ with
its isomorphic copy $\check V_U$ via theorem \ref{Theorem.TwoPresentationsIsomorphic}.\end{proof}

\begin{lemma}\label{Lemma.RefinementDiagram}
Suppose that $\B$ is a complete Boolean algebra with ultrafilter $U$ and maximal antichains $A$ and $B$. If\/ $B$ refines $A$, then there
is an elementary embedding $\pi_{A,B}:V^A/U_A\to V^B/U_B$ making the following diagram commute.
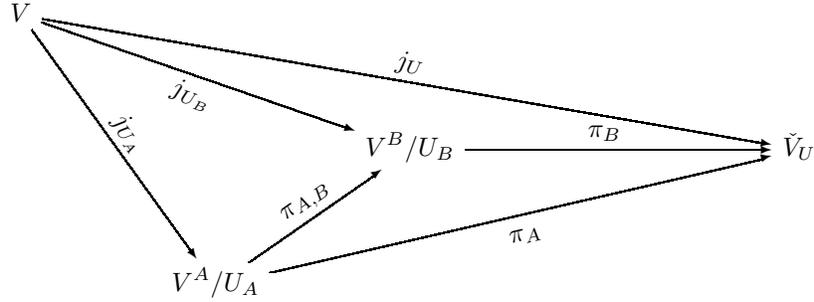
\begin{diagram}[height=2.5em,width=3.6em]
 V                     \\
            & \rdTo(4,2)_{j_{U_B}}   \rdTo(8,2)^{j_U} \rdTo(2,4)^{j_{U_A}}\\
   &   & &                         &     V^B/U_B   &   & \rTo^{\pi_B}    &  & \check V_U      \\
   &   &          &  \ruTo^{\pi_{A,B}}      &               &   &                    & \ruTo(6,2)_{\pi_A} \\
 & &  V^A/U_A               \\
\end{diagram}
\end{lemma}

\begin{proof} We simply define $\pi_{A,B}:[f]_{U_A}\mapsto
[f\downarrow B]_{U_B}$. This is well defined, because if $f\equiv_{U_A} g$, then $f\downarrow B\equiv_{U_B} g\downarrow B$. By \Los's
theorem, the map $\pi_{A,B}$ is easily seen to be elementary, using the fact that $\bigvee\set{a\in A\st
V\satisfies\varphi[f(a)]}=\bigvee\set{b\in B\st V\satisfies\varphi[(f\downarrow B)(b)]}$, which holds since every element of $A$ is the
join of the elements of $B$ refining it.\end{proof}

We now assemble these facts into the general conclusion that the Boolean ultrapower $j_U$ is the direct limit of these induced power set
ultrapowers $j_{U_A}$. Let $\mathcal{I}$ be the collection of maximal antichains $A\of \B$, ordered under refinement. This is a (downwards)
directed order, because any two maximal antichains have a common refinement. Lemmas \ref{Lemma.UltrapowerFactor} and
\ref{Lemma.RefinementDiagram} produce a commutative system of embeddings from $V$ to the various $V^A/U_A$ for $A\in\mathcal{I}$. Let us
denote the direct limit of this system by $\dirlim\<V^A/U_A\st j_{U_A},\pi_{A,B}>$. Elements of this direct limit consist of {\df threads},
or equivalence classes formed by identifying any element $x\in V^A/U_A$ with all its images $\pi_{A,B}(x)$ whenever $B$ refines $A$. For
every $A\in\mathcal{I}$, there is a natural map $\pi_{A,\infty}:V^A/U_A\to\dirlim\<V^A/U_A\st j_{U_A},\pi_{A,B}>$ defined by mapping any
$x\in V^A/U_A$ to its thread, and these maps commute with the maps in the diagram above, so that
$\pi_{A,\infty}=\pi_{B,\infty}\compose\pi_{A,B}$. Similar direct limits were considered in
\cite{OuwehandRose1998:FiltralPowersOfStructures}.

\begin{theorem}\label{Theorem.BooleanUltrapowerAsDirectLimit}
The Boolean ultrapower $j_U:V\to \check V_U$, where $U$ is any ultrafilter on a complete Boolean algebra $\B$, is isomorphic to
$\dirlim\<V^A/U_A\st j_{U_A},\pi_{A,B}>$, the direct limit of the induced commutative system of power set ultrapowers $j_{U_A}:V\to
V^A/U_A$, indexed by maximal antichains in $\B$ ordered by refinement.
\end{theorem}

\begin{proof} Let us consider $V^{\downarrow\B}_U$ in place of $\check V_U$, as these are isomorphic. Theorem
\ref{Theorem.TwoPresentationsIsomorphic} shows that every element of $V^{\downarrow\B}_U$ has the form $[f]_U=\pi_A([f]_{U_A})$ for some
$A\in\mathcal{I}$, and so the maps $\pi_A$ are collectively onto $V^{\downarrow\B}_U$. Since the diagram of lemma
\ref{Lemma.RefinementDiagram} commutes, the association of $[f]_{U_A}\in V^A/U_A$ with $[f]_U$ is well defined on the threads. Since this
association is also onto (and preserves $\in$), it provides an isomorphism of the Boolean ultrapower with the direct limit.\end{proof}

Two degenerate instances of this direct limit phenomenon help to shed light on its nature. First, in the case that $U$ is $V$-generic, then
we know by theorem \ref{Theorem.GenericIffOnto} that the Boolean ultrapower $j_U$ is (isomorphic to) the identity function. This conforms
with theorem \ref{Theorem.BooleanUltrapowerAsDirectLimit}, because in this case we know $U$ meets every maximal antichain $A$, and so every
$U_A$ is a principal ultrafilter. Thus, in the case of generic filters, what we have essentially is the identity map as a direct limit of a
directed system of identity maps. Second, if the Boolean algebra $\B$ is a power set algebra, the power set of some set $Z$, then there is
a terminal node in the directed system $\mathcal{I}$, corresponding to the antichain of singletons $A=\set{\singleton{a}\st a\in Z}$, the
atoms of $\B$, which cannot be further refined. All threads therefore terminate in $V^A/U_A$ and the direct limit $j_U$ is consequently the
same as $j_{U_A}$, just as in theorem \ref{Theorem.BooleanUltrapowersGeneralizeUsualUltrapowers}.

Let us conclude this section by explaining how the direct limit provides an extender-like representation of the Boolean ultrapower.

\begin{theorem}\label{Theorem.ExtenderRepresentation}
If $j:V\to\Vbar$ is the Boolean ultrapower by the ultrafilter $U$ on the complete Boolean algebra $\B$, then
$\Vbar=\set{j(f)(b)\st f:\B\to V, b\in j(\B)}$. More precisely, $\Vbar=\set{j(f)(b_A)\st A\of\B\text{ maximal antichain}, f:A\to V}$, where
$b_A$ is the unique member of $j(A)$ in the $\Vbar$-generic filter $G$ on $j(\B)$ arising from $[\dot G]_U$.
\end{theorem}

\begin{proof} Every element of the classical ultrapower
$V^A/U_A$, of course, has the form $[f]_{U_A}$, where $f:A\to V$ is a function on a maximal antichain $A\of\B$, and furthermore this is the
same as $j_{U_A}(f)([\id]_{U_A})$. Thus, every element of $\check V_U$ has the form
$\pi_{A,\infty}(j_{U_A}(f)([\id]_{U_A}))=\pi_{A,\infty}(j_{U_A}(f))(\pi_{A,\infty}([\id]_{U_A})=j(f)(b)$, where
$b=\pi_{A,\infty}([\id]_{U_A})$, which is in $j(\B)$ because $[\id]_{U_A}$ is an element of $j_{U_A}(A)$, which is a subset of
$j_{U_A}(\B)$. Note that $b$ depends only on $A$, and not on $f$, so we may call it $b_A$. The identity function $\id$ on $A$ corresponds
via theorem \ref{Theorem.TwoPresentationsIsomorphic} to the name $\tau_A$ such that $\[\tau_A=\check a]=a$ for all $a\in A$. If $\dot G$ is
the canonical name of the generic filter, then $\[\check a\in\dot G]=a$, and so $\[\check A\intersect\dot G=\singleton{\tau_A}]=1$. Thus,
$j(A)\intersect G=\singleton{b_A}$ in $\Vbar[G]$, as desired.\end{proof}

We also get a converse characterization.

\begin{theorem}\label{Theorem.FilterSeedCharacterizationOfBooleanUltrapower}
Suppose that $\B$ is a complete Boolean algebra and $j:V\to\Vbar$ is an embedding, such that there is a filter $F\of j(\B)$ that is
$\ran(j)$-generic and $\Vbar=\set{j(f)(b_A)\st A\of\B\text{ maximal antichain},f:A\to V}$, where $b_A$ is the unique member of\/
$F\intersect j(A)$. Then $j:V\to\Vbar$ is isomorphic to the Boolean ultrapower by $U=j^\inverse F\of\B$. In this case, $F$ is actually
$\Vbar$-generic and $\Vbar[F]$ is isomorphic to $V^\B/U$.
\end{theorem}

\begin{proof} Notice that $U=j^\inverse F$ is an ultrafilter on
$\B$, since $F$ meets the maximal antichain $\singleton{j(b),j(\neg b)}$. We define an isomorphism $\pi:V^{\downarrow\B}_U\cong\Vbar$ by
$\pi:[f]_U\mapsto j(f)(b_A)$ for any spanning function $f:A\to V$. This is well defined, since if $[f]_U=[g]_U$ for two spanning functions
$f:A\to V$ and $g:B\to V$, then we may find a common refinement $C\leq A,B$ such that $\bigvee\set{c\in C\st (f\downarrow C)(c)=(g\downarrow
C)(c)}\in U$, which means that $j$ of this value is in $F$. Since $F$ is $\ran(j)$-generic, it selects a unique element $b_C$ from $j(C)$,
and so this means that $j(f\downarrow C)(b_C)=j(g\downarrow C)(b_C)$. Since $C$ refines $A$ and $B$, it follows that $b_C$ must be below
$b_A$ and $b_B$, so this implies $j(f)(b_A)=j(g)(b_B)$, as desired. By replacing $=$ with $\in$ in this argument, we see that $\pi$ is also
$\in$-preserving. Since $\pi$ is surjective by assumption, it is an isomorphism. Observe that $j=\pi\compose j_U$, since for constant
functions $c_x:A\to V$ we have $\pi:[c_x]_U\mapsto j(c_x)(b_A)=j(x)$. So $j$ is isomorphic to the Boolean ultrapower.

We claim next that $F$ is actually $\Vbar$-generic. Suppose that $D\of j(\B)$ is a dense subset of $j(\B)$ in $\Vbar$. Thus, $D=j(\vec
D)(b_A)$ for some spanning function $\vec D=\<D_a\st a\in A>$, and we may as well assume that $D_a\of\B$ is dense for each $a\in A$. Let
$\bar D=\union_{a\in A} D_a\intersect\B_{\smallleq a}$, where $\B_{\smallleq a}=\set{b\in\B\st b\leq a}$. This is clearly dense, since any
$b\in\B$ can be refined below such an $a\in A$, and then into $D_a$, which is dense. Since $F$ is $\ran(j)$-generic, it follows that there
is some $b\in F\intersect j(\bar D)$. Necessarily, $b\leq b_A$, since every element of $\bar D$ is below an element of $A$. Thus, by
elementarity and the definition of $\bar D$, it follows that $b\in j(\vec D)(b_A)=D$. So $F\intersect D\neq\emptyset$, and so $F$ is
$\Vbar$-generic, and consequently also an ultrafilter on $j(\B)$.

Finally, we argue that $V^\B/U\cong \Vbar[F]$. To see this, let $\pi:[\tau]_U\mapsto \val(j(\tau),F)$. This is well defined, because if
$[\tau]_U=[\sigma]_U$, then $\[\tau=\sigma]^\B\in U$, so $\[j(\tau)=j(\sigma)]^{j(\B)}\in F$, which implies
$\val(j(\tau),F)=\val(j(\sigma),F)$. By replacing $=$ with $\in$, we see that $\pi$ preserves $\in$. And finally, every element of
$\Vbar[F]$ has the form $\val(\tau^*,F)$ for some $\tau^*\in\Vbar$. But $\tau^*=j(\vec\tau)(b_A)$ for some spanning function
$\vec\tau=\<\tau_a\st a\in A>$. By the mixing lemma, there is a name $\tau$ such that $\[\tau=\tau_a]^\B\geq a$. Applying $j$, we conclude
that $\[j(\tau)=j(\vec\tau)(b_A)]^{j(\B)}\geq b_A$, which means $\[j(\tau)=\tau^*]^{j(\B)}\in F$. Thus, $\val(\tau^*,F)=\val(j(\tau),F)$,
and so $\pi$ is surjective. Thus, $\pi$ is an isomorphism, as desired. Since $\pi:[\check x]_U\to \val(j(x\check),F)=j(x)$, this
isomorphism respects the Boolean ultrapower map. So $j:V\to\Vbar\of\Vbar[F]$ is isomorphic to the Boolean ultrapower by $U$.\end{proof}

Similarly, we may weaken the genericity hypothesis on $G$ in theorem \ref{Theorem.BooleanUltrapowerWellFounded} to mere $\ran(j)$-genericity.

\begin{theorem}\label{Theorem.RudinKeislerIff}
Suppose that $U\of\B$ is an ultrafilter on a complete Boolean algebra $\B$ and $j:V\to M$ is an elementary embedding, with $M$ not
necessarily well-founded. Then the following are equivalent:
\begin{enumerate}
 \item $j_U$ is an elementary factor of $j$, as in the diagram below.
    
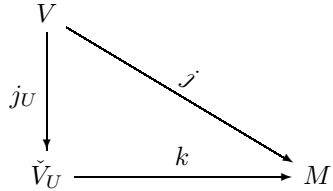
\begin{diagram}[width=5em]
     V \\
     \dTo^{j_U} & \rdTo^j \\
     \check V_U & \rTo^k & M \\
    \end{diagram}
 \item There is a $\ran(j)$-generic filter $H\of j(\B)$ extending $j\image U$. That is, $j\image U\of H$ and $H\intersect
     j(A)\neq\emptyset$ for all maximal antichains $A\of\B$.
\end{enumerate}
\end{theorem}

\begin{proof} Suppose that $j_U$ is a factor of $j$ for some
elementary embedding $k$ as in the diagram. Let $G=[\dot G]_U$ be the canonical $\check V_U$-generic filter determined by $U$, which is
$\check V_U$-generic. Let $H$ be the filter generated by $k\image G$. If $A\of\B$ is any maximal antichain in $V$, then $G\intersect
j_U(A)\neq\emptyset$, and so $H\intersect j(A)\neq\emptyset$, as desired for (2).

Conversely, suppose that there is a filter $H$ as in (2). For each maximal antichain $A\of\B$, let $b_A$ be the unique element of $j(A)$
selected by $H$. Note that if a maximal antichain $C$ refines $A$, then $b_C\leq b_A$. To define $k$, it will be convenient to use the
functional presentation of the Boolean ultrapower. Specifically, define $k:V^{\downarrow\B}_U\to M$ by $k:[f]_U\mapsto j(f)(b_A)$ for any
spanning function $f:A\to V$. This is well defined, since if $f=_U g$ for two spanning functions $f:A\to V$ and $g:B\to V$, then for any
common refinement $C$ of $A$ and $B$ we have $\bigvee\set{c\in C\st (f\downarrow C)(c)=(g\downarrow C)(c)}\in U$, which implies that $j$ of
this join is in $j\image U$, and consequently that $\bigvee\set{c\in j(C)\st (j(f)\downarrow j(C))(c)=(j(g)\downarrow j(C))(c)}\in H$.
Since $H\intersect j(C)=\singleton{b_C}$, this implies $(j(f)\downarrow j(C))(b_C)=(j(g)\downarrow j(C))(b_C)$. Since $j(C)$ refines
$j(A)$, the unique element of $j(A)$ above $b_C$ is $b_A$, as these are both in $H$ and hence compatible, while all members of $j(A)$ not
above $b_C$ will be incompatible with $b_C$. By the corresponding fact with $j(B)$, it follows that $j(f)(b_A)=j(g)(b_A)$, and so $k$ is
well defined. A similar argument now shows that $k$ is elementary. To illustrate, if $V^{\downarrow\B}_U\satisfies\varphi[[f]_U]$ for a
spanning function $f:A\to V$, then $\bigvee\set{a\in A\st V\satisfies\varphi[f(a)]}\in U$, which implies $\bigvee\set{a\in j(A)\st
M\satisfies\varphi[j(f)(a)]}\in j\image U\of H$. Since $H$ selects exactly one member $b_A$ from $j(A)$, this means
$M\satisfies\varphi[j(f)(b_A)]$, and so $k$ is elementary, as desired.\end{proof}

\section{Boolean ultrapowers via partial orders}\label{Section.BooleanUltrapowersViaPartialOrders}

Set theorists often prefer to understand forcing in terms of partial orders or incomplete Boolean algebras, rather than complete Boolean
algebras, and this works fine when the ultrafilters being considered are generic. When the focus is as it is here on non-generic
ultrafilters, however, then subtle issues arise. The principal initial difficulty is that without a certain degree of genericity, an
ultrafilter on a dense subset of a Boolean algebra simply may not generate an ultrafilter on the whole Boolean algebra.

If $\P$ is a partial order, then let us denote by $\B=\B(\P)$ the Boolean algebraic completion of $\P$, which is isomorphic to the
collection of regular open subsets of $\P$, the sets that are the interior of their closure. It is easy to see that $\B$ is a complete
Boolean algebra, and $\P$ maps in an order preserving manner to a dense subset of $\B$. In the case that $\P$ is separative, we may regard
$\P\of\B$ as a dense set itself. It is well known that forcing with $\P$ is equivalent to forcing with $\B$, because every $V$-generic
filter on $\P$ generates a $V$-generic filter on $\B$ and vice versa. Without genericity, however, the connection breaks down:

\begin{theorem} If\/ $\P$ is an incomplete Boolean algebra, then there is an ultrafilter $U\of\P$ that does not generate an ultrafilter on
the completion $\B=\B(\P)$.
\end{theorem}

\begin{proof} We may regard $\P\of\B$. Fix any
$a\in\B\setminus\P$. Thus also $\neg a\in\B\setminus\P$. Let $F=\set{p\in\P\st a\leq p\text{ or }\neg a\leq p}$. We use the fact that
$a\notin\P$ to see that $F$ has the finite intersection property: if $a\leq p_1,\ldots,p_n$ and $\neg a\leq q_1,\ldots,q_m$, then let
$p=p_1\wedge\cdots\wedge p_n$ and $q=q_1\wedge\cdots\wedge q_m$ and observe that because $a\notin\P$, it must be that $a<p$ and $\neg a<q$.
Thus, the difference $p-a=p\wedge\neg a$ is strictly above $0$. Since $\neg a<q$, this implies $0<p\wedge q$, and so $F$ has the finite
intersection property. Thus, there is an ultrafilter $U\of\P$ with $F\of U$. Note that $U$ can contain no elements below $a$, since $p<a$
implies $\neg a<\neg p$ and so $\neg p\in F\of U$. Similarly, $U$ contains no elements below $\neg a$ since $q<\neg a$ implies $a<\neg q$
and so $\neg q\in F\of U$. Therefore, the filter generated by $U$ in $\B$, namely $\bar U=\set{b\in\B\st \exists x\in U\, x\leq b}$,
contains neither $a$ nor $\neg a$, and consequently is not an ultrafilter.\end{proof}

Nevertheless, a mild genericity condition on $U$ will ensure that the generated filter $\bar U$ is an ultrafilter. We define that a {\df
$2$-split maximal antichain} is a partition of a maximal antichain $A=A_0\sqcup A_1$ into two disjoint nonempty pieces. A filter $U$ in
$\P$ {\df weakly decides} this split antichain if there is some condition $p\in U$ such that $p\perp A_0$ or $p\perp A_1$, meaning that
either $p$ is incompatible with every element of $A_0$ or $p$ is incompatible with every element of $A_1$. This is a weak form of
genericity, because if $U$ actually contains an element of $A$, then this element will of course weakly decide the antichain, since it is
incompatible with all other elements of $A$.

\begin{theorem}\label{Theorem.FilterOnPosetGeneratesUltrafilter}
Suppose that $U$ is an ultrafilter on a separative partial order $\P$, regarded as a suborder of\/ $\B$, the regular open algebra of\/
$\P$. Then the following are equivalent:
\begin{enumerate}
 \item The filter $\bar U$ generated by $U$ in $\B$ is an ultrafilter.
 \item $U$ weakly decides all $2$-split maximal antichains in $\P$.
\end{enumerate}
\end{theorem}

\begin{proof} Let $\bar U=\set{b\in\B\st \exists p\in U\, b\geq
p}$ be the filter generated by $U$ in $\B$. For any element $a\in\B$, let $A_0\of\P$ be a maximal antichain in $\P$ below $a$, and let
$A_1\of\P$ be a maximal antichain in $\P$ below $\neg a$. It follows that $A=A_0\sqcup A_1$ is a $2$-split maximal antichain in $\P$. If
$a\in\bar U$, then there is some $p\leq a$ with $p\in U$. It follows that $p\perp\neg a$ and consequently $p\perp A_1$. Conversely, if
$p\perp A_1$, then in $\B$ it must be that $p\perp \vee A_1=\neg a$, and so $p\leq a$. Thus, $a\in\bar U$ if and only if there is some
$p\in U$ with $p\perp A_1$. It follows that $\bar U$ contains $a$ or $\neg a$ if and only if $U$ weakly decides $A=A_0\sqcup
A_1$.\end{proof}

A generalization of this idea provides a characterization of well-foundedness. If $\P$ is a partial order, then we define that a {\df
countably split maximal antichain} is a partition $A=\bigsqcup_n A_n$ of a maximal antichain $A$ into countably many sub-antichains,
allowing $A_n=\emptyset$ to handle finite partitions. A filter $U$ {\df weakly decides} such a split antichain, if there is a condition
$p\in U$ and a natural number $n$ such that $p\perp A_k$ for all $k\neq n$. Thus, the condition $p$ rules out all the other partitions
except $A_n$.

\begin{theorem} Suppose that $\P$ is a partial order and $U$ is an ultrafilter on $\P$ in $V$. Then the following are equivalent:
 \begin{enumerate}
  \item The filter $\bar U$ generated by $U$ in $\B=\RO(\P)$ is a well-founded ultrafilter.
  \item $U$ weakly decides all countably split maximal antichains in $\P$.
 \end{enumerate}
\end{theorem}

\begin{proof} Suppose that the filter $\bar U$ generated by $U$
in $\B$ is an ultrafilter and well-founded, and suppose $A=\bigsqcup_n A_n$ is a countably split maximal antichain in $\P$. In $\B$, let
$a_n=\vee A_n$. If $k\neq n$, then every element of $A_k$ is incompatible with every element of $A_n$, and so $a_k\wedge a_n=0$. Observe
that $\bigvee_n a_n=\vee (\Union_n A_n)=\vee A=1$, and so $\set{a_n\st n<\omega}$ is a countable maximal antichain in $\B$. Since $\bar U$
is well-founded, it follows by theorem \ref{Theorem.WellFoundedEquivalents} that $a_n\in\bar U$ for some $n<\omega$. Thus, there is some
$p\in U$ with $p\leq a_n$. If $k\neq n$, then since $a_k\perp a_n$, it follows that $p\perp A_k$, and so $U$ weakly decides the countably
split maximal antichain, as desired.

Conversely, suppose that $U$ weakly decides all countably split maximal antichains. In particular, this means that $U$ weakly decides all
$2$-split maximal antichains, and so by theorem \ref{Theorem.FilterOnPosetGeneratesUltrafilter} the filter $\bar U$ generated by $U$ in
$\B$ is an ultrafilter. To see that $\bar U$ is well-founded, it suffices by theorem \ref{Theorem.WellFoundedEquivalents} to verify that
$\bar U$ meets all countable maximal antichains in $\B$. If $\set{a_n\st n<\omega}\of\B$ is such a countable maximal antichain, then since
$\P$ is dense in $\B$ there are antichains $A_n\of\P$ such that $a_n=\vee A_n$ in $\B$. (That is, every element of $\B$ is the join of an
antichain in $\P$.) Thus, $A=\bigsqcup_n A_n$ is a partition as in (2). Thus, there is some $p\in U$ and $n<\omega$ such that $p\perp A_k$
for all $k\neq n$. It follows that $p\wedge a_k=0$ for all $k\neq n$, and so $p\leq a_n$. Thus, $a_n\in \bar U$, and so $\bar U$ meets the
antichain, as desired.\end{proof}
%
%

Set theorists are quite used to a generic filter $G\of\B$ being generated by a linearly ordered set. This occurs, for example, when forcing
with a tree, a common situation, for in this case the generic filter is determined by a path through this tree. Nevertheless, we prove next
that nonprincipal ultrafilters in the ground model can never be generated in this way by a linearly ordered set.

\begin{lemma}\label{Lemma.NoUltrafilterIsLinearlyGenerated} If\/ $U\of\B$ is a nonprincipal ultrafilter on a
complete Boolean algebra, then $U$ is not generated by any linearly ordered collection of elements. In particular, no nonprincipal
ultrafilter on a complete Boolean algebra is generated by a countable set.
\end{lemma}

\begin{proof} This lemma refers to $U$ in $V$ on a complete
Boolean algebra $\B$ in $V$. If $U$ is generated by a linearly ordered subset $L\of U$, in the sense that $b\in U\iff\exists c{\in}L\,
c\leq b$, then we may choose a cofinal descending sequence $\<c_\alpha\st \alpha<\gamma>$ such that $c_0\geq c_1\geq\cdots\geq
c_\alpha\geq\cdots$ and $U$ is generated by $\set{c_\alpha\st \alpha<\gamma}$. We may assume without loss of generality that $c_0=1$ and
$c_\lambda=\bigwedge_{\alpha<\lambda}c_\alpha$ for limit ordinals $\lambda<\gamma$. Since $U$ is nonprincipal, it follows that
$\bigwedge_{\alpha<\gamma}c_\alpha=0$. The difference elements $d_\alpha=c_\alpha-c_{\alpha+1}$, therefore, constitute a maximal antichain
$\set{d_\alpha\st\alpha<\gamma}$. Let $x=\bigvee\set{d_\alpha\st \alpha\text{ is even}}$, the join of the even differences. It follows
that the negation $\neg x=\bigvee\set{d_\alpha\st \alpha\text{ is odd}}$ is the join of the odd differences. Note that every $c_\alpha$ has
nonzero meet with both $x$ and $\neg x$, since $c_\alpha=\bigvee_{\alpha\leq \beta}d_\beta$, which includes both even and odd differences
beyond $\alpha$. Thus, no $c_\alpha$ is below $x$ or $\neg x$, and so the filter generated by the $c_\alpha$ includes neither $x$ or $\neg
x$, a contradiction. So $U$ cannot have been generated by a linearly ordered set. This implies, in particular, that $U$ cannot be generated
by a countable set $\set{b_n\st n<\omega}$, for then it would be generated by the linearly ordered set $b_0\geq (b_0\wedge b_1)\geq
(b_0\wedge b_1\wedge b_2)\geq\cdots$, which is impossible.
\end{proof}

If one forces to add a $V$-generic ultrafilter $G\of\B$ that is generated by a linearly ordered set in the extension, which as we said is
often the case, then this lemma shows that $G$ does not generate an ultrafilter on the {\it new} completion of $\B$, meaning the regular
open algebra of $\B$ as computed in $V[G]$.

%
We conclude this section with an intriguing example that illustrates some of the subtleties of working with partial orders rather than
complete Boolean algebras.

\begin{example}\rm Consider the forcing $\Add(\omega_1,1)$, the complete Boolean algebra corresponding to the the partial order
$\P=2^{\ltomega_1}$, which adds a subset to $\omega_1$ by initial segments. This forcing $\P$ is a tree. Suppose that $U$ is an ultrafilter
on $\P$ in $V$, meaning that $U$ is a maximal filter. It is easy to see that $U$ is generated by the branch $b=\Union U$, which has order
type $\omega_1$. Every countable maximal antichain $A$ in $\P$ is bounded and therefore refined by a maximal antichain consisting of any
level of the tree above this bound. Since $b$ passes through all the levels of $\P$, it follows that $U$ meets $A$. So we have established
that $U$ meets every countable maximal antichain in $\P$. A naive reading of theorem \ref{Theorem.WellFoundedEquivalents} would suggest,
consequently, that the Boolean ultrapower by $U$ must be well-founded. But this is absurd, since the critical point of the ultrapower would
therefore have to be a measurable cardinal, even though we made no large cardinal assumption, and anyway the size of this forcing is far
below the least measurable cardinal. So it is impossible for this Boolean ultrapower to be well-founded. What is going on? The answer is
that although $U$ is an ultrafilter on $\P$, it follows by lemma \ref{Lemma.NoUltrafilterIsLinearlyGenerated} that it does not generate an
ultrafilter on the corresponding Boolean algebra, so there is no such thing as the Boolean ultrapower by $U$. Although one can extend $U$
to an ultrafilter on the Boolean algebra, it is not possible to do this in such a way so as to meet all the countable maximal antichains
there. The Boolean algebra simply has many more countable maximal antichains than the partial order does, and many of these do not arise
from countable maximal antichains in $\P$.
\end{example}

\section{Subalgebras, iterations and quotients}\label{Section.SubalgebrasEtc}

In this section, we investigate the interaction of the Boolean ultrapowers arising from complete subalgebras, iterations and the
corresponding quotients of complete Boolean algebras. A Boolean algebra $\B$ is a {\df subalgebra} of another $\C$, if $\B\of\C$ and the
Boolean operations $\wedge,\vee,\neg$ for elements in $\B$ are computed the same in $\B$ as they are in $\C$. This subalgebra is {\df
complete} if the infinitary meets and joins are also computed the same in $\B$ as in $\C$. It follows that any maximal antichain in $\B$
remains maximal in $\C$ (and this property alone is sufficient for a subalgebra to be complete). In order to avoid confusion when comparing
the Boolean ultrapowers by different Boolean algebras, we introduce the notation $\check V_{(\B)}$ and $\check V_{(\C)}$ to indicate the
respective Boolean-valued ground models of $V^\B$ and $V^\C$, built from $\B$-names and $\C$-names, respectively.

\begin{theorem}\label{Theorem.CompleteSubalgebraElementary}
Suppose that $\B$ is a complete subalgebra of\/ $\C$. Then:
 \begin{enumerate}
  \item The ground model $\check V_{(\B)}$ is a Boolean elementary submodel of\/ $\check V_{(\C)}$, meaning for any formula $\varphi$
      and $\B$-names $\tau_i$ with $\[\tau_i\in\check V]^\B=1$ that
    $$\[\varphi^{\check V}(\tau_0,\ldots,\tau_n)]^\B=\[\varphi^{\check V}(\tau_0,\ldots,\tau_n)]^\C.$$
 \item The forcing extension $V^\B$ is a Boolean elementary submodel of the extension $\check V[\dot G\intersect\check\B]$ in $V^\C$,
     meaning for any formula $\varphi$ and $\B$-names $\tau_i$ that
     $$\[\varphi(\tau_0,\ldots,\tau_n)]^\B=\[\varphi^{\check V[\dot G\intersect\check\B]}(\tau_0,\ldots,\tau_n)]^\C.$$
 \item The forcing extension $V^\C$ believes with Boolean value one that it is the forcing extension of\/ $\check V[\dot
     G\intersect\check\B]$ by the quotient forcing $\dot G\of\check\C/(\dot G\intersect\check\B)$.
\end{enumerate}
\end{theorem}

\begin{proof} This theorem can be viewed as a generalization of
lemma \ref{Lemma.VandVcheck}, since $\singleton{0,1}$ is a complete subalgebra of any Boolean algebra. We prove (1) by induction on
$\varphi$. The atomic case is proved by sub-induction on the names appearing in $\[\tau\in\sigma]$ and $\[\tau=\sigma]$. The point is that
for $\B$-names, the calculation of this Boolean value in $\B$ or $\C$ will be the same, precisely because the supremum of a set of values
in $\B$ is the same whether computed in $\B$ or in $\C$. Proceeding with the induction on $\varphi$, one easily handles Boolean
combinations. For the quantifier case, observe that $\[(\exists x\,\varphi(x,\vec\tau))^{\check V}]^\B=\bigvee_{x\in V}\[\varphi^{\check
V}(\check x,\vec\tau)]^\B=\bigvee_{x\in V}\[\varphi^{\check V}(\check x,\vec\tau)]^\C=\[(\exists x\,\varphi(x,\vec\tau))^{\check V}]^\C$,
using the induction hypothesis in the second equality and the fact that $\check x$ is the same whether computed with $\B$ or $\C$.

We prove (2) by induction on $\varphi$. The atomic case follows from the proof of (1), and Boolean combinations are handled easily as
before. Before handling quantifiers, we first observe that if $\sigma\in V^\B$, then because $\sigma$ refers only to Boolean values in
$\B$, it follows that $\[\val(\check\sigma,\dot G)=\val(\check\sigma,\dot G\intersect\check\B)]^\C=1$, and consequently\break $\[\sigma\in
{\check V[\dot G\intersect\check\B]}]^\C=1$. More generally, for any $\C$-name $\tau$ we have $\[\tau\in{\check V[\dot
G\intersect\check\B]}]^\C=\bigvee_{\sigma\in V^\B}\[\tau=\sigma]^\C$, since any such $\tau$, to the extent that it is in $\check V[\dot
G\intersect\B]$,  will be a mixture of such names $\sigma\in V^\B$. We may now observe
$$\begin{array}{rcl}
\[\exists x\,\varphi(x,\vec\tau)]^\B &=& \break\bigvee_{\sigma\in V^\B}\[\varphi(\sigma,\vec\tau)]^\B \\
  &=& \bigvee_{\sigma\in V^\B}\[\varphi^{\check V[\dot G\intersect\check\B]}(\sigma,\vec\tau)]^\C\wedge\[\sigma\in{\check V[\dot G\intersect\check\B]}]^\C \\
  &=& \[\exists x\in{\check V[\dot G\intersect\B]}\,\varphi^{\check V[\dot G\intersect\check\B]}(x,\vec\tau)]^\C \\
  &=& \[(\exists x\,\varphi(x,\vec\tau))^{\check V[\dot G\intersect\check\B]}]^\C, \\
\end{array}$$
as desired for (2).


Statement (3) amounts to the standard fact about quotient forcing. Specifically, suppose that $\B\of\C$ is a complete subalgebra of
complete Boolean algebras and $G\of\C$ is $V$-generic. It follows easily that $G_0=G\intersect\B$ is $V$-generic for forcing over $\B$. In
$V[G_0]$, one may form the quotient partial order $\C/G_0$, consisting of all $c\in C$ that are compatible with every element of $G_0$, for
which we now argue that $G$ is $V[G_0]$-generic. Suppose that $D\of\C/G_0$ is open dense, and $D\in V[G_0]$. Choose a $\B$-name $\dot D$
such that $D=\val(\dot D,G_0)$ and $\dot D$ is forced by $1$ over $\B$ to be open dense in $\check\C/(\dot G\intersect\check\B)$. Since
$\dot D$ is a $\B$-name, it is also a $\C$-name, and one may easily establish $\val(\dot D,G_0)=\val(\dot D,G)$. It follows that $1$ forces
over $\C$ that $\dot D$ is open dense in $\check\C/(\dot G\intersect\check\B)$. In $V$, let $E=\set{c\in\C\st c\forces_\C\check c\in \dot
D}$. To see that $E$ is dense, fix any $d\in \C$. Clearly, $d$ forces $\check d\in\check\C/(\dot G\intersect\check\B)$, and so there is
some stronger $c\leq d$ forcing over $\C$ that some particular $e\leq d$ has $\check e\in\dot D$. But if $c$ forces $\check e\in\dot D$,
then it must be that $c$ and $e$ are compatible, since $\dot D$ was forced to be contained in $\C/(\dot G\intersect\check\B)$. By
strengthening $c$, we may assume without loss of generality that $c\leq e$. In this case, since $\dot D$ was forced to be open, we see that
$c$ forces $\check c\in\dot D$ also, so $c\in E$ below $d$. So $E$ is dense, and so there is $c\in E\intersect G$, meaning that $c\in
D\intersect G$, as desired. Thus, $G$ is $V[G\intersect\B]$-generic for the quotient forcing $\C/(G\intersect\B)$. Since we've established
this as a general consequence of \ZFC, it holds with Boolean value one in $V^\C$, as stated in (3).\end{proof}

One should not in general expect, of course, that $V^\B$ is a Boolean elementary submodel of $V^\C$, because forcing with $\B$ may result
in a different theory than forcing with $\C$. For example, adding one Cohen real preserves \CH, while adding many Cohen reals will force
$\neg\CH$, even though the forcing $\Add(\omega,1)$ is a complete subalgebra of $\Add(\omega,\theta)$ for any $\theta$. More generally, if
the quotient forcing $\C/(G\intersect\B)$ is nontrivial, then $V^\C$ will have a $V^\B$-generic filter for it, but $V^\B$ will not, and so
except in trivial cases, $V^\B$ is not a $\Sigma_1$-elementary Boolean submodel of $V^\C$.

Having theorem \ref{Theorem.CompleteSubalgebraElementary}, we may now take the quotient by an ultrafilter.

\begin{theorem}\label{Theorem.CompleteSubalgebraFactorEmbedding}
Suppose that $\B$ is a complete subalgebra of\/ $\C$ and $U$ is an ultrafilter on $\C$. Then $U_0=U\intersect\B$ is an ultrafilter on $\B$,
whose corresponding Boolean ultrapower $j_{U_0}:V\to V_0=\check V_{(\B)}/U_0$ is an elementary factor of the Boolean ultrapower $j_U:V\to
\Vbar=\check V_{(\C)}/U$ as in the following commutative elementary diagram.
\begin{diagram}[height=2em]
 V & & \rTo^{j_U} & & \Vbar  & \of\hskip2em & \Vbar[G\intersect j_U(\B)] & \hskip2em\of & \Vbar[G] \\
 &\rdTo_{j_{U_0}}&  & \ruTo^{k\restrict V_0}    & & \ruTo_k \\
 & & V_0 & \of & V_0[G_0]\\
\end{diagram}
The full Boolean extensions are $V_0[G_0]=V^\B/U_0$ and $\Vbar[G]=V^\C/U$, with generic filters $G_0=[\dot G]_{U_0}\of j_{U_0}(\B)$ and
$G=[\dot G]_U\of j_U(\C)$, and $k(G_0)=G\intersect j_U(\B)$.
\end{theorem}

\begin{proof} It is clear that $U_0=U\intersect\B$ is an
ultrafilter on $\B$. So we may form the ultrapowers $j_U:V\to\Vbar$ and $j_{U_0}:V\to V_0$, and we will have $\Vbar\of\Vbar[G]$ and $V_0\of
V_0[G_0]$, using the canonical generic filters arising from $U$ and $U_0$ as in lemma \ref{Lemma.Vcheck/UtransitiveEtc}, so that
$V_0[G_0]=V^\B/U_0$ and $\Vbar [G]=V^\C/U$. To define the factor embedding, let $k:[\tau]^\B_{U_0}\mapsto [\tau]^\C_U$ for any $\B$-name
$\tau$, where the superscript indicates the Boolean algebra that is used. This is well defined since if $\tau\equiv_{U_0}\sigma$, then
$\[\tau=\sigma]^\B\in U_0=U\intersect\B$, and so $\[\tau=\sigma]^\C\in U$ by theorem \ref{Theorem.CompleteSubalgebraElementary} and
consequently $\tau=_U\sigma$ in $V^\C/U$. Note that if $\[\tau\in\check V]^\B=1$, then also $\[\tau\in\check V]^\C=1$, and so $k$ carries
$V_0$ to $\Vbar $. The resulting restriction $k\restrict V_0$ is elementary precisely because of theorem
\ref{Theorem.CompleteSubalgebraElementary}(1), since for $\B$-names $\tau$ with $\[\tau\in\check V]=1$, we have
$V_0\satisfies\varphi([\tau]_{U_0})$ if and only if $\[\varphi^{\check V}(\tau)]^\B\in U_0$, which is equivalent to $\[\varphi^{\check
V}(\tau)]^\C\in U$, since these Boolean values are the same. Similarly, the full embedding $k$ is elementary from $V_0[G_0]$ to $\Vbar
[G\intersect j_U(\B)]$ precisely because of theorem \ref{Theorem.CompleteSubalgebraElementary}(2), since
$V_0[G_0]=V^\B/U_0\satisfies\varphi([\tau]_{U_0})$ if and only if $\[\varphi(\tau)]^\B\in U_0$, which is equivalent to $\[\varphi^{\check
V[\dot G\intersect\check\B]}(\tau)]^\C\in U$, as these Boolean values are the same, and this holds if and only if $\Vbar [G\intersect
j_U(\B)]\satisfies\varphi([\tau]_U)$. Since $G_0=[\dot G_{(\B)}]_{U_0}$, using the canonical name $\dot G_{(\B)}$ for the generic filter on
$\B$, it follows that $k(G_0)$ is $[\dot G_{(\B)}]_U$. But $\[\dot G_{(\B)}=\dot G_{(\C)}\intersect\check\B]^\C=1$, so this means that
$k(G_0)=G\intersect j_U(\B)$ in $\Vbar [G]$. The diagram commutes, because $k$ takes $[\check x]_{U_0}$ to $[\check x]_U$.\end{proof}

We turn now to iterations. Suppose that $\B_0$ is a complete Boolean algebra and $\dot\B_1$ is a full $\B_0$-name of a complete Boolean
algebra. We define the iteration algebra $\B_0*\dot\B_1=\RO(\P)$ to be the regular open algebra of the partial pre-order $\P=\set{(b,\dot
c)\st b\in\B_0^\plus, \dot c\in\dom\dot\B_1,b\forces\dot c\in\dot\B_1^\plus}$, where $\B^\plus$ denotes the collection of nonzero elements in any Boolean algebra $\B$, and where as usual the order is $(b,\dot c)\leq (d,\dot e)\iff b\leq d$
and $b\forces\dot c\leq \dot e$. For convenience, we regard elements of $\P$ as being members of $\B_0*\dot\B_1$, although technically it
is the regular open set generated by the lower cone of the element that is in $\B_0*\dot\B_1$. If $U_0\of\B_0$ is an ultrafilter on $\B_0$
and $U_1\of\B_1$ is an ultrafilter on the resulting complete Boolean algebra $\B_1=[\dot\B_1]_{U_0}$ in $V^{\B_0}/U_0$, then we define the
iteration ultrafilter $U_0*U_1$ by
$$U_0*U_1\quad\text{is the filter generated by}\quad\set{(b,\dot c)\in\P\st b\in U_0, [\dot c]_{U_0}\in U_1}.$$
By the next theorem, this is an ultrafilter on $\B_0*\dot\B_1$.

\begin{theorem}\label{Theorem.IterationsOfBooleanUltrapowers}
If\/ $U_0\of\B_0$ is an ultrafilter on the complete Boolean algebra $\B_0$ and $U_1\of\B_1$ is an ultrafilter on the complete Boolean
algebra $\B_1=[\dot\B_1]_{U_0}$ in the Boolean extension $V^\B/U_0$, where $\dot\B_1$ is a full $\B_0$-name for a complete Boolean algebra,
then $U_0*U_1$ is an ultrafilter on $\B_0*\dot\B_1$, whose Boolean ultrapower obeys the following commutative diagram, where
$V_0[G_0]=V^{\B_0}/U_0$ and $\Vbar[\bar G_0*\bar G_1]=V^{B_0*\dot\B_1}/U_0*U_1$.
\begin{diagram}[height=2em]
 V & & \rTo^{j_{U_0*U_1}} & & \Vbar  & \of & \Vbar[\bar G_0] & \hskip-1em\of & \Vbar[\bar G_0*\bar G_1] \\
 &\rdTo_{j_{U_0}}&  & \ruTo^{j_{U_1}\restrict V_0}    & & \ruTo_{j_{U_1}} \\
 & & V_0 & \of & V_0[G_0]\\
\end{diagram}
Conversely, every ultrafilter $U\of\B_0*\dot\B_1$ has the form $U=U_0*U_1$ for some such ultrafilters $U_0$ and $U_1$.
\end{theorem}

\begin{proof} Because $U_0$ is an ultrafilter on $\B_0$ and $U_1$
is an ultrafilter on $\B_1=[\dot\B_1]_{U_0}$ in $V_0[G_0]$, we may construct the Boolean ultrapowers $j_{U_0}:V\to V_0\of
V_0[G_0]=V^{\B_0}/U_0$ and $j_{U_1}:V_0[G_0]\to \Vbar[\bar G_0]\of\Vbar[\bar G_0][\bar G_1]=(V_0[G_0])^{B_1}/U_1$, where $\bar
G_0=j_1(G_0)$ and $\bar G_1=[\dot G]_{U_1}$, using the canonical name $\dot G$ for the generic filter for $\B_1$ over $V_0[G_0]$. To
simplify notation, we denote $j_{U_0}$ and $j_{U_1}$ simply by $j_0$ and $j_1$, respectively. Let $j=j_1\compose j_0$. Since $\bar G_0$ is
$\Vbar$-generic and $\bar G_1$ is $\Vbar[\bar G_0]$-generic, we may view the forcing extension as $\Vbar[\bar G_0][\bar G_1]=\Vbar[\bar
G_0*\bar G_1]$, arising by forcing over $\Vbar$ with $j(\B_0*\dot\B_1)$, where $\bar G_0*\bar G_1$ is the (ultra)filter generated by the
set of pairs $(x,\dot y)\in j(\B_0*\dot\B_1)$ with $x\in\bar G_0$ and $\val(\dot y,\bar G_0)\in\bar G_1$.

We claim now that $U_0*U_1=j^\inverse(\bar G_0*\bar G_1)$. Since $j_0$ is the Boolean ultrapower of $V$ by $U_0$, we know that
$U_0=j_0^\inverse G_0$. Similarly, since $j_1$ is the Boolean ultrapower of $V_0[G_0]$ by $U_1$, we know $U_1=j_1^\inverse\bar G_1$.
Combining these facts, observe that $(b,\dot c)\in U_0*U_1$ if and only if $b\in U_0$ and $[\dot c]_{U_0}\in U_1$, which holds if and only
if $j_0(b)\in G_0$ and $\val(j_0(\dot c),G_0)\in U_1$. Applying $j_1$, this is equivalent to $j(b)\in\bar G_0$ and $\val(j(\dot c),\bar
G_0)\in G_1$, which means $(j(b),j(\dot c))\in \bar G_0*\bar G_1$, establishing that $U_0*U_1=j^\inverse(\bar G_0*\bar G_1)$, as we
claimed. Since the pre-image of an ultrafilter is an ultrafilter, we conclude that $U_0*U_1$ is an ultrafilter on $\B_0*\dot \B_1$.

Since $\Vbar[\bar G_0]$ is the Boolean ultrapower of $V_0[G_0]$ by $U_1\of\B_1$, it follows that every element of $\Vbar$ has the form
$x=[\tau]_{U_1}$, where $\tau$ is a $\B_1$-name in $V_0[G_0]$ for an element of $\check V_0$. Since $V_0[G_0]$ is the Boolean extension
$V^\B/U_0$, it follows that $\tau=[\tau_0]_{U_0}=\val(j_0(\tau_0),G_0)$, where $\tau_0$ is a $\B_0$-name in $V$. Since $\tau$ is a
$\B_1$-name for an element of $V_0$, we may assume without loss of generality that the $\B_0$-Boolean value that $\tau_0$ names a
$\dot\B_1$-name for an element of the ground model is $1$. That is, $\[\tau_0\text{ is a }\dot\B_1\text{-name and
}\[\check\tau_0\in\check{\check V}]^{\dot\B_1}=\check 1]^{\B_0}=1$, where $\check{\check V}$ is the $\B_0$-name for the $\dot\B_1$-name for
the ground model $V$. In short, $x=[[\tau_0]_{U_0}]_{U_1}$. In $V$, the object $\tau_0$ is a $\B_0$-name of a $\dot\B_1$-name for an
element of $V$. So any condition $(r,\dot s)\in\P\of\B_0*\dot\B_1$ can be strengthened to a condition $(b,\dot c)\leq (r,\dot s)$ such that
for some $z\in V$ we have that $b$ forces over $\B_0$ that $\dot c$ forces over $\dot\B_1$ that the object named by $\tau_0$ is $z$.
Succinctly, $b\forces_{\B_0}( \dot c\forces_{\dot\B_1}\tau_0=\check{\check z})$,
where $\check{\check z}$ is the $\B_0$ check name of the $\dot\B_1$ check name for $z$. Let's denote this $z$ by $z_{b,\dot c}$. We have
argued that the collection $D\of\B_0*\dot\B_1$ consisting of all $(b,\dot c)$ that decide $\tau_0$ in this way is dense. Let $A\of D$ be a
maximal antichain in $V$, and let $f:A\to V$ be the spanning function defined by $f(b,\dot c)=z_{b,\dot c}$. Since $j(A)\of
j(\B_0*\dot\B_1)$ is a maximal antichain in $\Vbar$ and $\bar G_0*\bar G_1$ is $\Vbar$-generic, there is a unique condition $b_A=(d,\dot
e)\in j(A)\intersect\bar G_0*\bar G_1$. Since this is in $j(D)$, it means that there is some $z=j(f)(d,\dot e)$ such that $d$ forces over
$j(\B_0)$ that $\dot e$ forces over $j(\dot B_1)$ that the object named by $j(\tau_0)$ is $z$. But the object named by $j(\tau_0)$ is
$\val(\val(j(\tau_0),\bar G_0),\bar G_1)=\val(\val(j_1(j_0(\tau_0)),j_1(G_0)),\bar G_1)=\val(j_1(\val(j_0(\tau_0),G_0)),\bar
G_1)=\val(j_1([\tau_0]_{U_0}),\bar G_1)=[[\tau_0]_{U_0}]_{U_1}=x$, our original arbitrary object from $\Vbar$. So we have established that
$j(f)(b_A)=x$, and so every element of $\Vbar$ has the form $j(f)(b_A)$ for some spanning function $f:A\to V$ in $V$ on some maximal
antichain $A\of\B_0*\dot\B_1$. From this, it follows by theorem \ref{Theorem.FilterSeedCharacterizationOfBooleanUltrapower} that
$j:V\to\Vbar\of\Vbar[\bar G_0*\bar G_1]$ is the Boolean ultrapower of $V$ by $U_0*U_1\of\B_0*\dot\B_1$, as desired.

Conversely, suppose that $U$ is any ultrafilter on $\B_0*\dot\B_1$. It is easy to see that the projection $U_0=\set{b\in \B_0\st (b,\dot
1)\in U}$ onto the first coordinate is an ultrafilter on $\B_0$. Thus, we may form the corresponding Boolean ultrapower $j_{U_0}:V\to
V_0\of V_0[G_0]=V^{\B_0}/U_0$, and consider the complete Boolean algebra $\B_1=[\dot\B_1]_{U_0}$ in $V_0[G_0]$. It is easy to see that
$U_1=\set{[\dot c]_{U_0}\st \exists b\, (b,\dot c)\in U}$ is an ultrafilter on $\B_1$. Note that $(b,\dot d)\in U$ if and only if $b\in
U_0$ and $[\dot c]_{U_0}\in U_1$, so $U$ and $U_0*U_1$ agree on $\P$, and consequently $U=U_0*U_1$.\end{proof}

It now follows that the factor embedding $k$ of theorem \ref{Theorem.CompleteSubalgebraFactorEmbedding} is isomorphic to the Boolean
ultrapower of $V_0[G_0]$ by the quotient forcing $\B_1=\RO(j_{U_0}(\C)/G_0)$ using the ultrafilter $U_1=k^\inverse G/(G\intersect
j_U(\B))$, and $\Vbar [G]$ there is isomorphic to $V_0[G_0]^{\B_1}/U_1$. The reason is that whenever $\B$ is a complete subalgebra of $\C$,
then $\C$ is isomorphic to $\B*\dot\B_1$, where $\dot\B_1$ is a $\B$-name for the complete Boolean algebra in $V^\B$ corresponding to the
quotient $\check\C/\dot G$. Thus, any ultrafilter $U\of\C$ is isomorphic to an ultrafilter $\tilde U\of\B*\dot\B_1$, which by theorem
\ref{Theorem.IterationsOfBooleanUltrapowers} has the form $U_0*U_1$, where $U_0=U\intersect\B$ and $U_1$ is an ultrafilter in the quotient
forcing $j_0(\C)/G_0$ in $V_0[G_0]$. The map $j_{U_1}$ sends $x\mapsto[\check x]_{U_1}$, using the $\B_1$-name $\check x$. If
$x=[\tau_0]_{U_0}$ for a $\B$-name $\tau_0$, then this $\check x$ is the $\B_1$-name for the same object named by $\tau$ in $V^\B$. The
canonical $\C$-name for this object is, of course, $\tau$ itself, which as a $\B$ name is also a $\C$-name for the same object. So
$j_{U_1}$ has mapped $[\tau_0]_{U_0}$ to (the isomorphic copy of) $[\tau_0]_U$, which is precisely how we defined $k$.

\section{Products}\label{Section.Products}

We now consider products of complete Boolean algebras $\B_0$ and $\B_1$. The first step is to dispense with an incorrect choice for the
product, namely the {\df direct sum} Boolean algebra $\B_0\oplus\B_1=\set{(b,c)\st b\in\B_0, c\in \B_1}$, with coordinate-wise operations. Although this is a complete Boolean algebra, it is not the right choice for product forcing or for product Boolean ultrapowers, because the maximal antichain $\singleton{(1,0),(0,1)}$ in $\B_0\oplus\B_1$ means that any ultrafilter concentrates in effect on only one factor; so this algebra corresponds to what is known as side-by-side forcing or the lottery sum of $\B_0$ and $\B_1$. Instead, for product forcing what one wants is
$$\B_0\times\B_1=\RO(\B_0^\plus\times\B_1^\plus).$$
It is easy to see that any ultrafilter $U\of\B_0\times\B_1$ projects
to ultrafilters $U_0\of\B_0$ and $U_1\of\B_1$ in each factor, and these give rise to the following commutative diagram.

\begin{theorem}\label{Theorem.ProductForcingBooleanUltrapower}
For any ultrafilter $U\of\B_0\times\B_1$, projecting to ultrafilters $U_0\of\B_0$ and $U_1\of\B_1$, there are elementary maps $k_0$ and
$k_1$ making the following diagram commute (dotted lines indicate forcing extensions).
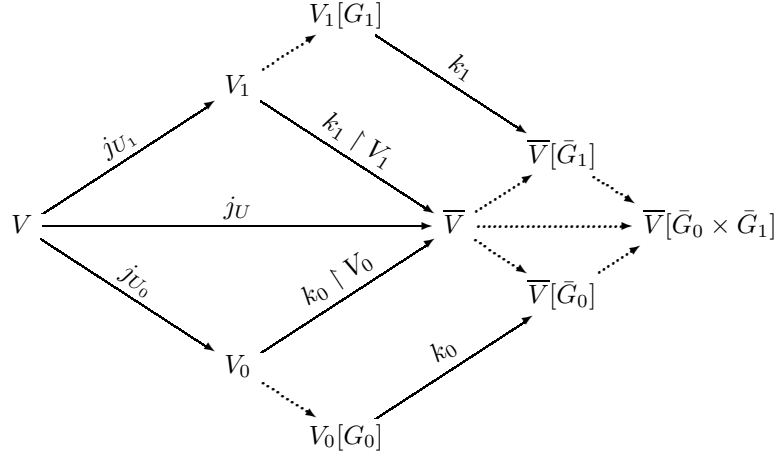
\begin{diagram}[height=1.3em,width=2em,textflow]
   &  &                 &                      &        &           & V_1[G_1]                 &  &        &           &                 &             &           \\
   &  &                 &                      &        & \ruExtension &                          & \rdTo(4,4)^{k_1} &        &           &                 &             &           \\
   &  &                 &                      &   V_1  &           &                          &  &        &           &                 &             &           \\
   &  &                 & \ruTo(4,4)^{j_{U_1}} &        & \rdTo(4,4)^{k_1\restrict V_1}  &                          &  &        &           &                 &             &           \\
   &  &                 &  &                            &           &  &  &        &           & \Vbar[\bar G_1] &             &           \\
   &  &                 &  &                            &           &                          &  &        & \ruExtension &                 & \rdExtension   &           \\
 V &                      &                 &  & \rTo^{j_U}  &           &                          &  &  \Vbar &    & \rExtension      &        & \hskip2em\rlap{\hskip-2em$\Vbar[\bar G_0\times\bar G_1]$} \\
   & \rdTo(4,4)^{j_{U_0}} &                 &  &                            &           &                          & \ruTo(4,4)^{k_0\restrict V_0} &        & \rdExtension &                 & \ruExtension   &           \\
   &  &  &  &                            &           &         &  &        &              & \Vbar[\bar G_0] &             &           \\
   &  &  &  &                            &           &                          &  &        & \ruTo(4,4)^{k_0} &                 &             &           \\
   &  &                 &  &   V_0                      &           &                          &  &        &           &                 &             &           \\
   &  &                 &  &                            & \rdExtension &                          &  &        &           &                 &             &           \\
   &  &                 &  &                            &           & V_0[G_0]                 &  &        &           &                 &             &           \\
\end{diagram}
In the diagram, $V_0=\check V_{U_0}\of V_0[G_0]=V^{\B_0}/U_0$ is the Boolean ultrapower by $U_0$, with $G_0=[\dot G]_{U_0}$; $V_1=\check
V_{U_1}\of V_1[G_1]=V^{\B_1}/U_1$ is the Boolean ultrapower by $U_1$, with $G_1=[\dot G]_{U_1}$; and $\Vbar=\check V_U\of \Vbar[\bar
G_0\times\bar G_1]=V^{\B_0\times\B_1}/U$ is the ultrapower by $U$, with $\bar G_0\times\bar G_1=[\dot G]_U$.
\end{theorem}

\begin{proof} Since $\B_0$ and $\B_1$ are each completely
embedded in $\B_0\times\B_1$, this is a simple consequence of theorem \ref{Theorem.CompleteSubalgebraFactorEmbedding}. The diagram here consists simply
of two copies of the diagram in theorem \ref{Theorem.CompleteSubalgebraFactorEmbedding}, one above and one below the central axis.\end{proof}

We now consider to what extent an ultrafilter $U$ on the product $\B_0\times\B_1$ is determined by its factors $U_0$ and $U_1$. When $U$ is
a generic filter, then a basic fact about product forcing is that the corresponding projections $U_0$ and $U_1$ are mutually generic and
the rectangular product $U_0\sqtimes U_1=\set{(b,c)\st b\in U_0, c\in U_1}$ generates $U$. Corollary \ref{Corollary.U0xU1NotAnUltrafilter}
shows, however, that when $U$ is not generic, then $U_0\sqtimes U_1$ does not always generate an ultrafilter. Indeed, even in the case of
power set algebras, one does not usually expect to measure all subsets of the plane with rectangles. Rather, one defines the product
measure, which slices every subset of the plane into vertical strips, demanding that a large number of them is large. This idea generalizes
naturally to the context of arbitrary complete Boolean algebras as follows. Since the elements of $\B_0\times\B_1$ are exactly the regular
open subsets $X\of\B_0^\plus\cross\B_1^\plus$, we use the notation $X_b=\set{c\st (b,c)\in X}$ and define the {\df product filter}
$U_0\times U_1$ by
$$X\in U_0\times U_1\qquad\text{ if and only if }\qquad\vee\set{b\in\B_0\st \vee X_b\in U_1}\in U_0$$
It is easy to see that $U_0\times U_1$ is indeed a filter on $\B_0\times\B_1$. In the classical power set context, of course, where $U_0\of
P(X)$ and $U_1\of P(Y)$, then $\B_0\times\B_1$ is isomorphic with $P(X\times Y)$ and $U_0\times U_1$ is the familiar product measure, an
ultrafilter measuring all subsets of the plane. But in the general context of complete Boolean algebras, unfortunately, $U_0\times U_1$ need
not be an ultrafilter (see corollary \ref{Corollary.U0oxU1NotAnUltrafilter}), even when $U_0$ and $U_1$ are ultrafilters. In the power set
case, the ultrapower by the product measure $U_0\times U_1$ can be expressed either as the composition of the ultrapower by $U_0$ followed
by the ultrapower by $j_0(U_1)$, or as the (external) composition $j_0\compose j_1$. That is, $j_{U_0\times U_1}=j_{j_0(U_1)}\compose
j_0=j_0\compose j_1$. In the general setting of arbitrary complete Boolean algebras, however, since $U_0\times U_1$ may not be an
ultrafilter, we have a somewhat more complex picture:

\begin{theorem} If\/ $U$ is an ultrafilter on $\B_0\times\B_1$ extending $U_0\times U_1$, then there are embeddings $k_0$ and $k^*$ making the
following diagram commute, where $j_0:V\to V_0\of V_0[G_0]$ and $j_1:V\to V_1\of V_1[G_1]$ are the Boolean ultrapowers by $U_0$ and $U_1$,
respectively, $j_1^*:V_0\to V^*\of V^*[G^*_1]$ is the ultrapower by $j_0(U_1)$ as computed in $V_0$ and $j_U:V\to\Vbar\of\Vbar[\bar
G_0\times\bar G_1]$ is the ultrapower by $U$.
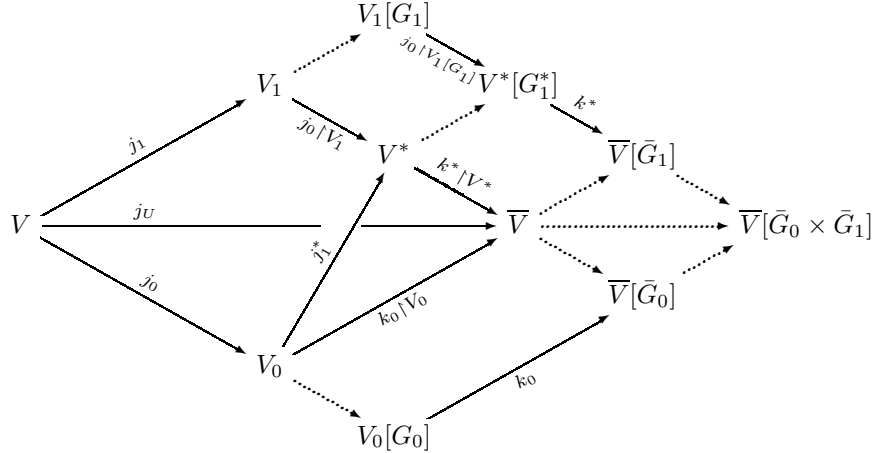
\begin{diagram}[height=1.3em,width=2.3em,labelstyle=\scriptstyle]
   &  &                 &                      &        &           & V_1[G_1]                 &  &        &           &                 &             &           \\
   &  &                 &                      &        & \ruExtension &                       & \rdTo_{\!\!\!\!\!\!\scriptscriptstyle j_0\restrict V_1[G_1]} &        &           &                 &             &           \\
   &  &                 &                      &   V_1  &           &                          &  & V^*[G_1^*]       &           &                 &             &           \\
   &  &                 & \ruTo(4,4)^{j_1} &        & \rdTo_{j_0\restrict V_1}  &     & \ruExtension &        & \rdTo^{k^*}  &                 &             &           \\
   &  &                 &                      &        &           & V^*  &  &        &           & \Vbar[\bar G_1] &             &           \\
   &  &                 &                      &        &  \ruTo(2,6)^{j^*_1}  &     & \rdTo^{k^*\restrict V^*}  &        & \ruExtension &                 & \rdExtension   &           \\
 V &                     &       \hLine^{j_U}           &  &   &  \VonH\ \       &                  &\rTo    &  \Vbar &    & \rExtension      &        & \hskip2em\rlap{\hskip-2em$\Vbar[\bar G_0\times\bar G_1]$} \\
   & \rdTo(4,4)^{j_0} &                 &  &             &           &                          &\ruTo(4,4)_{k_0\restrict V_0}  &        & \rdExtension &                 & \ruExtension   &           \\
   &  &  &  &                            &           &       &   &           &              & \Vbar[\bar G_0] &             &           \\
   &  &  &  &                            &           &         &    &        & \ruTo(4,4)_{k_0} &                 &             &           \\
   &  &  &  &   V_0                      &           &    &  &        &           &                 &             &           \\
   &  &  &  &                            & \rdExtension &                          &  &        &           &                 &             &           \\
   &  &  &  &                            &           & V_0[G_0]                 &  &        &           &                 &             &           \\
\end{diagram}
\end{theorem}

\begin{proof}
Suppose that $U$ is an ultrafilter on $\B_0\times\B_1$ projecting to $U_0\of\B_0$ and $U_1\of\B_1$ and that $U_0\times U_1\of U$. Let $j_0:V\to V_0\of V_0[G_0]$ and $j_1:V\to V_1\of V_1[G_1]$ be the corresponding Boolean ultrapowers by $U_0$ and $U_1$, and let $j_U:V\to \Vbar\of\Vbar[\bar G_0\times\bar G_1]$ be the Boolean ultrapower by $U$. By theorem \ref{Theorem.ProductForcingBooleanUltrapower}, we have the elementary embedding
$k_0:V_0[G_0]\to\Vbar[\bar G_0]$, defined by $k_0:[\tau]_{U_0}\mapsto [\tau^*]_U$, where $\tau$ is any $\B_0$-name and $\tau^*$ is the
natural translation into a $\B_0\times\B_1$ name via the complete embedding $b\mapsto (b,1)$. Furthermore, that theorem shows
$j_U=k_0\compose j_0$, and so the lower part of the diagram commutes.

Consider $j_0(U_1)$, which is an ultrafilter on $j_0(\B_1)$ in $V_0$. We claim that $k_0\image j_0(U_1)\of\bar G_1$. To see this, suppose
that $x$ is any element of $j_0(U_1)$. Thus, $x=[\dot x]_{U_0}$ for some $\B_0$-name $\dot x$ with $\[\dot x\in \check U_1]^{\B_0}=1$ and
hence $\bigvee_{c\in U_1}\[\dot x=\check c]=1$. Let $X=\set{(b,c)\in\B_0^\plus\times\B_1^\plus\st b\forces \check c\leq\dot x}$. This is
clearly open, and it is regular, because if $(b,c)\notin X$, then there is some strengthening $(b_1,c_1)\leq (b,c)$ such that $b_1\forces
\check c_1\perp\dot x$, which means the cone below $(b_1,c_1)$ avoids $X$, establishing that $X$ is regular. If $b\forces\dot x=\check c$,
then clearly $\vee X_b=c$. Since there is a maximal antichain of $b$ forcing that $\dot x=\check c$ for some $c\in U_1$, it follows that
$\vee\set{b\st\vee X_b\in U_1}\in U_0$ and so $X\in U_0\times U_1$, and consequently $X\in U$. It follows that $j(X)$ contains an element
$(b^*,c^*)\in \bar G_0\times\bar G_1$, with the consequence that $b^*$ forces that $\check c^*\leq j(\dot x)$, and so $\val(j(\dot x),\bar
G_0)\in \bar G_1$. But $\val(j(\dot x),\bar G_0)=k_0(\val(j_0(\dot x),G_0))=k_0([\dot x]_{U_0})=k_0(x)$, establishing that $k_0\image
j_0(U_1)\of\bar G_1$, as we claimed.

Let $j^*_1:V\to V^*\of V^*[G^*_1]$ be the Boolean ultrapower as computed in $V_0$ using the ultrafilter $j_0(U_1)\of j_0(\B_1)$. Every
element of $V^*[G^*]$ has the form $[\tau]_{j_0(U_1)}$ (as a set in $V_0$), where $\tau$ is a $j_0(\B_1)$ name in $V_0$. Let
$k^*:[\tau]_{j_0(U_1)}\mapsto \val(k_0(\tau),\bar G_1)$. This is well defined and elementary precisely because $k_0\image j_0(U_1)\of \bar
G_1$. Furthermore, $k^*$ carries $V^*$ to $\Vbar$, since if $\tau$ names an element of the ground model in $V_0$, then $k_0(\tau)$ has this
feature in $\Vbar$. Note that for $x\in V_0$, then $j_1^*(x)=[\check x]_{j_0(U_1)}$, using the $j_0(\B_1)$-name $\check x$ for $x$, and
this object maps by $k^*$ to $\val(k_0(\check x),\bar G_1)=k_0(x)$. This shows that $k^*\compose j_1^*=k_0\restrict V_0$, and so this part
of the diagram commutes.

Finally, observe that by applying $j_0$ to the entire embedding $j_1:V\to V_1\of V_1[G_1]$, which is the Boolean ultrapower of $V$ by
$U_1\of\B_1$, we get by elementarity the Boolean ultrapower of $V_0$ by $j_0(U_1)\of j_0(\B_1)$ as computed in $V_0$, which is exactly how
we defined $j^*:V_0\to V^*\of V^*[G^*_1]$. In short, $j^*=j_0(j_1)$, $V^*=j_0(V_1)$ and $V^*[G_1^*]=j_0(V_1[G_1])$. It follows that if
$j_1(x)=y$, then $j_0(j_1)(j_0(x))=j_0(y)$, which is to say, $j^*\compose j_0=j_0\compose j_1$, and so the entire diagram
commutes.\end{proof}

%
Let us now turn to the possibility that either $U_0\sqtimes U_1$ or $U_0\times U_1$ does not generate an ultrafilter on $\B_0\times\B_1$. We
begin with a necessary and sufficient condition for $U_0\times U_1$ to generate an ultrafilter.

\begin{theorem}\label{Theorem.U0oxU1UltrafilterIff}
Suppose that $U_0\of\B_0$ and $U_1\of\B_1$ are ultrafilters in $V$ on complete Boolean algebras in $V$. Then the following are equivalent:
\begin{enumerate}
 \item $U_0\times U_1$ is an ultrafilter in $\B_0\times\B_1$.
 \item $\[\check U_1\text{ generates an ultrafilter in }\RO(\check\B_1)]^{\B_0}\in U_0$.
\end{enumerate}
And in this case, the Boolean ultrapower by $U_0\times U_1$ is isomorphic to the Boolean ultrapower by $U_0*\<j_{U_0}(U_1)>$ on the
iteration algebra $\B_0*\RO(\check\B_1)$.
\end{theorem}

\begin{proof} ($1\implies 2$) We prove the contrapositive. If
$(2)$ fails, then there is a $\B_0$-name $\dot Y$ such that $\[\dot Y\of\check\B_1\text{ is regular open, but }\dot Y,\neg\dot Y\text{
contain no elements of }\check U_1]^{\B_0}\in U_0$, where $\neg\dot Y$ denotes the negation of $\dot Y$ in the sense of $\RO(\check\B_1)$.
Let $X=\set{(b,c)\st b\forces\check c\in\dot Y}$. Since $\dot Y$ names an open set, $X$ is also open. Observe that if $(b,c)\notin X$, then
$b\not\forces\check c\in\dot Y$, and so there is some $b_0\leq b$ and $c_0\leq c$ such that $b_0\forces\check c_0\in\neg\dot Y$. It follows
that the entire cone below $(b_0,c_0)$ in $\B_0^\plus\sqtimes\B_1^\plus$ avoids $X$. So we have proved that every neighborhood of any point
not in $X$ contains a neighborhood of the complement of $X$, and so $X$ is regular open. Thus, $X\in\B_0\times\B_1$. For any $b\in\B_0$,
since $X_b$ is a regular open subset of $\B_1$, which is complete, it follows that $X_b$ is simply the cone below $\vee X_b$ in $\B_1$.
Since $b\forces \check X_b\of\dot Y$, it follows that $\vee X_b\notin U_1$, for otherwise $b$ would force that $\dot Y$ was in the filter
generated by $\check U_1$, contrary to our assumption. The set $\set{b\st \vee X_b\in U_1}$, therefore, is empty, and consequently $X\notin
U_0\times U_1$. Now consider the negation $\neg X=\set{(b,c)\st \text{the cone below }(b,c)\text{ avoids }X}=\set{(b,c)\st b\forces \check
c\in\neg\dot Y}$, which is the interior of the complement of $X$. Since $b\forces(\neg\check X)_b\of\neg\dot Y$, it follows again that
$\bigvee(\neg X)_b\notin U_1$, since otherwise $b$ would force that $\neg\dot Y$ was in the filter generated by $\check U_1$, contrary to
our assumption. The set $\set{b\in\B_0\st \bigvee(\neg X)_b\in U_1}$ is therefore also empty, and consequently $\neg X\notin U_0\times U_1$.
So $U_0\times U_1$ is not an ultrafilter.

($2\implies 1$) Suppose that $2$ holds. Let $j_0:V\to V_0\of V_0[G_0]$ be the Boolean ultrapower by $U_0$. Applying $j_0$ to statement $2$,
it follows that $j_0(U_1)$ generates an ultrafilter $U_1^*\of\RO(j_0(\B_1))$ in $V_0[G_0]$, with corresponding Boolean ultrapower $j_1$,
giving rise to the following commutative diagram, as in theorem \ref{Theorem.IterationsOfBooleanUltrapowers}, where $j$ is the ultrapower
by the ultrafilter $U_0*U_1^*$ on the iteration algebra $\B_0*\RO(\check\B_1)$.
\begin{diagram}[height=2em,textflow]
 V & & \rTo^{j} & & \Vbar  & \of & \Vbar[\bar G_0] & \hskip-1em\of & \Vbar[\bar G_0*\bar G_1] \\
 &\rdTo_{j_0}&  & \ruTo^{j_1\restrict V_0}    & & \ruTo_{j_1} \\
 & & V_0 & \of & V_0[G_0]\\
\end{diagram}
The filters $\bar G_0*\bar G_1$ are $\Vbar$-generic for the forcing $j(\B_0*\RO(\check\B_1))$. But this forcing is equivalent to the
product forcing $j(\B_0^\plus\sqtimes\B_1^\plus)$, so by genericity $\Vbar[\bar G_0*\bar G_1]$ is the same as $\Vbar[\bar G_0\times\bar G_1]$, with $\bar G_0\times\bar G_1\of j(\B_0^\plus\times\B_1^\plus)$. Suppose that $X\in B_0\times\B_1$, a regular open subset of
$\B_0^\plus\sqtimes\B_1^\plus$. We now observe that
$$\begin{array}{rcl}
 X\in U_0\times U_1 &\iff& \vee\set{b\st \vee X_b\in U_1}\in U_0            \\
                     &\iff& j_0(\vee\set{b\st \vee X_b\in U_1})\in G_0       \\
                     &\iff& \exists b^*\in G_0\intersect j_0(\set{b\st \vee X_b\in U_1})        \\
                     &\iff& \exists b^*\in G_0\text{ such that }\vee j_0(X)_{b^*}\in j_0(U_1)    \\
                     &\iff& \exists b^*\in G_0\text{ such that }j_1(\vee j_0(X)_{b^*})=\vee j(X)_{j_1(b^*)}\in\bar G_1  \\
                     &\iff& \exists b^*\in G_0, c^*\in\bar G_1\text{ such that }c^*\in j(X)_{j_1(b^*)}        \\
                     &\iff& \exists b^*\in G_0, c^*\in\bar G_1\text{ such that }(j_1(b^*),c^*)\in j(X)       \\
\end{array}$$
Since $j_1(G_0)=\bar G_0$, this means that $U_0\times U_1$ is precisely the pre-image by $j$ of the filter generated by $\bar G_0\times\bar
G_1=\bar G_0\sqtimes\bar G_1$ in $j(\B_0\times\B_1)$. Since this latter filter is $\Vbar$-generic, it is an ultrafilter, and so $U_0\times U_1$ is the pre-image of an
ultrafilter, and hence also an ultrafilter, as desired.

Because we know that the product forcing $\B_0\times\B_1$ is forcing equivalent to the iteration $\B_0*\RO(\check\B_1)$, it follows that
these complete Boolean algebras are canonically isomorphic. What we have argued above in essence is that hypothesis $2$ ensures that the
image of $U_0\times U_1$ under this isomorphism is the same as $U_0*U_1^*$. Thus, the Boolean ultrapower by $U_0\times U_1$ is isomorphic to
the Boolean ultrapower by $U_0*U_1^*$.\end{proof}

%
It follows quite generally that $U_0\times U_1$ need not be an ultrafilter.

\begin{corollary}\label{Corollary.U0oxU1NotAnUltrafilter}
For any atomless complete Boolean algebra $\B_1$ there is a complete Boolean algebra $\B_0$ such that for any ultrafilters $U_0\of\B_0$ and
$U_1\of\B_1$ in $V$, the product filter $U_0\times U_1$ on $\B_0\times\B_1$ is not an ultrafilter.
\end{corollary}

\begin{proof} Suppose that $\B_1$ is a given atomless complete
Boolean algebra. Let $\B_0$ be any complete Boolean algebra necessarily forcing that $\B_1$ becomes countable (this is much more than
required). Suppose that $U_0\of\B_0$ and $U_1\of\B_1$ are ultrafilters in $V$. In any forcing extension $V[H]$ by a $V$-generic filter
$H\of\B_0$, the filter $U_1$ becomes countable. In particular, the filter in $\RO(\B_1)^{V[G_0]}$ generated by $U_1$ is generated by a
countable set, and so by lemma \ref{Lemma.NoUltrafilterIsLinearlyGenerated}, it is not an ultrafilter. In short, $U_1$ does not generate an
ultrafilter in $\RO(\B_1)$ in $V[H]$. By theorem \ref{Theorem.U0oxU1UltrafilterIff}, it follows that $U_0\times U_1$ is not an ultrafilter.
%
\end{proof}

\begin{theorem}\label{Theorem.U0xU1GeneratesU0otimesU1Iff}
Suppose that $U_0\of\B_0$ and $U_1\of\B_1$ are ultrafilters in $V$ on complete Boolean algebras in $V$, and that $j_0:V\to V_0$ is the
Boolean ultrapower by $U_0$. Then the following are equivalent:
\begin{enumerate}
 \item $U_0\sqtimes U_1$ generates the product filter $U_0\times U_1$ in $\B_0\times\B_1$.
 \item $j_0\image U_1$ generates the ultrafilter $j_0(U_1)$ in $V_0$.
\end{enumerate}
\end{theorem}

\begin{proof} $(1\implies 2)$ Suppose that $U_0\sqtimes U_1$
generates $U_0\times U_1$, and consider any $x\in j_0(U_1)$. This object is represented by a $\B_0$-name $\dot x$ such that $x=[\dot
x]_{U_0}$ and $\[\dot x\in\check U_1]^{\B_0}=1$. Let $X=\set{(b,c)\in\B_0^\plus\times\B_1^\plus\st b\forces \check c\leq\dot x}$. This is
clearly open, and it is regular because if $(b,c)\notin X$, then $b\not\forces\check c\leq\dot x$, so there are $b'\leq b$ and $c'\leq c$
such that $b'\forces \check c\perp\dot x$, which means the cone below $(b',c')$ avoids $X$. For each $c\in U_1$, let $b_c=\[\dot x=\check
c]^{\B_0}$, and observe that $\bigvee_{c\in U_1}b_c=\[\dot x\in \check U_1]=1$. Note that $(b_c,c)\in X$ for every $c$ for which $b_c\neq
0$. In particular, $c\leq \vee X_{b_c}\in U_1$ in this case, and so $\vee\set{b\st \vee X_b\in U_1}\geq\bigvee_{c\in U_1}b_c=1$. Thus,
$X\in U_0\times U_1$. By $1$, this means that this is some $b\in U_0$ and $c\in U_1$ with $(b,c)\in X$, and so $b\forces \check c\leq\dot
x$. This implies $[\check c]_{U_0}\leq [\dot x]_{U_0}$ in $[\check\B_1]_{U_0}$, or in other words, $j_0(c)\leq x$ in $j_0(\B_1)$. So
$j_0\image U_1$ generates $j_0(U_1)$, as desired.

$(2\implies 1)$ Conversely, suppose that $j_0\image U_1$ generates $j_0(U_1)$ in $j_0(\B_1)$ in $V_0$. Consider any $X\in U_0\times U_1$, so
that $X$ is a regular open subset of $\B_0^\plus\times\B_1^\plus$ and $\vee\set{b\st \vee X_b\in U_1}\in U_0$. Since this outer join is in
$U_0$, it follows that $j_0(\vee\set{b\st \vee X_b\in U_1})\in G_0$. Since $G_0$ is $V_0$-generic for $j_0(\B_0)$, there is some $b^*\in
G_0$ with $\vee j_0(X)_{b^*}\in j_0(U_1)$. By $2$, there is some $c\in U_1$ with $j_0(c)\leq\vee j_0(X)_{b^*}$. Thus, $(b^*,j_0(c))\in
j_0(X)$, and so $b^*\in j_0(X^c)$, slicing now horizontally. Thus, $j_0(\vee X^c)\in G_0$ and so $\vee X^c\in U_0$. Let $b=\vee X^c$. Since
$X$ is regular, it follows that $b\in X^c$ and so $(b,c)\in X$. So $X$ contains an element of $U_0\sqtimes U_1$, as desired.\end{proof}

The following theorem is an analogue of mutual genericity of generic filters, although the statements in the theorem can hold for
non-generic filters. Note that although the previous theorem only required $j_0\image U_1$ to generate an ultrafilter in $V_0$, here we ask
that it generate an ultrafilter on the completion of $j_0(\B_1)$ in $V_0[G_0]$. In the below, $U_0\rtimes U_1$ refers to the dual product,
defined by $X\in U_0\rtimes U_1$ if and only if $\vee\set{c\st \vee X^c\in U_0}\in U_1$, where $X^c=\set{b\st (b,c)\in X}$. This amounts
merely to taking the factors in the other order, and so $U_0\rtimes U_1\cong U_1\times U_0$.

\begin{theorem} Suppose that $U_0\of\B_0$ and $U_1\of\B_1$ are ultrafilters in $V$ on complete Boolean algebras in $V$, with corresponding
Boolean ultrapowers $j_0:V\to V_0\of V_0[G_0]=V^{\B_0}/U_0$ and $j_1:V\to V_1\of V_1[G_1]=V^{\B_1}/U_1$. Then the following are equivalent:
\begin{enumerate}
 \item $U_0\cross U_1$ is an ultrafilter in $\B_0\times\B_1$.
 \item $j_0\image U_1$ generates an ultrafilter on $\RO(j_0(\B_1))$ in $V_0[G_0]$.
 \item $j_1\image U_0$ generates an ultrafilter on $\RO(j_1(\B_0))$ in $V_1[G_1]$.
\end{enumerate}
And in this case, the ultrafilter generated by $U_0\sqtimes U_1$ is precisely $U_0\times U_1$, which in this case is equal to $U_0\rtimes U_1$.
\end{theorem}

\begin{proof} We argue first that ($1\iff 2$). Since $U_0\sqtimes
U_1$ is contained within the filter $U_0\times U_1$, it is clear that $U_0\sqtimes U_1$ generates an ultrafilter on $\B_0\times\B_1$ if and
only if $U_0\sqtimes U_1$ generates $U_0\times U_1$ and $U_0\times U_1$ generates an ultrafilter on $\B_0\times\B_1$. By theorem
\ref{Theorem.U0xU1GeneratesU0otimesU1Iff}, the former requirement is equivalent to the assertion that $j_0\image U_1$ generates $j_0(U_1)$
on $j_0(\B_1)$ in $V_0$. By theorem \ref{Theorem.U0oxU1UltrafilterIff}, the latter requirement is equivalent to the assertion that
$j_0(U_1)$ in turn generates an ultrafilter on $\RO(j_0(\B_1))$ in $V_0[G_0]$. Altogether, therefore, we see that $U_0\sqtimes U_1$ generates
an ultrafilter if and only if $j_0\image U_1$ generates an ultrafilter on $\RO(j_0(\B_1))$ in $V_0[G_0]$, establishing ($1\iff 2$). Because
statement $1$ is respected by the natural automorphism of $\B_0\times\B_1$ with $\B_1\times\B_0$, the dual equivalence ($1\iff 3$) follows
simply by swapping the order of the factors. Finally, if $U_0\sqtimes U_1$ generates an ultrafilter on $\B_0\times\B_1$, then because it is
included in the filters $U_0\times U_1$ and $U_0\rtimes U_1$, these must all be the same ultrafilter.\end{proof}

An equivalent formulation of (2) asserts that whenever $\dot X$ is a $\B_0$-name for which $\[\dot X\of\check\B_1\text{ is regular
open}]^{\B_0}=1$, then there is some $c\in U_1$ with $\[\check c\in\dot X]^{\B_0}\in U_0$ or $\[\check c\in\neg\dot X]^{\B_0}\in U_0$.

Let us now give an analysis of whether $U_0\sqtimes U_1$ generates an ultrafilter in terms of the descent spectrums of $U_0$ and $U_1$.

\begin{theorem}\label{Theorem.U0xU1UltrafilterIffDisjointDescentSpectrum}
If\/ $U_0\of\B_0$ and $U_1\of\B_1$ are ultrafilters on complete Boolean algebras whose descent spectrums have a common order type $\kappa$,
then $U_0\sqtimes U_1$ does not generate an ultrafilter on $\B_0\cross\B_1$. Indeed, it does not even generate $U_0\times U_1$.
\end{theorem}

\begin{proof} Suppose that $U_0$ and $U_1$ admit descents
$\<b_\alpha\st\alpha<\kappa>$ and $\<c_\alpha\st\alpha<\kappa>$, respectively, of common order type $\kappa$. By moving to cofinal
subsequences, we may assume that $\kappa$ is regular and that these are strict descents. As in theorem
\ref{Theorem.DescentIffDiscontinuityPoint}, the difference elements $d_\alpha=b_\alpha-b_{\alpha+1}$ and $e_\alpha=c_\alpha-c_{\alpha+1}$
form maximal antichains $D=\set{d_\alpha\st\alpha<\kappa}\of\B_0$ and $E=\set{e_\alpha\st\alpha<\kappa}\of\B_1$. It is easy to see that
$D\cross E=\set{(d_\alpha,e_\beta)\st \alpha,\beta<\kappa}$ is a maximal antichain in $\B_0^\plus\cross\B_1^\plus$. In the completion
$\B_0\times\B_1$, let $x=\vee\set{(d_\alpha,e_\beta)\st \alpha<\beta}$, which corresponds to the upper left triangle. Since $D\cross E$ was
a maximal antichain, it follows that $\neg x=\vee\set{(d_\alpha,e_\beta)\st \alpha\geq\beta}$, the lower right triangle. We claim that
neither $x$ nor $\neg x$ are in the filter generated by $U_0\sqtimes U_1$. Note that every $b\in U_0$ has nonzero meet with every
$b_\beta=\bigvee_{\alpha\geq\beta}d_\alpha$, and consequently has nonzero meet with unboundedly many difference elements $d_\alpha$.
Similarly, every $c\in U_1$ has nonzero meet with unboundedly many difference elements $e_\beta$. Thus, every pair $(b,c)\in U_0\cross U_1$
has nonzero meet with both $x$ and $\neg x$. Consequently, the filter generated by $U_0\cross U_1$ in $\B_0\times\B_1$ does not decide
$\singleton{x,\neg x}$, and so it is not an ultrafilter. To prove the final claim of the theorem, we observe that for every $\alpha$, the
join $\vee\set{e_\beta\st\alpha<\beta}$ is precisely $c_{\alpha+1}$, which is $U_1$, meaning that all the vertical slices in the set
representing $x$ are large, and so $x\in U_0\times U_1$.\end{proof}

\begin{corollary}\label{Corollary.U0xU1NotAnUltrafilter}
For any pair $\B_0$ and $\B_1$ of infinite complete Boolean algebras, there are ultrafilters $U_0\of\B_0$ and $U_1\of\B_1$ such that
$U_0\sqtimes U_1$ does not generate an ultrafilter on $\B_0\times\B_1$.
\end{corollary}

\begin{proof} If $\B_0$ and $\B_1$ are infinite, then there are
countable maximal antichains $\set{a_n\st n<\omega}\of\B_0$ and $\set{b_n\st n<\omega}\of\B_1$, from which it is easy to construct
countable descending sequences $c_m=\bigvee_{n\geq m}a_n$ and $d_m=\bigvee_{n\geq m}b_n$ with $\bigwedge_m c_m=0$ and $\bigwedge_m d_m=0$.
Any ultrafilters $U_0$ and $U_1$ with $c_m\in U_0$ and $d_m\in U_1$ for all $m<\omega$ will therefore have descents of common order type
$\omega$. By theorem \ref{Theorem.U0xU1UltrafilterIffDisjointDescentSpectrum}, therefore, $U_0\sqtimes U_1$ does not generate an ultrafilter
in $\B_0\times \B_1$.\end{proof}

We have a partial converse:

\begin{theorem} If the cardinals in the descent spectrum of\/ $U_0$ all lie completely below the descent spectrum of\/ $U_1$, then $U_0\sqtimes
U_1$ generates $U_0\times U_1$. 
\end{theorem}

\begin{proof} Suppose that the cardinals in the descent spectrum
of $U_0$ all lie completely below the descent spectrum of $U_1$, and let $j_0:V\to V_0$ be the Boolean ultrapower by $U_0$. Consider any
element $x$ in $j_0(U_1)\of j_0(\B_1)$. By lemma \ref{Lemma.RepresentationOnDifferenceAntichain}, we may represent $x=[f]_{U_0}$ using a
spanning function $f:A\to V$ on a maximal antichain $A$, whose cardinality is in the descent spectrum of $U_0$. We may assume $f(a)\in U_1$
for all $a\in A$. Since $|A|$ lies below any member of the descent spectrum of $U_1$, it follows that $U_1$ is sufficiently complete that
$c=\bigwedge_{a\in A}f(a)$ is in $U_1$. Thus, $j_0(c)\leq [f]_{U_0}=x$, and so we have established that $j_0\image U_1$ generates
$j_0(U_1)$ in $V_0$. By theorem \ref{Theorem.U0xU1GeneratesU0otimesU1Iff}, this means that $U_0\sqtimes U_1$ generates $U_0\times
U_1$.\end{proof}

The question of whether or not $U_0\sqtimes U_1$ generates an ultrafilter, or whether or not it generates $U_0\times
U_1$, appears to be related to the descent spectrums.

\section{Ideals}\label{Section.Ideals}

In this section, we generalize the concept of precipitous ideal and other notions from the case of power sets to arbitrary complete Boolean
algebras. Suppose that $\B$ is a complete Boolean algebra. An {\df ideal} on $\B$ is a subset $I\subset\B$ such that $0\in I$, $1\notin I$,
closed under join $a, b\in I\implies a\vee b \in I$ and lower cones $a\leq b\in I \implies a\in I$. To each ideal there is an associated
equivalence relation on $\B$, where $a =_I b$ if and only if $a\vartriangle b\in I$, where $a\vartriangle b = (a-b)+(b-a)$ is the symmetric
difference. The induced order $a \leq_I b \iff a-b\in I$ is well defined on the equivalence classes $[a]_I=\set{b\st a=_I b}$, and the
quotient $\B/I=\set{[a]_I\st a\in \B}$ is a Boolean algebra, although it is not necessarily complete. If $F$ is a filter on $\B/I$, then
the corresponding {\df induced filter} on $\B$ is $\union F =\set{ b\in B\st [b]_I\in F }$ on $\B$. It is easy to the see that the induced
filter of an ultrafilter is also an ultrafilter.

\begin{definition}\rm
An ideal $I$ on a complete Boolean algebra $\B$ is {\df precipitous} if whenever $G\subset \B/I$ is a $V$-generic filter, then the Boolean
ultrapower $\check V_U$ by the induced ultrafilter $U=\union G$ on $\B$ is well-founded.
\end{definition}

This definition directly generalizes the usual Jech-Prikry definition of precipitous ideal in the case that $\B$ is a power set Boolean
algebra. (The word {\df precipitous} was chosen by Jech, in honor of Prikry, whose name reportedly has that meaning in Hungarian.)
Technically, the trivial ideal $I=\singleton{0}$ is precipitous, since if $G\of\B/I$ is $V$-generic, then $U=\union G\of\B$ is also
$V$-generic, and so $\check V_U$ is isomorphic to $V$, which is well-founded. Consequently, we are primarily interested in the case of {\df
nontrivial} ideals $I$, meaning that $I$ contains a maximal antichain from $\B$, or in other words, that $\vee I=1$, since it is exactly
this condition which ensures that the resulting filter $U=\Union G$ is not $V$-generic and consequently has a nontrivial Boolean ultrapower
$j_U:V\to\check V_U$. This requirement therefore generalizes the concept of a {\df non-principal} ideal from the power set context.

If $\B$ is a complete Boolean algebra and $A\of\B$ is a maximal antichain of size $\kappa$, then the {\df small ideal} relative to $A$ is
the ideal of all elements of $\B$ below the join of a small subset of $A$, that is, the ideal $I=\set{b\in\B\st\exists A_0\of A\, |A_0|<\kappa\text{
and }b\leq\vee A_0}$.

\begin{theorem}\label{Theorem.CriticalPointKappaIllfounded}
The small ideal, relative to any maximal antichain $A\of\B$ of infinite regular size $\kappa$, is not precipitous. If\/ $G\of\B/I$ is
$V$-generic, then the Boolean ultrapower $j_U$ by the induced ultrafilter $U=\union G$ is well-founded up to $\kappa$, but ill-founded
below $j_U(\kappa)$.
\end{theorem}

\begin{proof} In theorem \ref{Theorem.CriticalPointKappa}, we
already established that the induced ultrafilter $U=\union G$ meets all small maximal antichains in $\B$ in $V$, but not $A$, and so the
Boolean ultrapower map $j_U:V\to \check V_U$ has critical point $\kappa$. The ultrapower $\check V_U$ is consequently well-founded up to
$\kappa$.

We now show that $\check V_U$ is necessarily ill-founded below $j_U(\kappa)$. To do so, we adapt the proof from the classical power set
context. Enumerate $A=\set{a_\alpha\st\alpha<\kappa}$ in a one-to-one manner. For every $X\of\kappa$ of size $\kappa$, let
$a_X=\bigvee_{\alpha\in X}a_\alpha$ and define $f_X:A\to\kappa$ by $f_X(a_\alpha)=\ot(X\intersect\alpha)$. Observe that $f_X$ is a spanning
function on domain $A$, with value less than $\kappa$ at every point, but constant only on small sets. Thus, $[f_X]_U$ will represent an
ordinal less than $j_U(\kappa)$ in the (functional presentation of the) Boolean ultrapower, but not any ordinal less than $\kappa$. If $Y$
consists of the successor elements of $X$, in the canonical enumeration of $X$, then it is easy to see that $f_Y(a_\alpha)<f_X(a_\alpha)$
for any $\alpha\in Y$. Thus, provided that $a_Y\in U$, we will have $[f_Y]_U<[f_X]_U$. This is one step of how we shall build our
descending sequence in the Boolean ultrapower. Specifically, we build a sequence of $\B/I$-names $\dot f_n$ for spanning functions
$f_n:A\to \kappa$ in $V$. We will specify a sequence of maximal $I$-partitions $\mathcal{A}\of I^\plus$, with every element of
$\mathcal{A}$ having the form $a_X$ for some $X\of\kappa$ of size $\kappa$. Associated to each such $a_X$ have the spanning function
$f_X:A\to\kappa$, and by mixing these objects along the partition, we arrive at a $\B/I$-name $\dot f$ for a spanning function on $A$ in
$V$, with $a_X$ forcing $\dot f=\check f_X$. We begin with $\mathcal{A}_0=\singleton{1}$, resulting in the name $\dot f_0$ for the spanning
function $f_0:a_\alpha\mapsto\alpha$. Given $\dot f_n$ and $\mathcal{A}_n$, let us say that $Y\of\kappa$ of size $\kappa$ is {\df good}, if
there is some $X\of\kappa$ with $a_X\in \mathcal{A}_n$ and for all $\alpha\in Y$, $f_Y(a_\alpha)<f_X(a_\alpha)$. The argument we gave above
established that the collection of good $Y\of\kappa$ is dense in $P(\kappa)$ modulo the ideal of small sets. By using the canonical
complete embedding of $P(\kappa)$ into $\B$ via $P(A)$, this means that the set $D=\set{a_Y\st Y\of\kappa\text{ is good}}$ is predense in
$I^\plus$. Thus, we may select a maximal $I$-partition $\mathcal{A}_{n+1}\of D$, each having the form $a_Y$ for a good $Y\of\kappa$. By
mixing the corresponding $\check f_Y$, this gives rise to the desired name $\dot f_{n+1}$, with $a_Y$ forcing $\dot f_{n+1}=\check f_Y$ for
each $a_Y\in \mathcal{A}_{n+1}$.

Assume now that $G\of\B/I$ is $V$-generic, and let $U=\Union G\of\B$ be the corresponding ultrafilter on $\B$. For each $n<\omega$, the
filter $G$ selects a unique element $a_{X_n}$ from $\mathcal{A}_n$, which therefore decides $f_n=\val(\dot f_n,G)=f_{X_n}$. By
construction, consequently, $[f_{n+1}]_U<[f_n]_U$ in the Boolean ultrapower, and so the Boolean ultrapower is ill-founded below
$j_U(\kappa)$, as desired.\end{proof}

\begin{corollary}If one allows $U\notin V$, then statement $1$
of theorem \ref{Theorem.WellFoundedEquivalents} is no longer necessarily equivalent to statements $2$ through $5$.
\end{corollary}

What we did in the previous argument essentially was to adapt the classical proof that the ideal $I_0$ of small subsets of $\kappa$ is not
precipitous. In corollary \ref{Corollary.GenericEmbeddingBooleanUltrapowerCorrespondence}, we will show that all the classical generic
ultrapowers obtained by forcing over the quotient of a power set $P(\kappa)/I$ can also be obtained as generic Boolean ultrapowers obtained
by forcing over the quotient of $\B/I^*$, for any complete Boolean algebra having an antichain of size $\kappa$. Because of this, the
generic Boolean ultrapowers we consider in this section generalize the classical generic embeddings.

We define that a quotient $\B/I$ of a complete Boolean algebra $\B$ by an ideal $I$ has the {\df disjointifying property} if for every
maximal antichain $A\of\B/I$, there is a choice of representatives $b_a\in\B$ such that $a=[b_a]_I$ for each $a\in A$ and $\set{b_a\st a\in
A}$ forms a maximal antichain in $\B$. For such a disjointification, it is not difficult to see that if $A_0\of A$, then $\vee
A_0=[\bigvee_{a\in A_0} b_a]_I$. That is, the join in $\B/I$ of a subset of $A$ is simply the equivalence class of the join in $\B$ of the
corresponding collection of representatives (without disjointification and maximality, this is not typically the case). Consequently, the
quotient $\B/I$ is a complete Boolean algebra. In the classical power set context, the disjointifying property is a central concept, a
powerful intermediary between saturation and precipitousness. Here, we prove the analogous relations in the general context of complete
Boolean algebras.

\begin{theorem}\label{Theorem.QuotientByDisjointifyingIdealsFactor}
Suppose that $U_0$ is an ultrafilter on the quotient $\B/I$ of a complete Boolean algebra $\B$ by an ideal $I$ with the disjointifying
property. Then the Boolean ultrapower by the induced ultrafilter $U=\union U_0$ on $\B$ factors through $U_0$ as follows:
\begin{diagram}[height=2em]
 V & & \rTo^{j_U} & & \check V_U  \\
 &\rdTo_{j_{U_0}}&  & \ruTo^{k} \\
 & & \check V_{U_0} & \\
\end{diagram}
\end{theorem}

\begin{proof} We use the functional representation of $\check
V_{U_0}$ as $V^{\downarrow(\B/I)}_{U_0}$. The typical element of $V^{\downarrow(\B/I)}_{U_0}$ is $[f]_{U_0}$, where $f:A\to V$ is a
spanning function on a maximal antichain $A\of\B/I$. By the disjointifying property, we may select representatives $b_a\in\B$ for each
$a=[b_a]_I\in A$ such that $A^*=\set{b_a\st a\in A}$ forms a maximal antichain in $\B$. Let $f^*:A^*\to V$ be the corresponding spanning
function on $A^*$, with $f^*(b_a)=f(a)$. We define $k:[f]_{U_0}\mapsto [f^*]_U$.

We argue first that $k$ is well defined. To begin, we check that the particular choice of disjoint representatives $b_a$ does not matter.
If $c_a$ is some other way of choosing disjoint representatives, leading to the maximal antichain $A^{**}=\set{c_a\st a\in A}$ in $\B$,
then we compare the function $f^*:b_a\mapsto f(a)$ with $f^{**}:c_a\mapsto f(a)$. Let $d_a=b_a\wedge c_a$, and observe that
$[d_a]_I=[b_a]_I=[c_a]_I$ and also that the $d_a$ are disjoint for different $a$. If $x$ is disjoint from $d_a$ for all $a\in A$, then
$[x]_I$ would be incompatible in $\B/I$ with all $[d_a]_I$, which is to say, with every element of $A$, and so $x\in I$. It follows that
$\bigvee_{a\in A}d_a\in U$. Extend $\set{d_a\st a\in A}$ to a maximal antichain $D\of\B$ refining both $A^*$ and $A^{**}$. Note that
$(f^*\downarrow D)(d_a)=f(a)=(f^{**}\downarrow D)(d_a)$, and so $\bigvee\set{d\in D\st (f^*\downarrow D)(d)=(f^{**}\downarrow D)(d)}\geq
\bigvee_{a\in A}d_a\in U$, and so $[f^*]_U=[f^{**}]_U$, as desired. Next, we show that $k$ is well defined for spanning functions with the
same domain. Suppose $[f]_{U_0}=[g]_{U_0}$ for two spanning functions $f:A\to V$ and $g:A\to V$ in $\B/I$, both with domain $A$. Thus,
$\vee\set{a\in A\st f(a)=g(a)}\in U_0$. By the observation just before the theorem, this means that $[\vee\set{b_a\st a\in A,
f(a)=g(a)}]_I\in U_0$ and so $\vee\set{b_a\st a\in A, f(a)=g(a)}=\vee\set{b\in A^*\st f^*(b)=g^*(b)}\in U$, which means $[f^*]_U=[g^*]_U$.
Finally, we show that $k$ is well defined for the reduction of a spanning function $f:A\to V$ to a finer domain $B\leq A$. Fix any disjoint
selection of representatives $x_b\in\B$ for $b\in B$ such that $[x_b]=b$, and let $B^*=\set{x_b\st b\in B}$ be the corresponding antichain.
Observe that $y_a=\vee\set{x_b\st b\leq a, b\in B}$ is a disjoint selection of representatives for $A$, with corresponding $A^*=\set{y_a\st
a\in A}$ and corresponding spanning functions $f^*:A^*\to V$ and $f^{**}:B^*\to V$ defined by $f^*(y_a)=f(a)$ and $f^{**}(x_b)=(f\downarrow
B)(b)$. Since $(f\downarrow b)=f(a)$ for the unique element $a\in A$ above $b$, it follows that $(f^*\downarrow B^*)(x_b)=f^{**}(x_b)$ for
every $x_b\in B^*$, which implies that $[f^*]_U=[f^{**}]_U$, as desired. So altogether, $k$ is well defined.

We now show that $k$ is fully elementary by induction on formulas. The atomic case and Boolean connectives are straightforward, as is the
upward direction of the quantifier case. Suppose now that $V^{\downarrow\B}_U\satisfies\exists x\,\varphi(x,[f^*]_U)$, where $f^*:A^*\to V$
arises from the spanning function $f:A\to V$ for a maximal antichain $A\of\B/I$. Define a spanning function $g:A^*\to V$ by $g(b)$ is a
witness such that $V\satisfies\varphi(g(b),f^*(b))$, if there is any such witness. By the \L o\'s Theorem, it follows that
$V^{\downarrow\B}_U\satisfies\varphi([g]_U,[f^*]_U)$, and we have found a witness using the very same antichain $A^*$. Define $g_0:A\to V$
by $g_0(a)=g(b_a)$, where $b_a$ is the unique representative of $a$ in $A^*$. It follows that $g=g_0^*$ and so $[g]_U=k([g_0]_{U_0})$.
Thus, $V^{\downarrow(\B/I)}_{U_0}\satisfies\varphi([g_0]_{U_0},[f]_{U_0})$, as desired. So $k$ is fully elementary.\end{proof}

%
If $I\of\B$ is an ideal, then we say that a subset $A\of I^\plus$ is an {\df (maximal) antichain modulo $I$} if distinct elements of $A$
are not equivalent modulo $I$ and $\set{[a]_I\st a\in A}$ is a (maximal) antichain in the quotient $\B/I$. A {\df tree of maximal
antichains} modulo $I$ in $\B$ is a sequence $\<A_n\st n<\omega>$ of maximal antichains modulo $I$ such that $A_{n+1}$ refines $A_n$ modulo
$I$, meaning that for every element $a\in A_{n+1}$ there is $b\in A_n$ such that $[a]_I\leq [b]_I$ in $\B/I$.
A Boolean algebra $\B$ is {\df countably distributive} if as a forcing notion it adds no new
$\omega$-sequences of ordinals. This is equivalent to the countable distributivity rule: $\bigwedge_n \vee A_n=\vee\set{\bigwedge_n f(n)\st
f:\omega\to\B, f(n)\in A_n}$.

\begin{theorem}\label{Theorem.PreciptiousIFF}
Suppose that $\B$ is a countably distributive complete Boolean algebra and $I\of\B$ is an ideal. Then the following are equivalent:
 \begin{enumerate}
  \item $I$ is precipitous.
  \item For every tree of maximal antichains $A_n\of \B$, every $a_0\in A_0$ can be continued to a sequence $\<a_n\st n<\omega>$ with
      $a_n\in A_n$, $a_{n+1}\leq_I a_n$ and $\bigwedge_n a_n\neq 0$.
 \end{enumerate}
\end{theorem}

\begin{proof} Suppose that $I$ is precipitous. Fix the tree of
maximal antichains $A_n\of\B$, for $n<\omega$, and $a_0\in A_0$. Let $G\of\B/I$ be $V$-generic below $a_0$. Since $I$ is precipitous, this
means that the Boolean ultrapower of $V$ by the corresponding filter $U=\union G\of\B$ is well-founded. Let $j:V\to M=\check V_U$ be the
corresponding ultrapower map into the transitive class $M$. By theorem \ref{Theorem.BooleanUltrapowerWellFounded}, there is in $V[G]$ an
$M$-generic ultrafilter $H\of j(\B)$ with $j\image U\of H$. Since $H$ is $M$-generic, there is a unique condition $b_n\in j(A_n)\intersect
H$. In fact, $b_n$ is exactly the element of $j(\B)$ corresponding to the name $\set{\<\check a,a>\st a\in A_n}$. Since $A_{n+1}$ refines
$A_n$ modulo $I$, it follows that $b_{n+1}\leq_{j(I)} b_n$. The sequence $\<b_n\st n<\omega>$ exists in $M[H]$, and consequently, by
distributivity, in $M$. It follows that $\bigwedge_n b_n\in H$, since $H$ is $M$-generic and hence $M$-complete by the remarks before
theorem \ref{Theorem.WellFoundedEquivalents}. Consequently, $\bigwedge_n b_n\neq 0$. Thus, in $M$ we have found a path as desired through
the tree of maximal antichains $\<j(A_n)\st n<\omega>=j(\<A_n\st n<\omega>)$ in $j(\B)$ with respect to $j(I)$. By elementarity, it follows
that there is such a sequence in $V$ for $\<A_n\st n<\omega>$ in $\B$ with respect to $I$, as desired.

Conversely, suppose that $I$ is not precipitous. Thus, there is a condition $a_0\in I^\plus$ forcing via $\B/I$ that $\check V_{\union G}$
is ill-founded. Thus, there are $\B/I$-names $\dot f_n$ and $\Ddot_n$ such that $a_0$ forces via $\B/I$ that $\dot f_n:\Ddot_n\to\ORD$ are
open dense spanning functions in $V$ and that $\dot f_{n+1} <_{\union \dot G} \dot f_n$. We shall make use of the following lemma.

\begin{sublemma}\label{Lemma.NameForSpanningFunction}
If\/ $I$ is an ideal on a complete Boolean algebra $\B$ and $[b]_I\forces_{\B/I}``\dot f:\dot D\to\check V$ is an open dense spanning
function in $\check V$'', then there is $A\of\B$, a maximal antichain modulo $I$ below $b$, and for each $a\in A$ a function $f_a:\B_a\to
V$, where $\B_a$ is the cone below $a$, such that $[a]_I\forces_{\B/I}\dot f\restrict\check \B_a=\check f_a$.
\end{sublemma}

\begin{proof} Since $b$ forces that $\dot D$ is an open dense
subset of $\B$ in $\check V$, there is a maximal antichain $A$ of conditions $a$ deciding exactly which open dense subset in $V$ it is. By
further refining inside these dense sets, we obtain a maximal antichain $A$ of conditions $a$ forcing that the cone $\B_a$ below $a$ is
contained in $\dot D$. By the same reasoning, since $b$ forces that $\dot f$ is a function on $\dot D$ in $\check V$, we may also assume
that the conditions in $A$ also decide which function from $V$ $\dot f$ is. Thus, we obtain functions $f_a:\B_a\to V$ such that
$[a]_I\forces_{\B/I}\dot f\restrict\check \B_a=\check f_a$, as desired.\end{proof}

We now iteratively apply the lemma to build a tree of maximal antichains $A_n\of\B$ modulo $I$ such that for every $a\in A_n$ there are
functions $f_a^n:\B_a\to\ORD$ on the cone $\B_a$ below $a$, such that $[a]_I\forces_{\B/I}\dot f_n\restrict\check\B_a^n=\check f_a^n$.
Furthermore, we can arrange that every element of $A_{n+1}$ lies below an element of $A_n$, and lastly, since the sequence was forced to be
descending, that $f_a^{n+1}(b)< f_{a'}^n(b)$ for all $b\leq a$, whenever $a\leq a'$ with $a\in A_{n+1}$ and $a'\in A_n$. Suppose now that
there is a path $\<a_n\st n<\omega>$ through this tree of antichains in the sense of (2), so that $a_n\in A_n$ and $z=\wedge_n a_n\neq 0$.
Since $z\leq a_n$ for all $n$, it follows that $z\in \B_{a_n}$ for all $n$, and so our observations ensure
$f_{a_{n+1}}^{n+1}(z)<f_{a_n}^n(z)$. This is an infinite descending sequence of ordinals in $V$, a contradiction. So there can be no such
path as in (2), and the theorem is proved.\end{proof}

We emphasize that we have {\it not} assumed in the theorem above that the quotient forcing $\B/I$ is  countably distributive, but only that
the underlying Boolean algebra $\B$ itself is. This generalizes the classical power set situation, where one forces with the quotient
$P(Z)/I$ of a power set, since the underlying Boolean algebra there is the power set $P(Z)$ itself, which is trivial as a forcing notion,
of course, and consequently fully distributive. Thus, the previous theorem generalizes the classical characterization of precipitous
ideals. The question remains whether we actually need distributivity in theorem \ref{Theorem.PreciptiousIFF}.

%

In the classical power set context, every saturated ideal is disjointifying, and every disjointifying ideal is precipitous. We can easily
generalize these implications to the case of countably distributive complete Boolean algebras.

\begin{theorem} If\/ $I$ is a $\ltkappa$-complete ideal in a complete Boolean algebra $\B$ and $\B/I$ is $\kappa^\plus$-c.c., then $\B/I$ is
a complete Boolean algebra with the disjointifying property.
\end{theorem}

\begin{proof} This can be proved just as in the classical power
set context. If $A\of\B/I$ is any maximal antichain, then by $\kappa^\plus$-saturation, we know $A$ has size at most $\kappa$, so we may
enumerate $A=\set{a_\alpha\st \alpha<|A|}$. Choose any representatives $x_\alpha$ such that $a_\alpha=[x_\alpha]_I$, and then disjointify
them: let $b_\alpha=x_\alpha-\bigvee_{\beta<\alpha}x_\beta$. Clearly, $b_\alpha\leq x_\alpha$ and $b_\alpha\wedge b_\beta=0$ if
$\alpha\neq\beta$. Observe that $x_\alpha-b_\alpha=\bigvee_{\beta<\alpha}(x_\alpha-x_\beta)$, which is a small join of elements of $I$ and
hence in $I$, by the $\ltkappa$-completeness of $I$, and so $[b_\alpha]_I=[x_\alpha]_I=a_\alpha$. Thus, we have found a disjoint selection
of representatives. Since the join of a subset of a maximal antichain in $\B/I$ is the equivalence class of the join of these disjoint
representatives, it follows that $\B/I$ is a complete Boolean algebra.\end{proof}

\begin{theorem} If\/ $I$ is a countably complete ideal with the disjointifying property on a countably distributive complete Boolean algebra
$\B$, then $I$ is precipitous.
\end{theorem}

\begin{proof} We use theorem \ref{Theorem.PreciptiousIFF}.
Suppose that $\<A_n\st n<\omega>$ is a tree of maximal antichains. This means that each $A_n\of\B$ is a maximal antichain modulo $I$, and
every element of $A_{n+1}$ is below an element of $A_n$ modulo $I$. Because $I$ has the disjointifying property, we may systematically
disjointify these antichains, replacing each element with a slightly smaller element but $I$-equivalent member, in such a way that each
$A_n$ becomes an antichain in $\B$, not merely an antichain in $\B$ modulo $I$. Let $F$ be the dual filter to $I$. Since $A_n$ is a maximal
antichain modulo $I$, it follows that $\vee A_n\in F$. Since $I$ is countably complete, it follows that $\bigwedge_n(\vee A_n)\in F$ as
well. By the countable distribution law in $\B$, it follows that $\bigwedge_n(\vee A_n)=\bigvee\set{\bigwedge_n a_n\st
\<a_0,a_1,\ldots>\in\B^\omega, a_n\in A_n}$. In particular, there is some such sequence $\<a_0,a_1,\ldots>$ with $a_n\in A_n$ and $\wedge_n
a_n\neq 0$. Since every element of $A_{n+1}$ lies below an element of $A_n$, and is incompatible with the other elements of $A_n$, it
follows that this sequence is descending, and fulfills the criterion of theorem \ref{Theorem.PreciptiousIFF}, as desired.\end{proof}

\begin{corollary}
If\/ $I$ is a $\ltkappa$-complete ideal in a complete, countably distributive Boolean algebra $\B$ and $\B/I$ is $\kappa^\plus$-c.c., for
some uncountable cardinal $\kappa$, then $I$ is precipitous.
\end{corollary}

Next, we generalize Solovay's observation from the power set context that disjointification allows for a convenient anticipation of names
for spanning functions in the ground model.

\begin{lemma}\label{Lemma.DisjointifyingNameForSpanningFunction} 
Suppose that $\B/I$ has the disjointifying property and that $\dot f$ is a $\B/I$-name for which $b=\[\dot f\text{ is a spanning function in }\check
V]^{\B/I}\neq 0$. Then there is an actual spanning function $g:A\to V$ in $V$ such that $b$ forces that $\dot f$ and $\check g$ are
equivalent modulo $\union\dot G$.
\end{lemma}

\begin{proof} It is somewhat more convenient to use open dense
spanning functions, so let us assume that $\dot f$ is defined not merely on an antichain, but is extended to the open dense set of
conditions below elements of an antichain, by copying the value on the antichain to the entire lower cone of that element. By lemma
\ref{Lemma.NameForSpanningFunction}, we may find a maximal antichain $A\of\B$ modulo $I$ below $b$ such that for each $a\in A$ there is a
function $f_a:\B_a\to V$, where $\B_a$ is the cone of nonzero elements below $a$, such that $[a]_I$ forces that $\dot
f\restrict\check\B_a=f_a$. Since $I$ has the disjointifying property, we may replace the elements of $a$ with smaller elements but still
equivalent modulo $I$, so that $A$ is an antichain in $\B$. In this case, the cones $\B_a$ and $\B_b$ for $a\neq b$ in $A$ do not overlap,
and so $g=\union\set{f_a\st a\in A}$ is a function. Every $a\in A$ has $[a]_I$ forcing that $\dot f\restrict\B_a=\check g\restrict\B_a$,
and since $A$ is a maximal antichain modulo $I$ below $b$, this means that $b$ forces that $\dot f$ is equivalent to $\check g$ modulo
$\union\dot G$, as desired.\end{proof}

The next theorem generalizes theorem \ref{Theorem.ClosureOfBooleanUltrapowers}, by obtaining closure of the Boolean ultrapower in the
quotient forcing extension $V[G]$, rather than merely in $V$ as before.

\begin{theorem}\label{Theorem.ClosureWithDisjointification}
Suppose that $j:V\to M\of M[H]\of V[G]$ is the generic Boolean ultrapower arising from the induced filter $U=\union G$ by a $V$-generic
filter $G\of\B/I$, where $\B/I$ has the disjointifying property. Then $({}^\theta M[H])^{V[G]}\of M[H]$ if and only if $j\image\theta\in
M[H]$. In particular, if $\kappa=\cp(j)$, then $({}^\kappa M[H])^{V[G]}\of M[H]$.
\end{theorem}

\begin{proof} To be clear, we assume here that $M=\check V_U$ and
$M[H]=V^\B/U$, where $H=[\dot H]_U$, using the canonical $\B$-name $\dot H$ for the generic filter. Since $I$ was disjointifying, the
ultrapower is well-founded and we identify the elements of the ultrapower with their Mostowksi collapse. The forward implication is
immediate, since $j$ is definable in $V[G]$. Conversely, suppose that $j\image\theta\in M[H]$ and consider any $\<z_\alpha\st
\alpha<\theta>\in ({}^\theta M[H])^{V[G]}$. Since these are \ZFC\ models, it suffices to consider sequences of ordinals, so we may assume
$z_\alpha\in\ORD$. Let $\dot z$ be a $\B/I$ name for this sequence, and we may assume that $\[\dot z\in {}^\theta\check\ORD]=1$. From $\dot
z$, we may construct names $\dot z_\alpha$ such that $\[\dot z_\alpha=\dot z(\check\alpha)]=1$. Since each $\dot z_\alpha$ is forced to be
an ordinal, there is a $\B/I$-name $\dot f_\alpha$ for the spanning function representing $\dot z_\alpha$. That is, $1$ forces that $\dot
z_\alpha=[\dot f_\alpha]_{\union\dot G}$. By lemma \ref{Lemma.DisjointifyingNameForSpanningFunction}, we may find an actual spanning
function $g_\alpha:A_\alpha\to\ORD$ in $V$ such that $1$ forces that $\dot z_\alpha$ is represented by $[\check g_\alpha]_{\union\dot G}$.
Let $\tau_\alpha$ be the $\B$-name arising by mixing the values of $g_\alpha$ along $A_\alpha$. Thus, in the Boolean ultrapower we have
$z_\alpha=[\tau_\alpha]_U$, and the sequence of names $\<\tau_\alpha\st\alpha<\theta>$ is in $V$. As in theorem
\ref{Theorem.ClosureOfBooleanUltrapowers}, construct a $\B$-name $\sigma$ for the $\theta$-sequence, so that $\[\sigma\text{ is a
}\check\theta\text{-sequence}]^\B= 1$ and for each $\alpha <\theta$ we have $\[\sigma(\check\alpha)=\tau_\alpha]^\B=1$. Let $s=[\sigma]_U$,
which is a $j(\theta)$-sequence in $M[H]$, with $s(j(\alpha))=[\tau_\alpha]_U=z_\alpha$. Since $j\image\theta\in M[H]$, we may restrict $s$
to $j\image\theta$ and collapse to the domain to see that $\<z_\alpha\st\alpha<\theta>\in M[H]$, as desired. The final claim, about
$\kappa=\cp(j)$, is a special case since $j\image\kappa=\kappa\in M$.\end{proof}

\begin{corollary}In the context of the previous theorem, if $\B$
is $\ltdelta$-distributive, then for any $\theta<j(\delta)$ we have $({}^\theta M)^{V[G]}\of M$ if and only if $j\image\theta\in M[H]$. In
particular, if $\kappa=\cp(j)$ and $\B$ is $\ltkappa$-distributive, then $({}^\kappa M)^{V[G]}\of M$.
\end{corollary}

\begin{proof} This follows from theorem \ref{Theorem.ClosureWithDisjointification} just as corollary \ref{Corollary.ClosureOfBooleanUltrapowers} followed from theorem \ref{Theorem.ClosureOfBooleanUltrapowers}. Namely, if $j\image\theta\in M[H]$, then we know ${}^\theta M\of {}^\theta M[H]\of M[H]$ by
theorem \ref{Theorem.ClosureWithDisjointification}. But since the forcing $H\of j(\B)$ is $\theta$-distributive, this means that ${}^\theta
M\of M$, as desired.\end{proof}

We turn now to what we find to be an exciting new possibility in the general setting of complete Boolean algebras, namely, the notion of
{\it relative} precipitousness and disjointification.

\begin{definition}\rm
Suppose that $I$ and $J$ are ideals in a complete Boolean algebra $\B$, the quotients $\B/I$ and $\B/J$ are also complete, and $J\of I$. We
define that $I$ is {\df precipitous over} $J$ if whenever $G\of\B/I$ is $V$-generic, then the Boolean ultrapower by the induced ultrafilter
$U=\set{[a]_J\st [a]_I\in G}$ on $\B/J$ is well-founded. We say that $I$ is {\df disjointifying over} $J$, if for every maximal antichain
modulo $I$, there is a choice of representatives that form an antichain modulo $J$.
\end{definition}

It is clear that every ideal is precipitous over itself, and an ideal is precipitous if and only if it is precipitous over the zero ideal
$\singleton{0}$. Similarly, every ideal is disjointifying over itself, and an ideal is disjointifying if and only if it is disjointifying
over the zero ideal. If $J\of I\of\B$ are ideals, then $I/J=\set{[a]_J\st a\in I}$ is an ideal in $\B/J$. Furthermore, $(\B/J)/(I/J)$ is
isomorphic to $\B/I$ by the canonical map $[[a]_J]_{I/J}\mapsto[a]_I$.

\begin{theorem}\label{Theorem.RelativeQuotients}
Suppose that $\B$ is a complete Boolean algebra, $J\of I$ are ideals and $\B/I$ and $\B/J$ are both complete. Then:
\begin{enumerate}
 \item $I$ is precipitous over $J$ in $\B$ if and only if\/ $I/J$ is precipitous in $\B/J$.
 \item $I$ is disjointifying over $J$ if and only if\/ $I/J$ is disjointifying in $\B/J$.
\end{enumerate}
\end{theorem}

\begin{proof} (1) The point here is that $\B/I$ is canonically
isomorphic to $(\B/J)/(I/J)$ by the map $[a]_I\mapsto [[a]_J]_{I/J}$. Therefore, forcing over either of these Boolean algebras results in
$V$-generic filters $G\of\B/I$ and $\tilde G=\set{[[a]_J]_{I/J}\st [a]_I\in G}\of (\B/J)/(I/J)$. The filter induced on $\B/J$ in either
case is $U=\set{[a]_J\st [a]_I\in G}=\set{[a]_J\st [[a]_J]_{I/J}\in\tilde G}$. The ideal $I$ is precipitous over $J$ if and only if the
Boolean ultrapower by $U$ is necessarily well-founded, since $U$ is the filter on $\B/J$ induced by $G\of\B/I$. At the same time, $I/J$ is
precipitous in $\B/J$ if and only if the Boolean ultrapower by $U$ is necessarily well-founded, since $U$ is the filter on $\B/J$ induced
by $\tilde G$ on $(\B/J)/(I/J)$. Thus, $I$ is precipitous over $J$ if and only if $I/J$ is precipitous in $\B/J$, as desired.

%

(2) Suppose that $I$ is disjointifying over $J$. This means that every antichain modulo $I$ in $\B$ has a choice of representatives that
are disjoint modulo $J$. Suppose now that $\set{[a_\alpha]_J\st \alpha<\gamma}$ is an antichain modulo $I/J$ in $\B/J$. Thus, if
$\alpha\neq\beta$, we have that $[a_\alpha]_J\wedge [a_\beta]_J \in I/J$. This means $a_\alpha\wedge a_\beta\in I$, and so
$\set{a_\alpha\st \alpha<\gamma}$ is an antichain modulo $I$. Thus, by our assumption, we can find representatives $b_\alpha$ such that
$a_\alpha=_Ib_\alpha$ and $b_\alpha\wedge b_\beta\in J$ for $\alpha\neq\beta$. It follows that $[a_\alpha]_J$ is equivalent to
$[b_\alpha]_J$ modulo $I/J$ in $\B/J$, and the $[b_\alpha]_J$ form an actual antichain in $\B/J$. So $I/J$ is disjointifying in $\B/J$, as
desired.

Conversely, suppose that $I/J$ is disjointifying in $\B/J$ and $\set{a_\alpha\st\alpha<\gamma}$ is an antichain modulo $I$ in $\B$. Thus,
$\set{[a_\alpha]_J\st\alpha<\gamma}$ is an antichain modulo $I/J$ in $\B/J$, and so by our assumption we may find $[b_\alpha]_J$ equivalent
to $[a_\alpha]_J$ modulo $I/J$ such that $[b_\alpha]_J\wedge[b_\beta]_J=0$ for $\alpha\neq\beta$. This means exactly that $a_\alpha$ is
equivalent to $b_\alpha$ modulo $I$ and $b_\alpha\wedge b_\beta\in J$. Thus, we have found a choice of representatives that are disjoint
modulo $J$, as desired.\end{proof}

It is easy to see that relative disjointification is a transitive relation.

\begin{lemma} If\/ $I$ is disjointifying over $J$ and $J$ is
disjointifying over $K$, then $I$ is disjointifying over $K$.
\end{lemma}

\begin{proof} Suppose that $I$ is disjointifying over $J$ and $J$
is disjointifying over $K$. Any antichain modulo $I$ in $\B$ has a choice of representatives that forms an antichain modulo $J$, and this
has a choice of representatives forming an antichain modulo $K$, as desired.\end{proof}

It is not clear to what extent this holds for relative precipitousness. With a bit of disjointification, we have the following transitivity
converse:

\begin{theorem} Suppose that $K\of J\of I$ are ideals in a complete Boolean algebra $\B$ and all the quotients are complete. If\/ $I$ is
precipitous over $K$ and $J$ is disjointifying over $K$, then $I$ is precipitous over $J$.
\end{theorem}

\begin{proof} Suppose that $I$ is precipitous over $K$ and $J$ is
disjointifying over $K$ in $\B$. We want to show that $I$ is precipitous over $J$. Suppose that $G\of\B/I$ is $V$-generic, and consider the
induced ultrafilter $U=\set{[a]_J\st [a]_I\in G}$ on $\B/J$. Since $\B/J$ is canonically isomorphic to $(\B/K)/(J/K)$, the ultrafilter $U$
is canonically isomorphic to $\tilde U=\set{[[a]_K]_{J/K}\st [a]_J\in U}$ on $(\B/K)/(J/K)$. Since $J$ is disjointifying over $K$, it
follows by theorem \ref{Theorem.RelativeQuotients} that $J/K$ is disjointifying in $\B/K$. Consequently, by theorem
\ref{Theorem.QuotientByDisjointifyingIdealsFactor}, the Boolean ultrapower by $\tilde U$ is a factor of the Boolean ultrapower by the
induced ultrafilter $U^*=\union \tilde U=\set{[a]_K\st [[a]_K]_{J/K}\in \tilde U}=\set{[a]_K\st [a]_I\in G}$ on $\B/K$.
\begin{diagram}[height=2em,textflow]
 V & & \rTo^{j_{U^*}} & & \check V_{U^*}  \\
 &\rdTo_{j_{\tilde U}}&  & \ruTo^{k} \\
 & & \check V_{\tilde U} & \\
\end{diagram}
But this $U^*$ is precisely the ultrafilter induced by $G$ on $\B/K$. Since $I$ is precipitous over $K$, the ultrapower by $U^*$ is
consequently well-founded. Since the ultrapower by $\tilde U$ is a factor of it, the ultrapower by $\tilde U$ is also well-founded. And
since $\tilde U$ is isomorphic to $U$, we conclude that the ultrapower by $U$ is also well-founded. So $I$ is precipitous over $J$, as
desired.\end{proof}

\begin{corollary}Suppose that $J\of I$ are ideals in a complete
Boolean algebra $\B$, whose quotients are complete. If\/ $I$ is precipitous and $J$ is disjointifying, then $I$ is precipitous over $J$.
\end{corollary}

We have a number of questions to ask about the nature of relative precipitousness and disjointification, such as the following:

\begin{questions}\
\begin{enumerate}
 \item If\/ $I$ is precipitous over a precipitous ideal $J$, must $I$ be precipitous?
 \item If\/ $I$ is precipitous over $J$, must $I$ be precipitous over all the relevant intermediate ideals $K$, with $J\of K\of I$?
 \item To what extent can it happen that an ideal $I$ is precipitous over a strictly smaller ideal $J$, but is not precipitous?
 \item What are the large cardinal strengths of the hypotheses that various particular ideals on particular Boolean algebras are
     disjointifying or precipitous?
 \item What are the large cardinal strengths of the hypotheses that various particular ideals on particular Boolean algebras are
     disjointifying or precipitous over a strictly smaller ideal?
 \item For example, what is the large cardinal strength of the hypothesis that the nonstationary ideal ${\rm NS}$ on $P(\omega_1)$ is
     precipitous or disjointifying over a strictly smaller ideal $J\ofnoteq {\rm NS}$?
\end{enumerate}
\end{questions}

We shall leave these questions for a subsequent project.

\section{Boolean ultrapowers versus classical ultrapowers}\label{Section.BooleanUltrapowersVsClassicalUltrapowers}

In this section, we use the concept of relative genericity to characterize exactly when the Boolean ultrapower by an ultrafilter $U$ on a
complete Boolean algebra $\B$ is isomorphic to a classical ultrapower by an ultrafilter on a power set algebra. Suppose that $\B$ is a
complete Boolean algebra, $U$ is an ultrafilter on $\B$ and $A\of\B$ is a maximal antichain. If $C$ is a maximal antichain refining $A$,
then we define that {\df $U$ meets $C$ relative to $A$}, if for each $a\in A$ there is a smaller element $c_a\in C$ with $c_a\leq a$ such
that $\bigvee_{a\in A}c_a\in U$. We define that $U$ is {\df generic relative to $A$} if $U$ meets all such $C$ relative to $A$.  The idea
is that while generic ultrafilters select a single point from every maximal antichain, here we select single points below every element of
$A$, whose aggregate sum is in $U$. It is easy to see, for example, that if $U$ meets $C$ relative to $A$, then $U$ meets $C$ if and only
if $U$ meets $A$. Consequently, if $U$ is $V$-generic relative to $A$, then $U$ is $V$-generic if and only if $U$ meets $A$. Also, a filter
$U$ is $V$-generic if and only if it is $V$-generic relative to the trivial $\singleton{1}$, since in this case only one choice is made;
and a filter $U$ is $V$-generic if and only if it is generic relative to every maximal antichain.

\begin{theorem}\label{Theorem.BooleanUltrapowerIsClassicalIff}
Suppose that $U$ is an ultrafilter on the complete Boolean algebra $\B$. Then the following are equivalent:
\begin{enumerate}
 \item The Boolean ultrapower $j_U:V\to \check V_U$ is isomorphic to a classical power set ultrapower by a measure $U^*$ on some set
     $Z$.
 \item The ultrafilter $U$ is $V$-generic relative to some maximal antichain $A\of\B$.
\end{enumerate}
In this case, one can take $Z=A$ and the Boolean ultrapower $j_U$ is isomorphic to the power set ultrapower $j_{U_A}$.
\end{theorem}

\begin{proof} Suppose that $j_U:V\to \check V_U$ is isomorphic to
$j_{U^*}:V\to V^Z/U^*$. Thus, using the functional presentation of the Boolean ultrapower, there is an isomorphism $\pi$ making the
following diagram commute.
\begin{diagram}[width=2em]
 & & V & & \\
 & \ldTo^{j_U} & & \rdTo^{j_{U^*}} \\
V^{\downarrow\B}_U & &\rTo^\pi_{\iso} & &V^Z/U^* \\
\end{diagram}
Since $U^*$ is an ultrafilter on a power set Boolean algebra $P(Z)$, it follows that every element of $V^Z/U^*$ has the form
$[f]_{U^*}=j_{U^*}(f)(s^*)$, where $s^*=[\id]_{U^*}$ and $f:Z\to V$ with $f\in V$. Applying the isomorphism, it follows that every element
of $V^{\downarrow\B}_U$ has the form $j_U(f)(s)$, where $\pi(s)=s^*$. Since $V^{\downarrow\B}_U$ is the direct limit of $V^A/U_A$ for
maximal antichains $A$, it follows that $s=\pi_{A,\infty}(s_A)$ for some maximal antichain $A\of\B$ and some $s_A\in V^A/U_A$. If $C\of\B$
is any maximal antichain refining $A$ and $x\in V^C/U_C$, then $\pi_{C,\infty}(x)\in V^{\downarrow\B}_U$, and so
$\pi_{C,\infty}(x)=j_U(f)(s)$ for some function $f:Z\to V$ in $V$. Unwrapping this by the commutativity of the directed system, we observe
$\pi_{C,\infty}(x)=j_U(f)(s)=\pi_{A,\infty}(j_{U_A}(f)(s_A))=\pi_{C,\infty}(\pi_{A,C}(j_{U_A}(f)(s_A)))$. By peeling off the outer
$\pi_{C,\infty}$, we conclude that $x=\pi_{A,C}(j_{U_A}(f)(s_A))$. Since $x$ was arbitrary, it follows that $\pi_{A,C}$ is surjective and
hence an isomorphism. (From this, it follows that $j_U$ is isomorphic to $j_{U_A}$, since the direct limit system has only isomorphisms
below the antichain $A$.) We now argue that $U$ is generic relative to $A$. Consider the identity function $\id\restrict C : C\to C$. Since
$\pi_{A,C}$ is surjective, it follows that $[\id\restrict C]_{U_C} = \pi_{A,C}([f]_{U_A})$ for some spanning function $f:A\to V$. This
means that $\id\restrict C \equiv_{U_C} f\downarrow C$, and so $\vee Y \in U$, where $Y = \set{ c\in C \st c = (f\downarrow C)(c)}$. Note
that for each $a\in A$, there is at most one $c\in Y$ below it, since $c = (f\downarrow C)(c) = f(a)$. Let $c_a$ be this value of $c$ below
$a$, if $c\leq a$ and $c\in Y$, otherwise $c_a$ is arbitrary. Thus, $\bigvee_{a\in A} c_a\geq \vee Y\in U$. We have therefore proved that
$U$ is generic relative to $A$, as desired.

Conversely, suppose that $U$ is generic relative to $A$. We will show that $j_U : V\to \check V_U$ is isomorphic to the classical
ultrapower $j_{U_A} : V\to  V^A/U_A$ by the induced measure $U_A$, defined by $X \in U_A$ if and only if $X\subset A$ and $\vee X \in U$.
Since the Boolean ultrapower is the direct limit of the classical ultrapowers $V^C/U_C$ for maximal antichains $C$ refining $A$ as in lemma
\ref{Lemma.RefinementDiagram} and theorem \ref{Theorem.BooleanUltrapowerAsDirectLimit}, it suffices for us to show that $\pi_{A,C}$ is an
isomorphism for all such $C$. For this, it suffices to argue that $\pi_{A,C}$ is surjective. Consider any spanning function $g : C\to V$.
By genericity below $A$, there is for each $a\in A$ a selection $c_a\in C$ with $c_a\leq a$ such that $\bigvee_{a\in A}c_a\in U$. Define $f
: A\to V$ by $f(a) = g(c_a)$, and observe that $(f\downarrow C)\equiv_U g$, precisely because $\bigvee_{a\in A}c_a\in U$. Thus,
$\pi_{A,C}([f]_{U_A}) = [g]_{U_C}$ , and so $\pi_{A,C}$ is surjective. Thus, the direct limit presentation of the Boolean ultrapower is the
the identity on the antichains refining $A$, and so $j_U$ is isomorphic to $j_{U_A}$.\end{proof}

%
%
%
%
%
%
%

Canjar \cite{Canjar1987:CompleteBooleanUltraproducts} proved that not every Boolean ultrapower is isomorphic to a power set ultrapower.

Suppose that $A$ is an infinite maximal antichain in a complete Boolean algebra $\B$. Recall that the small ideal relative to $A$ in $\B$ is the ideal consisting of all elements $b\in\B$ that are below the join of a small subset of $A$, namely $b\leq \vee A_0$ for some $A_0\of A$ and $|A_0|<\kappa=|A|$. If $\kappa$ is regular, this ideal is $\kappa$-complete.

\begin{theorem} Suppose that $I$ is the small ideal relative to an infinite maximal antichain $A$ in a complete Boolean algebra $\B$. If\/ $G\of\B/I$ is $V$-generic, then the induced ultrafilter $U=\Union G$ on $\B$ does not meet $A$, but is $V$-generic relative to $A$.
\end{theorem}

\begin{proof}
It follows from the conclusion, of course, that the corresponding Boolean ultrapower $j_U$ is isomorphic to the induced power set ultrapower by $U_A$ on $P(A)$, which has critical point $\kappa$.

Since singletons are small, it follows that $A\of I$ and hence $U$ contains no elements from $A$. If $C$ is any maximal antichain refining
$A$, then let us write $f:A\searrow C$ to indicate that $f$ is a {\df regressive} function from $A$ to $C$, meaning that $f(a)\leq a$ and
$f(a)\in C$ for all $a\in A$. Such a function chooses, for each element of $A$, an element of $C$ below it. For any such regressive
function, let $b_f=\bigvee_{a\in A} f(a)$ be the join of the choices made by $f$. For $U$ to be $V$-generic relative to $A$, we must find
some $f$ such that $b_f\in U$. Let $D=\set{b_f\st f:A\searrow C}$. We claim that this is predense in $I^\plus$. To see this, suppose that
$b\in I^\plus$. Thus, $b\wedge a\neq 0$ for $\kappa$ many $a\in A$. Since each $a\in A$ has $a=\bigvee_{c\in C, c\leq a}c$, this means that
for $\kappa$ many $a$, there is $c_a\in C$ with $c_a\leq a$ and $b\wedge c_a\neq 0$. Let $f:A\searrow C$ have $f(a)=c_a$ for these $a$.
Thus, $b_f$ is above all the $c_a$, for $\kappa$ many $a$, and consequently, $b\wedge b_f$ is not below the join of any small subset of
$A$. Thus, $b$ is not incompatible with $b_f$ modulo $I$, and so $D$ is predense in $I^\plus$. It follows that $\set{[b_f]_I\st f:A\searrow
C}$ is predense in $\B/I$, and so there is some $f$ with $[b_f]_I\in G$, and consequently $b_f\in U$, as desired.\end{proof}

This argument can be somewhat generalized beyond the small ideal. Suppose that $A\of \B$ is a maximal antichain. If $J\of P(A)$ is an ideal
on the power set, then $J$ naturally induces an ideal $I=\set{b\in\B\st \exists A_0{\in}J\, b\leq\vee A_0}$ on $\B$. Let us say that such
ideals $I$ are {\df local to $A$}, being generated by an ideal on the power set of $A$. Of course, if $A\of\B$ is a maximal antichain, then
the power set $P(A)$ embeds completely into $\B$ via the map $X\mapsto \vee X$ for $X\of A$. As a result, we may regard $P(A)$ as a
complete subalgebra of $\B$. This relationship extends to the quotients:

\begin{lemma} If\/ $I$ is the local ideal on $\B$ induced by $J$ on
$P(A)$, then $P(A)/J$ complete embeds into $\B/I$ by the natural map $\pi:[X]_J\mapsto [\vee X]_I$.
\end{lemma}

\begin{proof} This is clearly a Boolean algebra homomorphism. For
completeness, suppose that $[X_\alpha]_J$ is a maximal antichain in $P(A)/J$, for $\alpha<\gamma$, and consider $[\vee X_\alpha]_I$.
Suppose that $[b]_I\neq 0$ is incompatible with all of them, so that $b\notin I$ and $b\wedge\vee X_\alpha\in I$ for all $\alpha$. Thus,
for each $\alpha$, there is $Y_\alpha\in J$ with $b\wedge \vee X_\alpha\leq \vee Y_\alpha$. Let $X=\set{a\in A\st b\wedge a\neq 0}$. Since
$b\notin I$, it follows that $X\notin J$. Since $b\wedge\vee X_\alpha\leq \vee Y_\alpha$, it follows that $X\intersect X_\alpha\of
Y_\alpha\in J$. Thus, $[X]_J$ is incompatible with every $[X_\alpha]_J$, contradicting our assumption that this was a maximal antichain in
$P(A)/J$.\end{proof}

\begin{theorem}\label{Theorem.LocalIdealsForceRelativeGenericity}
Suppose that $I$ is an ideal in $\B$ local to a maximal antichain $A$, and that $I$ contains the small ideal. If\/ $G\of\B/I$ is
$V$-generic, then the induced ultrafilter $U=\Union G$ on $\B$ is $V$-generic relative to $A$, but does not meet $A$.
\end{theorem}

\begin{proof} Suppose that $I$ is induced by the ideal $J\of
P(A)$ on the power set of $A$. Since $I$ contains the small ideal, it follows that $A\of I$ and consequently $U$ misses $A$. If $C$ is a
maximal antichain refining $A$, then consider the collection $D=\set{b_f\st f:A\searrow C}$ as before. We claim this is predense in
$I^\plus$. If $b\in I^\plus$, then since $b=\bigvee_{a\in A}b\wedge a$, it follows that $E=\set{a\in A\st b\wedge a\neq 0}\in J^\plus$. If
$b\wedge a\neq 0$, then there is some $c_a\leq a$ with $c_a\in C$ and $b\wedge c_a\neq 0$. Let $f:A\searrow C$ be such that $f(a)=c_a$ for
these choices. Now observe that $b\wedge b_f$ is at least as big as $\bigvee_{a\in E} c_a$. If this were below $\vee F$ for some $F\in J$,
then it would have to be that $E\of F$, since $c_a$ is incompatible with all elements of $A$ except $a$. This contradicts our earlier
observation that $E\in J^\plus$. Thus, $b$ is not incompatible with $b_f$ modulo $I$, and so $D$ is predense. Thus, there is some
$[b_f]_I\in G$ and consequently $b_f\in U$. So $U$ is $V$-generic relative to $A$, as desired.\end{proof}

A {\df generic Boolean ultrapower} by an ideal $I$ on a complete Boolean algebra $\B$ is the Boolean ultrapower by the filter $\union G$ on
$\B$ induced by a $V$-generic filter $G\of\B/I$ on the quotient. The classical situation occurs when $\B=P(Z)$ is a power set, and in this
case we refer to the resulting embedding as a {\df generic power set ultrapower}.

\begin{corollary}\label{Corollary.GenericEmbeddingBooleanUltrapowerCorrespondence}
Every generic Boolean ultrapower by a local ideal containing the small ideal is isomorphic to a generic power set ultrapower. Specifically,
if\/ $\B$ is a complete Boolean algebra with an antichain $A$ and $I$ is the local ideal generated by an ideal $J$ on the power set $P(A)$,
then the generic ultrapowers arising from $\B/I$ are isomorphic to those arising from $P(A)/J$. In particular, $J$ is precipitous on $P(A)$
if and only if\/ $I$ is precipitous on $\B$.
\end{corollary}

\begin{proof} In the situation of the previous theorem, the
Boolean ultrapower by $U$ will be isomorphic to the ultrapower by the ultrafilter $U_A=\set{X\of A\st \vee X\in U}$ on the power set
$P(A)$. This is because theorem \ref{Theorem.LocalIdealsForceRelativeGenericity} shows that the ultrafilter $U$ is $V$-generic relative to
$A$, fulfilling the condition of theorem \ref{Theorem.BooleanUltrapowerIsClassicalIff} that the Boolean ultrapower is classical.\end{proof}


\begin{corollary}If there is a precipitous ideal on a power set
$P(\kappa)$, then every Boolean algebra $\B$ with an antichain of size $\kappa$ has a precipitous ideal.
\end{corollary}

\begin{proof} If there is a precipitous ideal $J$ on $P(\kappa)$,
then by restricting to the smallest $J$-positive set, we may assume that $J$ contains the small ideal. If $A\of\B$ has size $\kappa$, we
may view $J$ as being a precipitous ideal on $P(A)$. Corollary \ref{Corollary.GenericEmbeddingBooleanUltrapowerCorrespondence} shows that
the induced ideal $I$ on $\B$ gives rise to exactly the same generic embeddings, and so $I$ also is precipitous.\end{proof}

%
%
%
%
%
%
%
%

\section{Boolean ultrapowers as large cardinal embeddings}\label{Section.BooleanUltrapowersAsLargeCardinalEmbeddings}

We will now investigate several instances in which various large cardinal embeddings can be realized as well-founded Boolean ultrapowers. By
theorem \ref{Theorem.BooleanUltrapowersGeneralizeUsualUltrapowers}, we know that every large cardinal embedding that is the ultrapower by a
measure on a set can also be viewed as a Boolean ultrapower by an ultrafilter on the corresponding power set Boolean algebra. Thus, the
large cardinal embeddings witnessing that a cardinal $\kappa$ is measurable, strongly compact, supercompact, huge and so on, are all
instances of well-founded Boolean ultrapowers.

Let us mention briefly how this meshes with theorem \ref{Theorem.BooleanUltrapowerWellFounded}. If $j : V\to M$ is the ultrapower by a
measure $\mu$ on the set $Z$, then the corresponding Boolean algebra is $\B = P(Z)$, and $j(\B)$ is the power set of $j(Z)$ in $M$. We know
that $X\in \mu \Iff [\id]_\mu\in j(X)$, and so $j\image \mu\subset F$, where $F$ is the principal ultrafilter generated by $[\id]_\mu\in
j(\B)$. Since $j(\B)$ is an atomic Boolean algebra, the principal ultrafilter $F$ generated by $[\id]_\mu$ is indeed $M$-generic, and
$j\image\mu\subset F$, fulfilling theorem \ref{Theorem.BooleanUltrapowerWellFounded}. Note that if $\mu$ is not principal, then $\mu$
itself is definitely not $V$-generic for $\B$, since it misses the maximal antichain of atoms $\set{ \singleton{z} \st z\in Z }$.

Let us now investigate a few circumstances under which other kinds of Boolean algebras can give rise to large cardinal embeddings. Theorem
\ref{Theorem.BooleanUltrapowerWellFounded} will allow us to construct examples of well-founded ultrafilters that are not generic.

\begin{theorem}\label{Theorem.MeasurableCardinalsHaveWellfoundedUlts}
If $\kappa$ is a measurable cardinal and $2^\kappa  =\kappa^\plus$  in $V$, then there are many well-founded ultrafilters in $V$ on the
Boolean algebra of the forcing $\Add(\kappa,1)$.
\end{theorem}

\begin{proof} Suppose that $\kappa$ is a measurable cardinal,
witnessed by the normal ultrapower embedding $j : V\to  M$, and suppose $2^\kappa =\kappa^\plus$. Let $\B$ be the regular open algebra of
$\Add(\kappa, 1)$, and consider $j(\B)$, which corresponds to the forcing $\Add(j(\kappa), 1)^M$. This forcing is $\leqkappa$-closed in
$M$. The number of subsets of $\Add(\kappa, 1)$ is $\kappa^\plus$, and so the number of subsets of $\Add(j(\kappa), 1)^M$ in $M$ is
$j(\kappa^\plus)$, which has size $\kappa^\plus$  in $V$ . We may therefore assemble the dense subsets for this forcing from $M$ into a
$\kappa^\plus$-sequence in $V$, and proceed to diagonalize against this list, using the closure of the forcing and the fact that
$M^\kappa\of M$ at limit stages, to produce in $V$ an $M$-generic filter $G\subset j(\B)$. Let $U = j^\inverse G$. Thus, $U$ is an
ultrafilter on $\B$ and $j\image U\subset G$, where $G$ is $M$-generic for $j(\B)$. By theorem \ref{Theorem.BooleanUltrapowerWellFounded},
it follows that $U$ is a well-founded ultrafilter on $\B$ in $V$. Further, the corresponding Boolean ultrapower embedding $j_U$ is a factor
of $j$.\end{proof}

\begin{corollary}
If $\kappa$ is a measurable cardinal with $2^\kappa=\kappa^\plus$  and $\mu$ is a normal measure on $\kappa$, then there is a well-founded
ultrafilter $U$ on the Boolean algebra corresponding to $\Add(\kappa, 1)$ such that the Boolean ultrapower by $U$ is the same as the
ultrapower by $\mu$.
\end{corollary}

\begin{proof} We saw above that if $j_\mu:V\to M$ is the
ultrapower by $\mu$, then we may find an ultrafilter $U$ on the Boolean algebra $\B$ corresponding to $\Add(\kappa, 1)$ such that $j\image
U\subset G$ for some $M$-generic filter $G\subset j(\B)$. By corollary \ref{Corollary.BooleanUltrapowerFactor}, it follows that $j_U$ is a
factor of $j_\mu$. Since $\mu$ was a normal measure, it is minimal in the Rudin-Kiesler order on embeddings, so it has no proper factors.
Consequently, $j_U = j_\mu$, as desired.\end{proof}

%

The proof of theorem \ref{Theorem.MeasurableCardinalsHaveWellfoundedUlts} used the closure of the forcing notion, in a manner similar to
many lifting arguments in large cardinal set theory. But we now show that if $\kappa$ is strongly compact, then we can produce well-founded
Boolean ultrafilters on any $\ltkappa$-distributive complete Boolean algebra.

\begin{theorem} For any infinite regular cardinal $\kappa$, the following are equivalent:
\begin{enumerate}
 \item Every $\kappa$-complete filter $F$ on a $\ltkappa$-distributive complete Boolean algebra $\B$ is contained in a $\kappa$-complete ultrafilter $U\of\B$.
 \item $\kappa$ is strongly compact.
%
\end{enumerate}
\end{theorem}

\begin{proof} Assume (1). Fix any ordinal $\theta$ and let $\B$ be the power set of $P_\kappa\theta$, which is trivial as a notion of forcing and is consequently $\ltkappa$-distributive. Let $F$ be the fineness filter on $P_\kappa\theta$, generated by the sets $A_\alpha=\set{\sigma\in P_\kappa\theta\st \alpha\in\sigma}$. This filter is easily seen to be $\kappa$-complete in $\B$. By (2), there is a $\kappa$-complete ultrafilter $U\of\B$ with $F\of U$. Thus, there is a $\kappa$-complete fineness measure on $P_\kappa\theta$, and so $\kappa$ is strongly compact.

Conversely, suppose $\kappa$ is strongly compact and $F$ is a $\kappa$-compete filter on $\B$, a $\ltkappa$-distributive complete Boolean algebra. Let $\mathcal{D}$ be the collection of open dense subsets of $\B$ and let $\theta=|\mathcal{D}|$. Let $j:V\to M$ be a $\theta$-strong compactness embedding for $\kappa$, so $\cp(j)=\kappa$ and $j(\kappa)>\theta$, and every size $\theta$ subset of $M$ is
covered by an element of $M$ of size less than $j(\kappa)$ in $M$. Thus, there is a set $\mathcal{E}\in M$ with
$j\image\mathcal{D}\of\mathcal{E}$ and $|\mathcal{E}|^M<j(\kappa)$. Thus, $\mathcal{E}$ consists of fewer than $j(\kappa)$ many open dense
subsets of $j(\B)$ in $M$. Since $j(\B)$ is ${\smalllt}j(\kappa)$-distributive in $M$, the intersection of these open dense sets is still open and dense. By the strong compactness cover property again, there is a set $s$ with $j\image F\of s\in M$ and $|s|^M<j(\kappa)$. We may assume $s\of j(F)$, and so by the completeness of $j(F)$ in $M$, it follows that $\wedge s\in j(F)$. Putting the two facts together, we may find $p^*\in j(\B)$ such that $p^*\in j(D)$ for all $D\in\mathcal{D}$ and $p^*\leq j(a)$ for all $a\in F$. If $H\of j(\B)$ is any filter
containing $p^*$, it follows that $H$ is $\ran(j)$-generic. Let $U=j^\inverse H$ be the inverse image, which is an ultrafilter on $\B$,
because $H$ selects an element of $j(\singleton{b,\neg b})$. Since $p^*\leq j(a)$ for all $a\in F$, it follows that $F\of U$. And since $j\image U\of H$, it follows by theorem \ref{Theorem.RudinKeislerIff} that the Boolean ultrapower $j_U$ is a factor of $j$. In particular, the Boolean ultrapower by $U$ is well-founded, as desired.\end{proof}

Let us now sharpen this result in a way that will assist in applications, such as those in \cite{ApterGitmanHamkins2012:InnerModelsWithLargeCardinals}. We define that a forcing notion $\P$ is {\df $\ltkappa$-friendly} if for every $\gamma<\kappa$, there is a nonzero condition $p\in\P$ below which the restricted forcing $\P\restrict p$ adds no subsets to $\gamma$.

\begin{theorem}\label{Theorem.FriendlyStrC}
If $\kappa$ is a strongly compact cardinal and $\P$ is a $\ltkappa$-friendly notion of forcing, then there is a well-founded ultrafilter $U$ on the Boolean algebra completion $\B(\P)$. In this case, there is an inner model $W$ satisfying every sentence forced by $\P$ over $V$. Indeed, there is an elementary embedding of the universe $j:V\to \Vbar\of \Vbar[G]=W$ into a transitive class $\Vbar$, such that in $V$ there is a $\Vbar$-generic ultrafilter $G\of j(\P)$.
\end{theorem}

\begin{proof}
Let $\B$ be the Boolean algebra completion of $\P$, and observe that $\B$ also is $\ltkappa$-friendly. We will find an ultrafilter $U\of\B$ for which the Boolean ultrapower $V^\B/U$ is well-founded. Consider any $\theta\geq |\B|$ and let $j:V\to M$ be a $\theta$-strong compactness embedding, so that $j\image\B\of s\in M$ for some $s\in M$ with $|s|^M<j(\kappa)$. Since $j(\B)$ is ${\smalllt}j(\kappa)$-friendly in $M$, there is a nonzero condition $p\in j(\B)$ such that forcing with $j(\B)\restrict p$ over $M$ adds no new subsets to $\lambda=|s|^M$. Thus, $j(\B)\restrict p$ is $(\lambda,2)$-distributive in $M$. Applying this, it follows in $M$ that $$p=p\wedge 1=p\wedge\bigwedge_{b\in s}(b\vee\neg b)=\bigwedge_{b\in s}(p\wedge b)\vee(p\wedge\neg b)=\bigvee_{f\in2^s}\bigwedge_{b\in s}(p\wedge(\neg)^{f(b)} b),$$ where $(\neg)^0b=b$ and $(\neg)^1b=\neg b$, and where we use distributivity to deduce the final equality. Since $p$ is not $0$, it follows that there must be some $f$ with $q=\bigwedge_{b\in s}p\wedge(\neg)^{f(b)}b\neq 0$. Note that $f(b)$ and $f(\neg b)$ must have opposite values. Using $q$ as a seed, define the ultrafilter $U=\set{a\in \B\st q\leq j(a)}$, which is the same as $\set{a\in\B\st f(j(a))=0}$. This is easily seen to be a $\kappa$-complete filter using the fact that $\cp(j)=\kappa$ (just as in the powerset ultrafilter cases known classically). It is an ultrafilter precisely because $s$ covers $j\image\B$, so either $f(j(a))=0$ or $f(\neg j(a))=0$, and so either $a\in U$ or $\neg a\in U$, as desired. So by theorem \ref{Theorem.WellFoundedEquivalents} the Boolean ultrapower $V^\B/U$ is well-founded, and the Mostowski collapse of this model is an inner model satisfying any sentence forced by $\B$ over $V$, as desired. The Boolean ultrapower embedding $j:V\to \Vbar\of \Vbar[G]$ has the properties stated in the theorem.
\end{proof}

Theorem \ref{Theorem.FriendlyStrC} is technically a theorem scheme, asserting separately for each sentence forced by $\P$ that $V^\B/U$ satisfies it. It follows that when there are sufficient large cardinals, then one may construct transitive inner models of \ZFC\ satisfying a variety of statements ordinarily obtained by forcing. For example, \cite{ApterGitmanHamkins2012:InnerModelsWithLargeCardinals} uses this method to show that if there is a supercompact cardinal, then there is a transitive inner model of \ZFC\ with a Laver indestructible supercompact cardinal, and furthermore one can ensure for this cardinal $\kappa$ either that $2^\kappa=\kappa^\plus$ or that $2^\kappa=\kappa^{++}$, or any of a number of other properties; the method is really quite flexible.

Certain instances of the phenomenon of theorem \ref{Theorem.FriendlyStrC} are familiar in large cardinal set theory. For example, consider Prikry forcing with respect to a normal measure $\mu$ on a measurable cardinal $\kappa$, which is $\ltkappa$-friendly because it adds no bounded subsets to $\kappa$. If $V\to M_1\to M_2\to\cdots$ is the usual iteration of $\mu$, with a direct limit to $j_\omega:V\to M_\omega$, then the critical sequence $\kappa_0,\kappa_1,\kappa_2,\ldots$ is well known to be $M_\omega$-generic for the corresponding Prikry forcing at $j_\omega(\kappa)$ using $j_\omega(\mu)$. This is precisely the situation occurring in theorem \ref{Theorem.FriendlyStrC}, where we have an embedding $j:V\to\Vbar$ and a $\Vbar$-generic filter $G\of j(\P)$ all inside $V$. Thus, theorem \ref{Theorem.FriendlyStrC} generalizes this classical aspect about Prikry forcing to all friendly forcing under the stronger assumption of strong compactness.

\begin{question}
 Which kinds of large cardinal embeddings are fruitfully realized as Boolean ultrapower embbeddings? 
\end{question}

While preparing earlier drafts of this article, we had asked whether the elementary embedding $j_\omega:V\to M_\omega$ arising from the $\omega$-iteration of a normal measure $\mu$ on a measurable cardinal $\kappa$ is a Boolean ultrapower. This question is now answered by
Fuchs and Hamkins \cite{FuchsHamkins:TheBukovskyDehornoyPhenomenonForBooleanUltrapowers}, who proved that indeed it is the Boolean ultrapower for Prikry forcing, and similar facts are true much more generally for other iterations, using generalized Prikry forcing. This case offers many very suggestive features, including the \Bukovsky-Dehornoy phenomenon, for which the forcing extension $M_\omega[s]$, where $s$ is the critical sequence that is $M_\omega$-generic for Prikry forcing, is the same as the intersection $M_\omega[s]=\Intersect_n M_n$ of the finite iterates of $\mu$. Another way to express this is that the full Boolean ultrapower $V^\B/U$ is the intersection of the induced power set ultrapowers $V^A/U_A$ for maximal antichains $A$.
\begin{question}
 To what extent does the \Bukovsky-Dehornoy phenomenon hold for Boolean ultrapowers generally?
\end{question}

To our way of thinking, the intriguing situation here is that in the case of an iterated normal measure, the critical sequence was independently known to be a highly canonical useful object, and it also happens to be identical to the canonical generic object of the Boolean ultrapower $\check V_U=M_\omega$. If we should be able to realize other large cardinal embeddings as Boolean ultrapowers, such as many of the extender embeddings that are used, then the hope is that their associated canonical generic objects, perhaps currently hidden, would prove to be similarly fruitful.

\bibliographystyle{alpha}
\bibliography{MathBiblio,HamkinsBiblio}

\begin{thebibliography}{AGH12}

\bibitem[AGH12]{ApterGitmanHamkins2012:InnerModelsWithLargeCardinals}
Arthur Apter, Victoria Gitman, and Joel~David Hamkins.
\newblock Inner models with large cardinal features usually obtained by
  forcing.
\newblock {\em Archive for Mathematical Logic}, 51:257--283, 2012.
\newblock 10.1007/s00153-011-0264-5.

\bibitem[Bel85]{Bell1985:BooleanValuedModelsAndIndependenceProofs}
J.~L. Bell.
\newblock {\em Boolean-valued models and independence proofs in set theory},
  volume~12 of {\em Oxford Logic Guides}.
\newblock The Clarendon Press Oxford University Press, New York, second
  edition, 1985.
\newblock With a foreword by Dana Scott, first edition 1977.

\bibitem[Can87]{Canjar1987:CompleteBooleanUltraproducts}
R.~Michael Canjar.
\newblock Complete {B}oolean ultraproducts.
\newblock {\em J. Symbolic Logic}, 52(2):530--542, 1987.

\bibitem[FH]{FuchsHamkins:TheBukovskyDehornoyPhenomenonForBooleanUltrapowers}
Gunter Fuchs and Joel~David Hamkins.
\newblock The {Bukovsk\'y-Dehornoy} phenomenon for boolean ultrapowers.
\newblock manuscript in preparation.

\bibitem[Man71]{Mansfield1970:TheoryOfBooleanUltrapowers}
Richard Mansfield.
\newblock The theory of {B}oolean ultrapowers.
\newblock {\em Ann. Math. Logic}, 2(3):297--323, 1970/71.

\bibitem[OR98]{OuwehandRose1998:FiltralPowersOfStructures}
P.~Ouwehand and H.~Rose.
\newblock Filtral powers of structures.
\newblock {\em J. Symbolic Logic}, 63(4):1239--1254, 1998.

\bibitem[Vop65]{Vopenka1965:OnNablaModelOfSetTheory}
Petr Vop{\v{e}}nka.
\newblock On {$\nabla$}-model of set theory.
\newblock {\em Bull. Acad. Polon. Sci. S\'er. Sci. Math. Astronom. Phys.},
  13:267--272, 1965.

\end{thebibliography}

\end{document}